\documentclass[reqno]{amsart}
\usepackage{amsfonts}
\usepackage{amssymb}
\usepackage[centertags]{amsmath}
\usepackage{amsthm}
\usepackage{graphicx}
\usepackage{color}
\usepackage{pinlabel}
\usepackage{esint}

\theoremstyle{plain}
\newtheorem{theorem}{Theorem}[section]
\newtheorem{corollary}[theorem]{Corollary}
\newtheorem{lemma}[theorem]{Lemma}

\theoremstyle{definition}
\newtheorem{definition}[theorem]{Definition}%[section]

\theoremstyle{remark}
\newtheorem{remark}[theorem]{Remark}
\numberwithin{equation}{section}

\newcommand{\ep}{\varepsilon}
\renewcommand{\geq}{\geqslant}
\renewcommand{\leq}{\leqslant}
\DeclareMathOperator{\curl}{\mathrm{curl}} 
\renewcommand{\div}{\operatorname{div}}

\DeclareMathOperator{\dist}{\mathrm{dist}}
\DeclareMathOperator{\ch}{\mathrm{co}}
\newcommand{\bzero}{\mathbf 0}
\newcommand{\ba}{\mathbf a}
\newcommand{\bb}{\mathbf b}

\newcommand{\be}{\mathbf e}
\newcommand{\bef}{\mathbf f}
\newcommand{\bF}{\mathbf F}
\newcommand{\bg}{\mathbf g}
\newcommand{\bh}{\mathbf h}
\newcommand{\bI}{\mathbf I}
\newcommand{\bj}{\mathbf j}
\newcommand{\bJ}{\mathbf J}
\newcommand{\bk}{\mathbf k}
\newcommand{\bL}{\mathbf L}

\newcommand{\bn}{\mathbf n}
\newcommand{\bp}{\mathbf p}
\newcommand{\bq}{\mathbf q}
\newcommand{\bv}{\mathbf v}
\newcommand{\bV}{\mathbf V}
\newcommand{\bw}{\mathbf w}
\newcommand{\bW}{\mathbf W}
\newcommand{\bx}{\mathbf x}
\newcommand{\bX}{\mathbf X}
\newcommand{\by}{\mathbf y}
\newcommand{\bY}{\mathbf Y}
\newcommand{\bz}{\mathbf z}
\newcommand{\bZ}{\mathbf Z}
\newcommand{\balpha}{\boldsymbol\alpha}
\newcommand{\bbeta}{\boldsymbol\beta}

\newcommand{\bxi}{\boldsymbol\xi}

\newcommand{\bLam}{{\boldsymbol \Lambda}}

\newcommand{\bbz}{\overline{\bz}}
\newcommand{\bzz}{\overline{z}}

\newcommand{\cA}{\mathcal A}
\newcommand{\cB}{\mathcal B}
\newcommand{\cD}{\mathcal D}
\newcommand{\cE}{\mathcal E}

\newcommand{\cG}{\mathcal G}
\newcommand{\cH}{\mathcal H}
\newcommand{\cI}{\mathcal I}

\newcommand{\cM}{\mathcal M}
\newcommand{\cP}{\mathcal P}
\newcommand{\cS}{\mathcal S}

\newcommand{\cZ}{\mathcal Z}

\newcommand{\de}{\mathrm{d}}

\newcommand{\R}[1]{\mathbb{R}^{#1}}
\newcommand{\RR}{\mathbb{R}}
\newcommand{\N}{\mathbb{N}}

\newcommand{\hbf}{\widehat {\bf f}}

\renewcommand{\vec}[2]{\left(\begin{array}{c}
		  #1 \\
		  #2
		  \end{array}\right)}

\newcommand{\vecc}[2]{\left(\begin{array}{c}
		  #1 \\
		  \vdots \\
		  #2
		  \end{array}\right)}
	  
\newcommand{\vecN}[3]{\left(\begin{array}{c}
		  #1 \\
		  #2 \\
		  \vdots \\
		  #3
		  \end{array}\right)}

\makeindex
\begin{document}
\title[Screw Dislocation Dynamics]{
Dynamics for Systems of  
Screw Dislocations}
\author{Timothy Blass, Irene Fonseca, Giovanni Leoni, Marco Morandotti}

\begin{abstract}
The goal of this paper is the analytical validation of a model of Cermelli and Gurtin \cite{Gurtin} for an evolution law for systems of screw dislocations 
under the assumption of antiplane shear.  
The motion of the dislocations is restricted to a discrete
set of glide directions, which are properties of the material. The evolution
law is given by a ``maximal dissipation criterion'', leading
to a system of differential inclusions.
Short time existence, uniqueness, cross-slip, and fine cross-slip of solutions are proved.
\end{abstract}

\maketitle
%\tableofcontents

%\input{intro}
\section{Introduction}
Dislocations are one-dimensional defects in crystalline materials \cite{N67}. Their modeling 
is of great interest in materials science since 
important material properties, such as rigidity and conductivity, can be strongly
affected by
the presence of dislocations. For example, large collections of dislocations can 
result in plastic deformations in solids under applied loads.

In this paper we study the motion of screw dislocations in
cylindrical crystalline materials using a continuum model introduced by Cermelli and Gurtin
\cite{Gurtin}.  
One of our main contributions is the analytical validation to 
this model by proving local existence and uniqueness of solutions to the equations of motions for a system of dislocations.
In particular, we prove rigorously the phenomena of cross-slip and fine cross-slip. We refer to  the work
of Armano and Cermelli \cite{AC04,CA07} for the case of a single dislocation.

Following the work of Cermelli and Gurtin \cite{Gurtin}, we consider an elastic body $B:=\Omega \times \RR$,  where
$\Omega \subset \RR^2$ is a bounded simply connected open set with $C^{2,\alpha}$ boundary. 
The body $B$ undergoes \emph{antiplane shear} 
deformations $\Phi:B\to B$ of the form
\begin{equation*}
  \Phi(x_1,x_2,x_3) := (x_1,x_2,x_3 + u(x_1,x_2)),
\end{equation*}
with $u:\Omega\to\RR$. The deformation gradient ${\bf F}$ is given by
\begin{equation}\label{deform_grad}
  {\bf F} := \nabla \Phi =\left(
  \begin{array}{ccc}
    1 & 0 & 0\\ 0 & 1 & 0\\ \frac{\partial u}{\partial x_1}& \frac{\partial u}{\partial x_2} & 1
  \end{array}
\right) = {\bf I} + {\bf e}_3{\otimes }\left(
    \begin{array}{c}
      \nabla u \\ 0
    \end{array}
\right).
\end{equation}
The assumption of antiplane shear allows us to reduce the three-dimensional problem
to a two-dimensional problem. We will consider \emph{strain fields}  $\bh$
that are defined on the cross-section
$\Omega$, taking values in $\RR^2$. 
%In this setting, w
In the absence of dislocations, the strain $\bh$ is the gradient of a function,
$\bh = \nabla u$. If  dislocations are present, then the strain field is singular at the sites of the dislocations, and in the case of screw dislocations this will be a line singularity. In the antiplane shear setting, this line is parallel
to the $x_3$ axis and the screw dislocation 
is represented as a point singularity on the cross-section $\Omega$.

A screw dislocation is characterized by a position $\bz \in \Omega$ and a vector
$\bb\in \RR^3$, called the \emph{Burgers} vector. The position $\bz \in \Omega$
is a point where the strain field fails to be the gradient of a smooth function and
the Burgers vector measures the severity of this failure. To be precise,
a strain field associated with a system of $N$ screw dislocations at 
positions 

\begin{equation*}%\label{Z}
\cZ:=\{\bz_1,\ldots,\bz_N\}\subset\Omega
\end{equation*} 
with corresponding Burgers vectors
\begin{equation*}\label{B}
\cB:=\{b_1\be_3,\ldots,b_N\be_3\}
\end{equation*}
 satisfies the relation
\begin{equation}
  \label{curl}
  \curl\bh=\sum_{i=1}^N b_i\delta_{\bz_i}\quad {\rm in}\,\, \Omega
\end{equation}
in the sense of distributions. Here $\curl\bh$ is
the scalar curl
$\frac{\partial h_2}{\partial x_1}-\frac{\partial h_1}{\partial x_2}$, $\delta_\bx$ is the Dirac mass at the point $\bx$, and the scalar $b_i$ is called the \emph{Burgers modulus}
for the dislocation at $\bz_i$, and in view of \eqref{curl} it is 
%The Burgers moduli for a system of dislocations
%at $\{\bz_1,\ldots,\bz_N\}$  are 
given by
  \begin{equation*}
    b_i=\int_{\ell_i}   {\bf h} \cdot {\bf t}\,\de s,
  \end{equation*}
where $\ell_i$ is any counterclockwise loop surrounding the
dislocation point ${\bf z}_i$ and no other dislocation points, 
${\bf t}$ is the tangent to $\ell_i$,
and $\de s$ is the line element.

When dislocations are present, \eqref{deform_grad} is replaced with
\begin{equation*}\label{sing_grad}
\bF=\bI+\be_3{\otimes}\vec{\bh}{0}.
\end{equation*}

To derive a motion law
for the system
of dislocations we need to introduce the free energy associated to the system. We work in the context of linear elasticity. 
%begins with
% two-dimensional linear elasticity.
%We will consider quadratic energy densities and linear stress-strain relations.
The energy density $W$ %and the shear stress $C$
%are 
is given by
\begin{equation*}\label{W_def}
  W({\bf h}) := \frac12 {\bf h} \cdot \bL\bh%, \qquad C({\bf h}) = {\bf L h},
\end{equation*}
 where the  
elasticity tensor ${\bf L}$ is
a symmetric, positive-definite matrix, which, in suitable coordinates,
can be written in terms of the 
Lam\'e moduli $\lambda,\mu$ of the material as
\begin{equation*}%\label{L}
  {\bf L} := \left(
    \begin{array}{cc}
      \mu & 0 \\ 0 & \mu \lambda^2
    \end{array}\right).
\end{equation*}
We require $\mu >0$, and 
the energy is isotropic if and only if $\lambda^2 = 1$. %, otherwise it is anisotropic.
% A sufficient condition for the energy density $W$ to be 
% convex is that the tensor $\bL$ be positive definite (i.e. $\mu >0$).
The energy of a strain field $\bh$ is given by
\begin{equation}\label{J_def}
J(\bh):=\int_\Omega W(\bh(\bx))\, \de\bx,
\end{equation}
and the equilibrium equation is
\begin{equation}
  \label{equilibrium}
  	\div\bL\bh=0 \quad {\rm in}\,\, \Omega.
\end{equation}
Equations \eqref{curl} and \eqref{equilibrium}
provide  a characterization of strain fields describing screw dislocation systems
in linearly elastic materials.
%Let $\cZ:=\{\bz_1,\ldots,\bz_N\}\subset\Omega$ 
%denote a set of $N\in\mathbb{N}$ dislocations contained in $\Omega$, and let
%$\cB:=\{\bb_1,\ldots,\bb_N\}$ stand for the associated Burgers vectors.
%we make the following definition. 
To be precise, we say that  a strain field $\bh \in L^2(\Omega; \R2)$ corresponds to
a {\emph {system of dislocations}} at the positions $\cZ$ with Burgers vectors $\cB$
 if $\bh$ satisfies 
\begin{equation}\label{curldiv}
        \left\{
          \begin{array}{l}
\curl\bh=\sum_{i=1}^N b_i\delta_{\bz_i}\\
	\div\bL\bh=0 
\end{array}\right.\quad {\rm in}\,\, \Omega,
\end{equation}
in the  sense of distributions.
% \note{I think it will be easier for notation 
% to use a scalar-valued curl instead of vector-valued curl, and to make
% the set $\cB$ the set of Burgers moduli instead of Burgers vectors.}

In analogy to the theory of Ginzburg-{L}andau vortices \cite{BBH}, 
 no variational principle can be associated with \eqref{curldiv} 
because the elastic energy of a system of screw dislocations is not finite (see, e.g., \cite{Leoni, Gurtin, N67}),
therefore the study of \eqref{curldiv} cannot be undertaken in terms of energy minimization.
Indeed, the simultaneous requirements of finite energy and \eqref{curl} are incompatible,
%(Any strain field satisfying $\curl \bh = \delta_{\bz}$
%will not have a finite elastic energy, $J(\bh) = +\infty$). Indeed, if one
%looks at the energy of such a field in $\Omega$ with a small ball around $\bz$ removed,
%one finds
since if $\curl \bh=\delta_{\bz_0}$, $\bz_0\in\Omega$, 
and if $B_\ep(\bz_0)\subset\subset\Omega$, then
\begin{equation*}
\int_{\Omega \setminus B_\ep(\bz_0)} |\bh|^2\de \bx = O(|\log \ep|).
\end{equation*}
In the engineering literature (see, e.g.,  \cite{Gurtin, N67}), this problem is usually overcome by regularizing the energy, namely, by replacing the energy $J$ in \eqref{J_def} 
with a new energy $J_\ep$ obtained by removing small cores of size $\ep>0$ centered at the dislocations points $\bz_i$. This allows to obtain finite-energy strains $\bh_\ep$ as minimizers of $J_\ep$. 
It was shown in \cite{BM14} that 
\begin{equation}\label{expansion}
J_\ep (\bh_\ep) = C |\log \ep| + U(\bz_1,\ldots,\bz_N) + O(\ep),
\end{equation}
where $U$ is the \emph{renormalized energy} associated with the limiting strain  $\bh_0 = \lim_{\ep\to 0} \bh_\ep$, 
satisfying \eqref{curldiv}. 

This type of asymptotic expansion was first proved by Bethuel, Brezis,
and  H{\'e}lein in \cite{BBH93} for  Ginzburg-{L}andau vortices. The
case of edge dislocations was studied in \cite{Leoni}. Asymptotic
expansions of the type  \eqref{expansion} can also be derived using
$\Gamma$-convergence techniques (see, e.g., \cite{AP14,SS03} and the
references therein for Ginzburg-{L}andau vortices,
\cite{DLGP12,GPPS13,GLP10}  for edge dislocations, and
\cite{ACP11,CG09,CGM11,FG07,GM05,GM06,SZ12} for other dislocations
models). Finally, it is important to mention that we ignore here the
\emph{core energy}, that is, 
the energy contribution proportional to $|\log \ep|$ in \eqref{expansion},
which comes from the small cores that were removed to obtain $J_\varepsilon$.
We refer  to \cite{N67,TOP96,VKHLO12} for a more detailed discussion of the core energy. 

The force on a dislocation at $\bz_i$ due to the elastic strain is called the
\emph{Peach-K\"ohler force}, and is denoted by $\bj_i$ (see \cite{Gurtin}, \cite{PK}). 
The renormalized energy $U$ is a function 
only of the positions $\{\bz_1,\ldots,\bz_N\}$ (and of the Burgers moduli),
and it is shown  in \cite{BM14} that 
its gradient with respect to $\bz_i$ gives the negative of the
Peach-K\"ohler  force on $\bz_i$. Specifically,
\begin{equation}\label{PK=DU}
\bj_i = -\nabla_{\bz_i}U=
\int_{\ell_i} \left\{ W(\bh_0){\bf I} - \bh_0 {\otimes}(\bL\bh_0) \right\}\bn \, \de s,
\end{equation}
where $\ell_i$ is a suitably chosen
loop around $\bz_i$ and $\bn$ is the outer unit 
normal to the set bounded by $\ell_i$ and containing $\bz_i$. 
The quantity $W(\bh_0){\bf I} - \bh_0 {\otimes}(\bL\bh_0)$
is the \emph{Eshelby stress tensor}, see \cite{Eshelby,Gurtin95}.

To study the motion of dislocations it is more convenient to rewrite $\bj_i$ in the form 
\begin{equation}\label{PKj}
\bj_i(\bz_i)=b_i \bJ \bL\Big[\sum_{j\neq i} \bk_j(\bz_i;\bz_j)+
\nabla u_0(\bz_i;\bz_1,\ldots,\bz_N)\Big]
\end{equation}
(see \cite{BM14} for a proof of this derivation).
Here 
$\bk_j(\cdot;\bz_j)$ is the \emph{fundamental singular strain} generated by 
the dislocation $\bz_j$, where
\begin{equation}
\label{k_def}
\bk_j(\bx;\by) := \frac{b_j}{2\pi} \frac{\lambda \bJ^T(\bx-\by)}{|\bLam(\bx-\by)|^2},
\qquad (\bx,\by)\in\R2{\times}\R2,\; \bx\neq\by,
\end{equation}
with 
\begin{align}\label{matrices}
\bJ := \left(\begin{array}{rr}
0 & 1\\ -1 & 0
\end{array} \right), \quad
\bLam := \left(\begin{array}{rr}
\lambda & 0\\ 0 & 1
\end{array} \right). 
\end{align}
%in order to simplify some of the formulas. 
Straightforward calculations show that, for $(\bx,\by)\in\RR^2{\times}\R2$, $\bx\neq\by$, we have
\begin{subequations}
	\begin{align}
	{\rm div}_{\by}( {\bf L}\nabla_{\by}\bk_j(\bx;\by)) 
	&= \bzero, \label{laplaciank}\\
	{\rm div}_{\bx}\left(\bL \bk_j (\bx;\by)\right)&=0, \label{divk}\\
\intertext{and, for $(\bx,\by)\in\R2{\times}\R2$,}
	\curl_\bx\bk_j (\bx;\by)&= b_j\delta_{\by}(\bx). \label{curlk}
	\end{align}
\end{subequations}
Also, for fixed $\bz_1,\ldots,\bz_N\in\Omega$, the function 
$u_0(\cdot;\bz_1,\ldots,\bz_N)$ is a solution of the Neumann problem
\begin{equation}\label{u0Neumann}
\left\{  \begin{array}{ll}
{\rm div}_{\bx}\left( {\bf L}\nabla_{\bx} u_0(\bx;\bz_1,\ldots,\bz_N)
\right) = 0,&  \bx\in\Omega, \\
{\bf L}\bigl(\nabla_{\bx} u_0(\bx;\bz_1,\ldots,\bz_N)
+ \sum_{i=1}^N{\bf k}_i (\bx;\bz_i)\bigr)\cdot \bn(\bx)\ = 0,&
\bx\in\partial \Omega.
\end{array}\right.
\end{equation}
The expression of \eqref{PKj} contains two contributions accounting for the two different kinds of forces acting on a dislocation when other
dislocations are present:
the interactions with the other dislocations and the interactions with $\partial \Omega$. The latter balances the tractions
 of the forces generated by all the dislocations. Indeed, the function  $\nabla u_0(\bx;\bz_1,\ldots,\bz_N)$ represents the elastic strain
 at the point $\bx \in \Omega$ due to the presence of $\partial \Omega$
 and the dislocations at $\bz_i$ 
 with Burgers moduli $b_i$. For this reason, 
 we refer to $\nabla u_0(\bx;\bz_1,\ldots,\bz_N)$  as the
\emph{boundary-response} strain at $\bx$ due to $\cZ$.

Following \cite{Gurtin}, we will 
assume the dislocations will move in the \emph{glide direction} that maximally dissipates 
the (renormalized) energy. The set of glide directions, $\cG:=\{\bg_1,\ldots,\bg_M\}$, is crystallographically determined
and is discrete. %This leads to jump discontinuities
% in the vector field for the evolution equations.
%Thus, we shall use Filippov's approach \cite{Filippov} to study the dynamics. 

When many dislocations are present, the dynamics is non-trivial.
Dislocations whose Burgers moduli have the same sign will repel each other,
while attraction occurs if the Burgers moduli have opposite signs.
This can be seen by investigating \eqref{PKj} in the case of two dislocations, 
and extended to an arbitrary number of dislocations by superposition,  since  the system \eqref{curldiv} is linear.
%Moreover, the boundary $\partial\Omega$ is attractive:
%the force that it exerts on a dislocation is directed toward the boundary itself, see Section \ref{sec:boundary}.
In addition, because
 $\cG$ a discrete set, the motion need not be continuous with respect to the direction.
\emph{Cross-slip} and \emph{fine cross-slip} may occur whenever it is more convenient for the system to 
switch direction, in the former case,
or to bounce at a faster and faster time scale between two glide directions, in the latter.
In this last situation, macroscopically, a dislocation is able to move along a direction 
which is not in $\cG$,
but belongs to the convex hull  of two glide directions. We discuss
this in more detail in Section \ref{sec:slip}.

Since the direction of the motion of dislocations can change discontinuously and may not be uniquely determined, we cannot use the standard theory of ordinary differential equations to study the dynamics. Instead we will use differential inclusions  (see \cite{Filippov}).

We refer to \cite{AlDeGaPo13,CCM12,P12,V10,ZCA13} and the 
references contained therein for other results on the dynamics of
dislocations. 
In particular, it is important to point out that, due to the discrete set of glide directions and the
maximal dissipation criterion 
introduced in \cite{Gurtin95},
our analysis significantly departs from
that of Ginzburg-Landau vortices, where the motion of vortices
can be derived from a gradient flow
(see the review paper of Serfaty \cite{S11}, see also \cite{AlDeGaPo13}).

In forthcoming work and in collaboration with Thomas Hudson, we plan to study the behavior of dislocations as they approach the boundary
and at collisions. In particular, preliminary results show that dislocations are attracted to the boundary.

\par The structure of the paper is as follows.
Section \ref{sec:dynamics} addresses the dynamics for a system of dislocations: 
a brief introduction on differential inclusion is presented in Subsection \ref{sec:setting}, 
and the framework for the dynamics is presented in Subsection \ref{sec:disdyn}.
Local existence of the solutions to the 
dynamics problem is addressed in Subsection \ref{sec:localexistence},
while Subsection \ref{sec:unique} deals with local uniqueness of the solution.
A description of cross-slip and fine cross-slip is presented in Subsection \ref{sec:slip}, 
where we give analytic proofs of the scenarios presented in
\cite{Gurtin}.
In
Section~\ref{section intersections} we discuss the case of multiple dislocations simultaneously 
exhibiting fine cross-slip and provide  numerical simulations
of the dynamics.
Some special cases are discussed in Section \ref{sec:special}, namely the unit disk (Subsection \ref{sec:undisk}), 
the half-plane and the plane (Subsections \ref{sec:hpp},
\ref{sec:fullplane}), 
and finally the notion of mirror dislocations is introduced in Subsection \ref{sec:mirror}.
We collect some technical proofs in the appendix.
%Section \ref{sec:boundary} deals with the behavior near the boundary.

%\input{dyn_intro}
\section{Dislocation Dynamics}\label{sec:dynamics}

We now turn our attention to the dynamics of 
the system $\cZ$. As explained in the introduction, the direction of the motion of 
dislocations can change discontinuously and this motivates its study using differential inclusions. We begin this section with 
some preliminaries on the theory developed by Filippov \cite{Filippov}.
We introduce the setting for dislocation dynamics in Subsections \ref{sec:setting} and \ref{sec:disdyn},
and prove local existence and uniqueness in Subsections \ref{sec:localexistence}
and \ref{sec:unique}, respectively.

\subsection{Preliminaries on Differential Inclusions}\label{sec:setting}
The theory developed by Filippov \cite{Filippov} 
provides a notion of solution to an ordinary differential inclusion.
Given an interval $I$ and a set-valued function $H:D\to\cP(\R{d})$,
where  $D\subset\RR^{d+1}$   and $\cP(\R{d})$ is the power set of $\R{d}$,
a {\emph {solution}} on $I$ of the differential inclusion 
\begin{equation}\label{f014}
\dot \bx\in H(t,\bx)
\end{equation}
is an absolutely continuous function $\bx:I\to \R{d}$
such that $(t,\bx(t))\in D$ 
and $\dot \bx (t) \in H(t,\bx(t))$ for almost every $t\in I$.
%5equation 
%whose right-hand side is discontinuous in $\bx$. 
%A weak notion of solution is introduced to solve a system of the type
%\begin{equation}\label{f012}
% \dot \bx=\bef(t,\bx)
%\end{equation}
%where $\bef$ is a piecewise continuous function 
%in a domain $\cD(\bef)\subset\R{n+1}$, and its set of points 
%of discontinuity, $M$, has measure zero. 
%The equation \eqref{f012} is replaced by
%a differential inclusion $\dot\bx\in F(t,\bx)$, 
%where $F$ is a set-valued function. 
%For each point $(t,\bx)\in\cD(\bef)$, the set $F(t,\bx)$ 
%is defined as the smallest convex closed set which contains all 
%the limit values of $\bef(t,\bx^*)$, for 
%$(t,\bx^*)\in\cD(\bef)\setminus M$ such that $\bx^*\to\bx$, 
%\begin{equation}\label{convexhull}
%F(t,\bx):=\ch\{\bef(t,\bz^*):\bx^*\in\cD(\bef)\setminus M,\bx^*\to\bx\}.
%\end{equation}
%Note that $F(t,\bx)$ is the singleton 
%$\{\bef(t,\bx)\}$ if $(t,\bx)$ is a point of continuity for $\bef$.
%%As it will be obvious to see, in our case the system is even autonomous, 
%%since there is not explicit dependence on time in the right-hand side.
%Following \cite[page 50]{Filippov}, we give the following definition
%\begin{definition}[see \cite{Filippov}]\label{f013}
%An is said
%to be a solution of the differential inclusion
%
%if $\bx(t)$ is defined on an interval  $I$ and
% $\dot \bx(t)\in F(t,\bx(t))$ for almost every $t\in I$.
%A solution of the discontinuous ODE
%\eqref{f012} is a solution of the differential inclusion \eqref{f014},
%when $F$ is defined as above.
%\end{definition}

In order to state a local existence theorem for \eqref{f014}, 
we need to introduce the definition of \emph{continuity} 
for a set valued map (see \cite{Filippov}).
Given two nonempty sets $A,B\subseteq\R{d}$, 
we recall that the Hausdorff distance between $A$ and $B$ is given by
\begin{equation*}%\label{f102}
d_\cH(A,B):=\max\Big\{\sup_{\ba\in A}\dist(\ba,B),\sup_{\bb\in B}\dist(\bb,A)\Big\}.
\end{equation*}
\begin{remark}\label{f107}
In the special case in which the sets $A$ and $B$ are cartesian products, that is, $A=A_1{\times}A_2\subseteq\R{d_1}\times\R{d_2}$ and $B=B_1{\times}B_2\subseteq\R{d_1}\times\R{d_2}$, we have that
\begin{equation}\label{f108}
d_\cH(A,B)\leq d_\cH(A_1,B_1)+d_\cH(A_2,B_2).
\end{equation}
To see this, let $\ba=(\ba_1,\ba_2)\in A$ and fix $\ep>0$. 
Then there exist $\bb_1^\ep\in B_1$ and $\bb_2^\ep\in B_2$ such that
\begin{equation*}%\label{f109}
||\ba_i-\bb_i^\ep||\leq \dist(\ba_i,B_i)+\ep\qquad\text{for $i=1,2$.}
\end{equation*}
Since $\bb^\ep:=(\bb_1^\ep,\bb_2^\ep)\in B$, we have that 
\begin{equation*}%\label{f110}
\dist(\ba,B)\leq||\ba-\bb^\ep||\leq||\ba_1-\bb_1^\ep||+||\ba_2-\bb_2^\ep||\leq\dist(\ba_1,B_1)+\dist(\ba_2,B_2)+2\ep.
\end{equation*}
Letting $\ep\to0$ and taking the supremum over all $\ba\in A$, it follows that
\begin{equation*}%\label{f111}
\begin{split}
\sup_{\ba\in A}\dist(\ba,B)&\leq\sup_{\ba_1\in A_1}\dist(\ba_1,B_1)+\sup_{\ba_2\in A_2}\dist(\ba_2,B_2) \\
%&=\sup_{\ba_1\in A_1}\dist(\ba_1,B_1)+\sup_{\ba_2\in A_2}\dist(\ba_2,B_2) \\
&\leq d_\cH(A_1,B_1)+d_\cH(A_2,B_2).
\end{split}
\end{equation*}
By exchanging the roles of $A$ and $B$, we obtain \eqref{f108}.
\end{remark}

\begin{definition}[Continuity and Upper Semicontinuity]\label{f103}
Given $D\subset\R{d+1}$ and a set-valued function $H:D \to\cP(\R{d})$, 
we say that $H$ is \emph{continuous} if
\begin{equation*}%\label{f104}
d_\cH(H(\by_n),H(\by))\to0 \qquad\text{for every $\by,\by_n\in D$ such that $\by_n\to \by$.}
\end{equation*}
We say that $H$ is \emph{upper semicontinuous} if 
\begin{equation*}
\sup_{\ba\in H(\by_n)}\dist(\ba,H(\by)) \to 0 
\qquad\text{for every $\by,\by_n\in D$ such that $\by_n\to \by$.}
\end{equation*}
\end{definition}
It follows from the definition that any continuous set-valued function is upper semicontinuous. 

The proof of the following theorem can be found in \cite[pg.\@ 77]{Filippov}.
\begin{theorem}[Local Existence]\label{thm:FilippovEx}
Let $D\subset \RR^{d+1}$ be open and let $H:D\to \mathcal{P}(\RR^{d})$
be upper semicontinuous, and such that $H(t,\bx)$ is nonempty, closed, bounded, and convex
for every $(t,\bx)\in D$. Then for every $(t_0,\bx_0) \in D$ there exist  $h>0$ and 
a solution $\bx:[t_0-h,t_0+h]\to \RR^d$ of the problem
\begin{equation}
\label{inclusion1}
\dot \bx (t) \in H(t,\bx(t)),\qquad \bx(t_0) = \bx_0.
\end{equation}
Moreover, if $D$ contains a cylinder $C:=[t_0-T,t_0+T]\times B_r(\bx_0)$, for some $r, T>0$, then
$h\ge \min\{T,r/m\}$, where
$m := \sup_{(t,\bx)\in C} |H(t,\bx)|$.
\end{theorem}

Next we address uniqueness of solutions to \eqref{inclusion1}. We say that \emph{right uniqueness} holds for \eqref{inclusion1} at a point $(t_0,\bx_0)$ if there exists $t_1>t_0$ such that any two solutions to the Cauchy problem \eqref{inclusion1} coincide on the subset of $[t_0,t_1]$ on which they are both defined. Similarly, we say that \emph{left uniqueness} holds for \eqref{inclusion1} at a point $(t_0,\bx_0)$ if there exists $t_1<t_0$ such that any two solutions to the Cauchy problem \eqref{inclusion1} coincide on the subset  of $[t_1,t_0]$ on which they are both defined.
We we say that \emph{uniqueness} holds for \eqref{inclusion1} at a point $(t_0,\bx_0)$ if both left and right uniqueness hold for \eqref{inclusion1} at  $(t_0,\bx_0)$.

Unlike the case of ordinary differential equations, for differential inclusions the question of uniqueness is significantly more delicate  We will consider here a very special case. 
Suppose that $V\subset \RR^d$ is an open set and is separated into open domains
$V^\pm$ by a $(d-1)$-dimensional $C^2$ surface $S$. Let 
$\bef:(a,b)\times (V\setminus S)\to \RR^d$, and define $\bef^\pm: (a,b)\times V^\pm \to \RR^d$ as 
$\bef^\pm(t,\bx):= \bef(t,\bx)$ for $\bx \in V^\pm$.
Assume that 
%\begin{equation*}%\label{C1extension}
%\text{
$\bef^\pm$ can both be extended in a $C^1$ way to $(a,b)\times V$, and denote these extensions by $\widehat \bef^\pm$.
%\end{equation*} 
Define 
\begin{equation}\label{901}
H(t,\bx) :=\left\{ \begin{array}{lr}
\{\bef(t,\bx)\}  & \text{for }\bx \notin S, \\
{\rm co} \{\widehat\bef^-(t,\bx),\widehat\bef^+(t,\bx)\}  & \text{for }\bx \in S, 
\end{array} \right.
\end{equation}
and consider the differential inclusion \eqref{inclusion1}. 
% If $(t_0,\bx_0)\in (a,b)\times V^\pm$,
% then local existence and uniqueness follow from the theory of ordinary differential equations
% since $\bef^\pm$ are of class $C^1$.
Here for a set $E\subset\mathbb{R}^d$ we denote by $\text{co}E$ the convex hull of $E$, that is, the smallest convex set that contains $E$. 

It can be shown that the function $H$ defined in \eqref{901} satisfies the  conditions of Theorem \ref{thm:FilippovEx},
and local existence follows. 
% If $(t_0,\bx_0)\in(a,b)\times S$, then local existence follows from Theorem \ref{thm:FilippovEx} and
% local uniqueness from 
In the following theorems, we denote by
$\bn(\bx_0)$ the unit normal to $S$ at $\bx_0\in S$ directed from $V^-$ to $V^+$. 
The following theorem can be found in \cite[pg.\@ 110]{Filippov}.

\begin{theorem}[Local Uniqueness]\label{thm:FilippovUn}
Let $H:(a,b)\times V\to\cP(\R{d})$ be given as in \eqref{901}, 
where $\bef$, $V$, and $S$ are as above. 
If $(t_0,\bx_0)\in(a,b)\times S$ is such that 
$\widehat\bef^-(t_0,\bx_0)\cdot\bn(\bx_0)>0$ or $\widehat\bef^+(t_0,\bx_0)\cdot\bn(\bx_0)<0$,
then right uniqueness holds for \eqref{inclusion1} at the point $(t_0,\bx_0)$.

Similarly, if $\widehat\bef^-(t_0,\bx_0)\cdot\bn(\bx_0)<0$ or $\widehat\bef^+(t_0,\bx_0)\cdot\bn(\bx_0)>0$,
then left uniqueness holds for \eqref{inclusion1} at the point $(t_0,\bx_0)$.

\end{theorem}
 Next we discuss cross-slip and fine cross-slip. 

\begin{theorem}[Cross-Slip; \cite{Filippov} Corollary 1, p.107]\label{cs001}
Let $(t_0,\bx_0)\in(a,b)\times S$ be such that 
$\widehat\bef^-(t_0,\bx_0)\cdot\bn(\bx_0)>0$ and $\widehat\bef^+(t_0,\bx_0)\cdot\bn(\bx_0)>0$.
Then uniqueness holds for \eqref{inclusion1} at the point $(t_0,\bx_0)$. Moreover, the unique solution 
$\bx$ to \eqref{inclusion1} passes from $V^-$ to $V^+$, that is, there exist $t_1<t_0<t_2$ such that $\bx(t)$ belongs to $V^-$ for $t\in[t_1,t_0)$ and to 
$V^-$ for $t\in (t_0,t_1]$. 
Similarly, if $\widehat\bef^-(t_0,\bx_0)\cdot\bn(\bx_0)<0$ and 
$\widehat\bef^+(t_0,\bx_0)\cdot\bn(\bx_0)<0$, then uniqueness holds for \eqref{inclusion1} at the point $(t_0,\bx_0)$ and the unique solution passes from $V^+$ to $V^-$.
\end{theorem}

\begin{theorem}[\cite{Filippov} Corollary 2, p.108]\label{cs002}
Let $(t_0,\bx_0)\in(a,b)\times S$ be such that 
\begin{equation}\label{aaa}
\widehat\bef^-(t_0,\bx_0)\cdot\bn(\bx_0)>0\qquad\text{and}\qquad\widehat\bef^+(t_0,\bx_0)\cdot\bn(\bx_0)<0.
\end{equation}
Then there exists $a\le t_1<t_0$ such that
the problem 
\eqref{f014} admits exactly one solution curve
$\bx^{-}$ with
$\bx^- (t)\in V^-$ 
for $t\in (t_1,t_0)$
and $\bx^-(t_0) = \bx_0$, and 
exactly one solution curve
$\bx^+$ with
$\bx^+ (t)\in V^+$ 
for $t\in (t_1,t_0)$
and $\bx^+(t_0) = \bx_0$.
\end{theorem}

\begin{lemma}\label{lem:pass}
Assume that the conditions \eqref{aaa} hold for $(t_0,\bx_0)\in (a,b)\times S$.
%Let $\widehat{\bef}^\pm$ denote $C^1$ extensions of $\bef^\pm$ to
%a neighborhood $(a,b)\times U$ of $(t_0,\bx_0)$. 
Let $\bx(t)$ be a solution to
$\dot \bx = \widehat{\bef}^+(t,\bx)$ on an interval $[t_0,T]$ with 
$\bx(t_0)=\bx_0\in S$.
Then there exists  $\delta>0$ such that 
$\bx(t)\in V^-\cap U$ for $t\in (t_0,t_0+\delta)$. 
Similarly, if 
$\dot \bx = \widehat{\bef}^-(t,\bx)$ 
on an interval $[t_0,T]$ with $\bx(t_0)=\bx_0\in S$,
then there exists  $\delta>0$ 
such that $\bx(t)\in V^+\cap U$ for $t\in (t_0,t_0+\delta)$. 
\end{lemma}

\begin{proof}
  Let $h := \min\{-\widehat\bef^+(t_0,\bx_0)\cdot \bn(\bx_0), 
\widehat\bef^-(t_0\bx_0)\cdot \bn(\bx_0) \}$. 
Then $h>0$ by hypothesis, and therefore, by continuity of
$\widehat{\bef}^\pm$ and $\bn$, %is continuous along $\cA_\ell$,
there exist neighborhoods $I_0$ and $U_0$ of $t_0$ and $\bx_0$,
respectively,
 % and a neighborhood
 %  $U_0$ of $\bx_0$
 such that 
$\widehat{\bef}^+(t,\bx)\cdot \bn(\tilde \bx)<-\frac12 h$ 
and $\widehat{\bef}^-(t,\bx)\cdot \bn(\tilde\bx)>\frac12 h$
for $(t,\bx)\in I_0\times U_0$ and $\tilde \bx\in U_0\cap S$.

We can write $S$ locally as the graph of a function. Denoting
points $\bx = (\bxi,y)\in\R{d-1}{\times}\R{}$,
there is $r>0$ such that we can write
(without loss of generality)
$S\cap B_r(\bx_0) = \{(\bxi,y)\in B_r(\bx_0)\, :\, y=\Phi(\bxi)\}$
for some $\Phi$ of class $C^2$. 
The
sets $V^\pm$ are locally defined as
$V^+\cap B_r(\bx_0)=\{(\bxi,y)\in B_r(\bZ_0)\, :\, y > \Phi(\bxi)\}$ and 
$V^-\cap B_r(\bx_0)=\{(\bxi,y)\in B_r(\bZ_0)\, :\, y < \Phi(\bxi)\}$.
By rotating the coordinate axes, if necessary,  we can assume that the tangent hyperplane to $S$ at $\bx_0$
is $\{(\bxi,y)\, :\, y=0\}$, so that $\nabla \Phi(\bxi_0) = \bzero$, where $\bx_0 = (\bxi_0,y_0)$. 
Then
the unit normal to $S$ at $\bx_0 $ 
is  $\bn(\bx_0)=\bn(\bxi_0,\Phi(\bxi_0)) = (\bzero,1)$.

Consider the solution to $\dot\bx=\widehat{\bef}^+(t,\bx)$ with $\bx(t_0)=\bx_0$.
Since $\bx$ is continuous, there is  $\delta_1>0$ such that  $\bx(t)\in U_0$
for $t\in(t_0,t_0+\delta_1)$, and in this interval it satisfies
$\bx(t) = \bx_0 + \int_{t_0}^t\widehat{\bef}^+(s,\bx(s)) \de s$. Hence,
\begin{equation}
  \label{x_inq1}
y(t)=  \bx(t)\cdot\bn(\bx_0) = \bx_0\cdot\bn(\bx_0) + 
\int_{t_0}^t\widehat{\bef}^+(s,\bx(s))\cdot\bn(\bx_0)\, \de s
< y_0 -\frac{h}{2}(t-t_0).
\end{equation}
Writing $\bx(t) = (\bxi(t),y(t))$, we have
$  \bx(t)\cdot\bn(\bx_0) =y(t)$. Additionally,
$\Phi(\bxi(t)) = \Phi(\bxi(t_0))+\nabla\Phi(\bxi(t_0))\cdot (\bxi(t)-\bxi(t_0)) +o(t-t_0)
=y_0 + o(t-t_0)$. Therefore, \eqref{x_inq1} implies there is $\delta<\delta_1$ 
such that 
\begin{equation*}
%  \label{x_inq2}
  y(t)<\Phi(\bxi(t))-\frac{h}{2}(t-t_0) + o(t-t_0) <\Phi(\bxi(t))
\end{equation*}
for $t\in(t_0,t_0+\delta)$. Thus, $\bx(t) = (\bxi(t),y(t))\in V^-\cap
B_r(\bx_0)$ for $t\in(t_0,t_0+\delta)$. 
The proof of the  result for solutions
to $\dot \bx = \widehat{\bef}^-(t,\bx)$ is similar.
\end{proof}

\begin{corollary}[Fine Cross-Slip] \label{cor:f0}
%Note that, by \eqref{C1extension}, % since ${\bf f}^+$ is extended to a $C^1$ function in $V^+\cup S$ and so 
%any solution originating at $(t_0,\bx_0)$ that stays in
%$V^+\cup S$ is unique. 
Assume that the conditions \eqref{aaa} hold for $(t_0,\bx_0)\in (a,b)\times S$. 
Then there exist $\delta>0$ and a unique solution $\bx$ defined on
$[t_0,t_0+\delta)$ 
to the initial value problem \eqref{inclusion1} that is confined to $S$.
\end{corollary}
\begin{proof}
Existence and uniqueness are 
consequences of Theorems \ref{thm:FilippovEx} and \ref{thm:FilippovUn}.
Let $T$ be the maximal existence time provided by Theorem \ref{thm:FilippovEx}.

As in the proof of Lemma \ref{lem:pass}, there are neighborhoods
$I_0$ and $U_0$  of $t_0$ and $\bx_0$, respectively, such that 
$\widehat{\bef}^+(t,\bx)\cdot \bn(\tilde \bx)<-\frac12 h$ 
and $\widehat{\bef}^-(t,\bx)\cdot \bn(\tilde\bx)>\frac12 h$
for $(t,\bx)\in I_0\times U_0$ and $\tilde \bx\in U_0\cap S$,
with $h = \min\{-\widehat\bef^+(t_0,\bx_0)\cdot \bn(\bx_0), 
\widehat\bef^-(t_0\bx_0)\cdot \bn(\bx_0) \}$. 
By continuity of $\bx(t)$, there exists a $\delta>0$ such that
$\bx(t)\in   U_0$ for $t\in (t_0,t_0+\delta)$.
Suppose there is $t_1\in (t_0,t_0+\delta)$ such that $\bx(t_1) \notin
S$. Without loss of generality, we can assume $\bx(t_1)\in
V^+$, and we define 
\[
s_1:=\sup\{s\in [t_0,t_1)\, :\, \bx(s)\notin V^+\},
\]
i.e., $s_1$ is the last time $\bx(t)$ belongs to $S$ before entering $V^+$ and remaining in $V^+$ for $t\in (s_1,t_1]$. 
%\red{(shouldn't this be the entrance time in $V^+$? Maybe they're the same by continuity)}. 
It follows that  $\bx(t)$ solves $\dot \bx = \widehat\bef^+(t,\bx)$
on $[s_1,t_1]$ with $\bx(s_1) \in S$. 
%Since
%$\widehat{\bef}^+$ is
%an extension of $\bef^+$ into a neighborhood 
%$(a,b)\times U$ of
%$(s_1,\bx(s_1))$ with $U\subset U_0$, as in Lemma \ref{lem:pass}, then
%$\dot\bx=\widehat{\bef}^+(t,\bx)$ on $[s_1,t_1]$ trivially, because
%$\bx(t)\in V^+$ on $[s_1,t_1]$. 
Since the hypotheses of Lemma
\ref{lem:pass} are satisfied,  there
is a unique solution to $\dot\bx=\widehat{\bef}^+(t,\bx)$ on
$[s_1,s_1+\hat \delta]$ for some $\hat \delta >0$ , 
where $\bx(t)\in V^-$ for $t\in (s_1,s_1+\hat\delta)$. This
contradicts the fact that $\bx(t) \in V^+$  on $[s_1,t_1]$.
We conclude that $\bx(t) \in S$ for $t\in [t_0,t_0+\delta)$.
% If there exists $\delta>0$ such that $\bx(t)\in V^+$ for $t\in(t_0,t_0+\delta)$, 
% then $\dot\bx=\bef^+(t,\bx)$ for $t\in[t_0,t_0+\delta)$.
% We have $(\bx(t)-\bx(t_0))\cdot\bn(\bx_0)\geq0$ for $t-t_0$ small enough.
% By dividing by $t-t_0$ and taking the limit as $t\to t_0^+$, we obtain $\dot\bx(t_0)\cdot\bn(\bx_0)\geq0$, which contradicts \eqref{aaa}, because $\dot\bx(t_0)=\bef^+(t_0,\bx_0)$.
% A similar conclusion holds if $\bx(t)\in V^-$ for $t\in(t_0,t_0+\delta)$.
% Therefore, the solution $\bx$ to the initial value problem \eqref{inclusion1} belongs to $S$ for some small time $\ep>0$ following $t_0$.
\end{proof}

\begin{remark}\label{rem:f0}
In view of Corollary \ref{cor:f0}, the velocity field $\dot\bx$ is tangent to $S$, therefore it must be orthogonal to $\bn(\bx)$, for $\bx\in S$. Moreover, by  \eqref{901}, 
$\dot\bx$ belongs to $\ch\{\widehat\bef^-(t,\bx),\widehat\bef^+(t,\bx)\}$, and so,
%Indeed, since $\bef^+(t_0,\bx_0)\cdot\bn(\bx_0)<0$,
%by continuity, $\bef^+(t,\bx)\cdot\bn(\bx_0)<0$ in a neighborhood
%of $(t_0,\bx_0)$ in $V^+\cup S$. Thus, there is no trajectory 
%satisfying $\dot \bx = \bef^+(t,\bx)$ that
%originates at $(t_0,\bx_0)$ and stays in $V^+$ as $t$ increases.
%Similarly, there is no trajectory 
%satisfying $\dot \bx = \bef^-(t,\bx)$ that
%originates at $(t_0,\bx_0)$ and stays in $V^-$ as $t$ increases.
%Thus, if the solution to the initial value problem \eqref{inclusion1}
%satisfies the hypotheses of Theorem \ref{cs002} then it must
%lie in $S$, and therefore
%follow the tangent to $S$. 
%Specifically, it satisfies the ordinary differential equation
\begin{equation*}%\label{f0_def}
\dot\bx=\bef^0(t,\bx)\in H(t,\bx),
\qquad\text{where}\quad \bef^0(t,\bx):=\alpha\widehat\bef^+(t,\bx)+(1-\alpha)\widehat\bef^-(t,\bx)
\end{equation*}
and $\alpha = \alpha(t,\bx)\in(0,1)$ is given by
\begin{equation*}%\label{alpha_def}
\alpha=\frac{\widehat\bef^-(t,\bx)\cdot\bn(\bx)}{\widehat\bef^-(t,\bx)\cdot\bn(\bx)-\widehat\bef^+(t,\bx)\cdot\bn(\bx)},
\end{equation*}
since $\bef^0(t,\bx) \cdot \bn(\bx)=0$.
\end{remark}

\subsection{Setting for the Dynamics}\label{sec:disdyn}
We now turn our attention to the dynamics of 
the system $\cZ$. We will neglect inertia and any
external body forces, and consider only the Peach-K\"ohler
force $\bj_i$ as given in \eqref{PKj}. 

Recall that a screw dislocation is a line in a three-dimensional cylindrical body
$B$, and is represented by a point in the cross-section $\Omega$. 
The motion of dislocations
(often called \emph{dislocation glide}) in crystalline materials
is restricted to a discrete set of crystallographic
planes called \emph{glide planes}, which are
spanned by ${\bf e}_3$ and vectors ${\bf g}$ called
%In the cross-section, the only relevant vector is ${\bf g}$,
%which defines a 
\emph{glide directions}, determined %The set of glide directions
%for a material is fixed 
by the lattice structure of that material.
We will consider the glide directions as a fixed finite collection
of unit vectors in $\R2$, denoted by 
\begin{equation*}%\label{G}
\mathcal G:=\{\bg_1,\ldots,\bg_M\}\subset S^1,
\end{equation*}
with the requirement that if ${\bf g}\in \mathcal G$ then 
$-{\bf g}\in \mathcal G$. 
The dislocation glide is restricted to the directions in $\mathcal G$,
so the equation of motion for $\bz_i$  has the form
\[
\dot \bz_i = \mathcal V_i \bg_i,\quad \bg_i \in \mathcal G
\]
and $\mathcal V_i$ is a scalar velocity. 

In \cite{Gurtin} motion laws are proposed, where a
variable mobility $M(\bg)$ and Peierls force $F(\bg)$ are incorporated
to obtain equations of the form
\begin{equation}
  \label{MF}
  \dot \bz_i = M(\bg_i)[\max\{\bj_i \cdot \bg_i - P(\bg_i) ,0\}]^p \bg_i,
\end{equation}
with the exponent
$p>0$ allowing for various ``power-law kinetics''.  
The mobility function $M$ favors some directions of dislocation glide. The Peierls force, $P\geq 0$, is a
threshold force, acting as a static friction. If the Peach-K\"ohler
force along $\bg_i$ is below the threshold, then the dislocation will
not move. Glide initiates when $\bj_i\cdot \bg_i > P(\bg_i)$.
In this paper we will assume
the simplest form of linear kinetics ($p=1$) 
with vanishing Peierls force ($P\equiv0$) 
and isotropic mobility ($M\equiv 1$). Thus \eqref{MF} takes the form
\begin{equation}
  \label{motion1}
  \dot \bz_i = (\bj_i(\bz_i) \cdot \bg_i ) \bg_i\quad \mbox{for}\,\, \bg_i \in \mathcal G,
\end{equation}
where we recall that 
\begin{equation}\label{PKj1}
\bj_i(\bz_i)=b_i \bJ \bL\Big[\sum_{j\neq i} \bk_j(\bz_i;\bz_j)+
\nabla u_0(\bz_i;\bz_1,\ldots,\bz_N)\Big],
\end{equation}
with $\bk_j$ and $u_0$ given in \eqref{k_def} and \eqref{u0Neumann}, respectively.
\begin{remark}\label{rem:pksmooth}
	The formula \eqref{PKj1} gives the force on 
	the dislocation at $\bz_i$, %. From this form, we see
	and it shows that, as a function of $\bz_i$, 
	the force $\bj_i$ is smooth in the interior
	of $\Omega \setminus \{\bz_1,\ldots,\bz_{i-1},\bz_{i+1},\ldots,\bz_N\}$.
	That is, provided $\bz_i$ is not colliding with another dislocation or with
	$\partial \Omega$, then the force is given by a smooth function.
	Of course, $\bj_i$ depends on the positions of \emph{all} the dislocations,
	and the same reasoning applies to $\bj_i$ as a function of any $\bz_j$. 
\end{remark}

Following the model presented in \cite{Gurtin},
the choice of glide direction in \eqref{motion1} is determined
by  a \emph{maximal dissipation inequality}
for  dislocation glide. This means that the direction of motion of
$\bz_i$ is the glide direction that is most closely aligned with
$\bj_i$. 
Thus, since $\bj_i$ is determined by all the dislocations $\bz_1,\ldots,\bz_N$,
and since $\cG$ is discrete,
the selection of the glide direction $\bg_i\in\cG$ depends in a discontinuous fashion on the dislocations positions.
To stress this fact, we will often write $\bg_i = \bg_i (\bz_1,\ldots, \bz_N)$, $i\in\{1,\ldots,N\}$.
%which is selected for the motion of the dislocation $\bz_i$
%of  , $\bg_i$ will generally be a discontinuous function of
%the dislocation positions. 

We note that, at any point where $\bz_i(t)$ is
differentiable and where \eqref{motion1} is satisfied,
we have $\dot \bz_i = -(\nabla_{\bz_i} U \cdot \bg_i ) \bg_i$ (see \eqref{PK=DU}), and
the energy dissipation inequality 
\begin{equation}
  \label{Udissip}
  \frac{d}{dt}U(\bz_1,\ldots,\bz_N) = 
\sum_{i=1}^N \nabla_{\bz_i} U \cdot \dot \bz_i
= -\sum_{i=1}^N (\nabla_{\bz_i} U \cdot   \bg_i )^2 \leq 0
\end{equation}
holds. The dissipation in \eqref{Udissip} is maximal when $\bg_i$ maximizes $\{\bj_i \cdot \bg\, | \, \bg\in\cG\}$.
 
Note, however, that when there is more than one glide direction $\bg$ that maximizes $\bj_i \cdot \bg$, then \eqref{motion1} {\emph {becomes ill-defined}}. This leads us to consider differential inclusions in place of differential equations. The problem consists in solving the system of differential inclusions
\begin{equation*}%\label{f001}
\left\{\begin{array}{l}
        \dot\bz_\ell \in F_\ell(\bZ), \\
        %= \bef_\ell(\bZ)=(\bj_\ell(\bZ)\cdot\be_\ell(\bZ))\be_\ell(\bZ),\\
	\bz_\ell(0) = \bz_{\ell,0},
       \end{array}
\right. 
\end{equation*}
where 
\begin{equation*}%\label{Z}
\bZ:=(\bz_1,\ldots,\bz_N) \quad\text{and}\quad\bZ_0:=(\bz_{1,0},\ldots,\bz_{N,0})
\end{equation*} 
belong to $\Omega^N\subset \R{2N}$ and,  for  $\ell=1,\ldots,N$,
\begin{equation}\label{f020}
F_\ell(\bZ):=\Big\{(\bj_\ell(\bZ)\cdot\bg)\,\bg:\bg\in\arg\max_{\bg'\in\cG}\{\bj_\ell(\bZ)\cdot\bg'\}\Big\}.
\end{equation}
Setting
\begin{equation}\label{f022}
\cG_\ell(\bZ):=\arg\max_{\bg'\in\cG}\{\bj_\ell(\bZ)\cdot\bg'\},
\end{equation}
the vectors $\bg\in\cG_\ell(\bZ)$ represent the glide 
directions \emph{closest to} $\bj_\ell(\bZ)$ (see \cite{Gurtin}), that is,
\begin{equation}\label{f021}
\bj_\ell(\bZ)\cdot\bg\geq\bj_\ell(\bZ)\cdot\bg',\qquad\text{for all $\bg'\in\cG$}.
\end{equation}
 We are interested in the physically realistic case where the span of
the glide directions is all of $\R2$, otherwise dislocations are
restricted to one-dimensional motion and cannot abruptly change
direction. Therefore, we assume that
\begin{equation}
  \label{spanG}
  {\rm span}(\cG) = \R2.
\end{equation}
When $\bj_\ell(\bZ)\neq0$, the set $F_\ell$ can either contain a single element, which we will call $\bg_\ell(\bZ)$, 
or two distinct elements, denoted by $\bg_\ell^-(\bZ)$ and $\bg_\ell^+(\bZ)$, and in %in case both of them maximize the scalar product $\bj_\ell(\bZ)\cdot\be$: 
this case $\bj_\ell(\bz_\ell)$ is the bisector of the angle formed by $\bg_\ell^-$ and $\bg_\ell^+$. 
\begin{remark}\label{912}
Notice that if $\bj_\ell(\bZ)=\bzero$, then any glide direction $\bg\in\cG$ satisfies \eqref{f021} and therefore $\cG_\ell(\bZ)=\cG$.% and $F_\ell(\bZ)=\{\bzero\}$.
\end{remark}
In view of the comments above, we have
\begin{equation}\label{f023}
F_\ell(\bZ)=
\begin{cases}
\{\bzero\} & \text{if $\bj_\ell(\bZ)=\bzero$}, \\
\{(\bj_\ell(\bZ)\cdot\bg_\ell(\bZ))\,\bg_\ell(\bZ)\} & \text{if $\bj_\ell(\bZ)\neq\bzero$ and $\cG_\ell(\bZ)=\{\bg_\ell(\bZ)\}$}, \\
\{(\bj_\ell(\bZ)\cdot\bg_\ell^\pm(\bZ))\,\bg_\ell^\pm(\bZ)\} & \text{if $\bj_\ell(\bZ)\neq\bzero$ and $\cG_\ell(\bZ)=\{\bg_\ell^\pm(\bZ)\}$}, \\
\end{cases}
\end{equation}
and the problem becomes %we need to solve is
\begin{equation}\label{f006}
\begin{cases}
\dot\bZ \in F(\bZ),\\
\bZ(0) = \bZ_0,
\end{cases}
\end{equation}
where 
\begin{equation}\label{f009}
F(\bZ):=F_1(\bZ){\times}\cdots{\times}F_N(\bZ)\subset\R{2N}.
\end{equation}
%so that it makes sense to write $F_\ell(\bZ);=\pr_\ell \cF(\bZ)=v_\ell G_\ell$; here we have denoted by $\pr_\ell:\R{2N}\to\R2$ the projection on the $(2\ell-1)$-th and $(2\ell)$-th coordinates. 

The domain of the set-valued function $F$ 
must be chosen in such a way that the forces $\bj_\ell(\bZ)$ are well-defined, and so
collisions must be avoided. We denote by
\begin{equation}\label{f024}
\Pi_{jk}:=\{\bZ\in\Omega^N:\bz_j=\bz_k,\,j\neq k\}
\end{equation}
the set where dislocations $\bz_j$ and $\bz_k$ collide, and we 
define the domain of $F$ to be
\begin{equation}\label{f007}
\cD(F):=\Omega^N\setminus\bigcup_{j<k} \Pi_{jk}.
\end{equation}
Recall that the force $\bj_i$ is not defined for $\bz_\ell
\in \partial \Omega$. Since $\Omega$ is open, boundary collisions are also excluded from $\cD(F)$.

\subsection{Local Existence}\label{sec:localexistence}
Following Section \ref{sec:disdyn}, and in view of \eqref{f006} 
and \eqref{f009}, we consider the differential inclusion
\begin{equation}\label{f010}
\left\{\begin{array}{l}
        \dot\bZ \in \ch F(\bZ),\\
	\bZ(0) = \bZ_0.
       \end{array}
\right. 
\end{equation}
The following lemma, whose proof is given in Section \ref{proofoflem:hulls}, shows that the convex 
hull of $F(\bZ)$ is given by 
\begin{equation}\label{f026}
\hat F(\bZ):=(\ch F_1(\bZ)){\times}\cdots{\times}(\ch F_N(\bZ)),
\end{equation}
where, by \eqref{f023},
\begin{equation}\label{f100}
\ch F_\ell(\bZ)=
\begin{cases}
\{\bzero\} & \text{if $\bj_\ell(\bZ)=\bzero$}, \\
\{(\bj_\ell(\bZ)\cdot\bg_\ell(\bZ))\,\bg_\ell(\bZ)\} & 
\text{if $\bj_\ell(\bZ)\neq\bzero$ and $\cG_\ell(\bZ)=\{\bg_\ell(\bZ)\}$}, \\
\Sigma_\ell(\bZ) & \text{if $\bj_\ell(\bZ)\neq\bzero$ and $\cG_\ell(\bZ)=\{\bg_\ell^\pm(\bZ)\}$}, \\
\end{cases}
\end{equation}
with $\Sigma_\ell(\bZ)$ the segment of endpoints $(\bj_\ell(\bZ)\cdot\bg_\ell^-(\bZ))\,\bg_\ell^-(\bZ)$  and $(\bj_\ell(\bZ)\cdot\bg_\ell^+(\bZ))\,\bg_\ell^+(\bZ)$.

\begin{lemma}\label{lem:hulls}%f027
Let $F_\ell(\bZ)$ be defined as in \eqref{f020} for $\ell=1,\ldots,N$, and let $F(\bZ)$ be as in \eqref{f009}.% let . 
Then $\ch F(\bZ)=\hat F(\bZ)$, where $\hat F(\bZ)$ is defined  in \eqref{f026}.
\end{lemma}
Lemma \ref{lem:hulls} is useful for understanding the dynamics in
$\Omega$ rather than in $\Omega^N$. Each $\bz_i$ moves in some
direction $\bg_i\in \mathcal{G}$, unless the $\arg\max$ in
\eqref{f022} is multivalued, in which case $\bz_i$ moves in a
direction belonging to the convex hull of $\bg_i^+$ and $\bg_i^-$. 
Lemma \ref{lem:hulls} makes this precise and validates the use of
\eqref{f010} as our model for dislocation motion.

\begin{lemma}\label{f105}
Let $\cD(F)$ be defined in \eqref{f007}. 
Then the set-valued map $F:\cD(F)\to\cP(\R{2N})$ defined in \eqref{f009} is continuous (according to Definition \ref{f103}).
%The set-valued map $F:\cD(F)\to\cP(\R{2N})$ defined in \eqref{f009} and \eqref{f007} is continuous 
%(according to Definition \ref{f103}).
\end{lemma}
\begin{proof} %Only in this proof (to simplify notation), if $\bZ\in\cD(F)$, $\bj_\ell(\bZ)\neq\bzero$, 
%and $\cG_\ell(\bZ)=\{\bg_\ell(\bZ)\}$ (see \eqref{f023}), 
%we let $\bg_\ell^-(\bZ)=\bg_\ell^+(\bZ):=\bg_\ell(\bZ)$.
Let $\bZ,\bZ_n\in\cD(F)$ be such that $\bZ_n\to\bZ$ as $n\to\infty$. 
In view of Remark \ref{f107}, %and Lemma \ref{lem:hulls}, 
it suffices to show that for every $\ell\in\{1,\ldots,N\}$, 
\begin{equation*}%\label{f106}
d_\cH(F_\ell(\bZ_n),F_\ell(\bZ))\to0\qquad\text{as $n\to\infty$}.
\end{equation*}
Fix $\ell\in\{1,\ldots,N\}$. We consider the two cases $\bj_\ell(\bZ)=0$ and $\bj_\ell(\bZ)\neq 0$.

If $\bj_\ell(\bZ)=0$, then by \eqref{f023} $F_\ell(\bZ)=\{\bzero\}$. 
In turn, again by \eqref{f023} the continuity of $\bj_\ell$ (cf. Remark \ref{rem:pksmooth} and \eqref{f007}), 
$d_\cH(F_\ell(\bZ_n),\bzero)\leq||\bj_\ell(\bZ_n)||\to0$ 
as $n\to\infty$.

If $\bj_\ell(\bZ)\neq\bzero$, then, again by continuity of 
$\bj_\ell$, $\bj_\ell(\bZ_n)\neq\bzero$ for all $n\geq\bar n$, for some $\bar n\in\N$.
Taking $\bar n$ larger, if necessary, we claim that 
$\bg_\ell^-(\bZ_n),\bg_\ell^+(\bZ_n)\in\{\bg_\ell^-(\bZ),\bg_\ell^+(\bZ)\}$ for $n\geq\bar n$.
Arguing by contradiction, if the claim fails, since $\cG$ is finite, 
there exists $\be\in\cG\setminus\{\bg_\ell^\pm(\bZ)\}$ such that 
$\bg_\ell^-(\bZ_n)=\be$ or $\bg_\ell^+(\bZ_n)=\be$ for infinitely many $n$. 
By \eqref{f021} and \eqref{f022}, $\bj_\ell(\bZ_n)\cdot\be\geq\bj_\ell(\bZ_n)\cdot\bg$ 
for all $\bg\in\cG$ and for infinitely many $n$.
Letting $n\to\infty$ and using the continuity of $\bj_\ell$, 
it follows that $\bj_\ell(\bZ)\cdot\be\geq\bj_\ell(\bZ)\cdot\bg$ for all $\bg\in\cG$, 
which implies that $\be\in\cG_\ell(\bZ)$, which is a contradiction. Thus the claim holds.

In particular, we have shown that
%\begin{equation}\label{f112}
$F_\ell(\bZ_n)=\{(\bj_\ell(\bZ_n)\cdot\bg_\ell^\pm(\bZ))\bg_\ell^\pm(\bZ)\}$ %\qquad\text{
for $n\geq\bar n$, %},
%\end{equation}
hence $d_\cH(F_\ell(\bZ_n),F_\ell(\bZ))\leq||\bj_\ell(\bZ_n)-\bj_\ell(\bZ)||\to0$ as $n\to\infty$. 
This concludes the proof.
\end{proof}

%\note{$\cD(F)$ is never mentioned, but it is (should be) the same as $\cD(G)$}
\begin{corollary}\label{f1071}
Let $F:\cD(F)\to\cP(\R{2N})$ be defined by \eqref{f009} and \eqref{f007}, and consider the set 
valued map $\ch F(\bZ)$, $\bZ\in\cD(G)$.
Then $\ch F(\bZ)$ is nonempty, closed, convex for every $\bZ\in\cD(F)$, and $\ch F$ is continuous.
\end{corollary}
\begin{proof}
For all $\bZ\in\cD(F)$, the set $\ch F(\bZ)$ is nonempty because $F(\bZ)$ is nonempty. 
By definition of convexification, $\ch F(\bZ)$ is  closed and convex. 
By Lemma \ref{f105}, the set valued map $F$ is continuous, and 
therefore so is $\ch F$ (see Lemma 16, page 66 in \cite{Filippov}).
This corollary is proved.
\end{proof}

Note that $\ch F$ is not bounded on $\cD(F)$ because $|\bz_i-\bz_j|$ 
and $\dist(\bz_i,\partial \Omega)$ can become arbitrarily small, and
thus $\bj_i$ can become unbounded (see \eqref{PKj} and \eqref{k_def}). 

\begin{theorem}[Local existence]\label{f008}
Let $\Omega\subset\R2$ be a connected open set. Let $F:\cD(F)\to\cP(\R{2N})$ be 
defined as in \eqref{f009} and \eqref{f007} with each $F_\ell$ as in \eqref{f023}, 
and let $\bZ_0\in\cD(F)$ be a given initial configuration of dislocations. 
%Finally, let $F(\bZ):=\ch G(\bZ)$ be the convex hull of the set valued function $G$. 
%Consider the initial value problem
%\begin{equation}\label{f0101}
%\left\{\begin{array}{l}
%        \dot\bZ \in \ch F(\bZ),\\
%	\bZ(0) = \bZ_0.
%       \end{array}
%\right. 
%\end{equation}
Then there exists a solution $\bZ:[-T,T]\to\cD(F)$ to  \eqref{f010}, with $T\ge r_0/m_0$,  where
\begin{equation}\label{f200}
0<r_0<\dist(\bZ_0,\partial\cD(F))\;\;\text{and}\;\; 
m_0:=\max_{\bZ\in\overline{B(\bZ_0,r_0)}}\Bigg(\sum_{\ell=1}^N |\bj_\ell(\bZ)|^2\Bigg)^{1/2}.
\end{equation}
\end{theorem}
\begin{proof}
The function $F$ is bounded on the ball $B(\bZ_0,r_0)\subset \cD(F)$. Hence,
by Corollary \ref{f1071}, the set valued map $\ch F$ satisfies the conditions of 
Theorem \ref{thm:FilippovEx} in $B(\bZ_0,r_0)$, and thus local existence holds.
\end{proof}
\begin{remark}\label{f201}
In view of \eqref{f007} and \eqref{f200}, solutions to the problem \eqref{f010} exist as long as dislocations stay away from $\partial\Omega$ and do not collide.
\end{remark}

\subsection{Local Uniqueness}\label{sec:unique}
The set where dislocations can move in
either of two different glide directions is called \emph{ambiguity set}
and denoted by $\cA$. To be precise, 
we define  
\begin{equation}
  \label{Amb_def}
  \cA := \bigcup_{\ell=1}^N \cA_\ell,\quad{\rm where}\quad
\cA_\ell := \left\{ \bZ \in \cD(F)\, : \, {\rm card}(\cG_\ell(\bZ))=2\right\},
\end{equation}
and $\cG_\ell(\bZ)$ is defined in \eqref{f022}. On $\cA_\ell$ the direction of the Peach-K\"ohler
force
$\bj_\ell$ bisects two different glide directions that are closest to it. 
Note that $\bj_\ell(\bZ) \neq {\bf 0}$ for $\bZ\in \cA_\ell$,
because $ {\rm card}(\cG)\geq 4$ by assumption \eqref{spanG} and since
$\bg\in \mathcal{G}$ implies $-\bg\in \cG$. 

 The uniqueness results in Subsection \ref{sec:setting} can only be applied at points $\bZ_0\in \cA$ in which the ambiguity set  $\cA$ is locally a $(2N-1)$-dimensional smooth surface separating $\cD(F)$ into two open sets in a neighborhood of $\bZ_0$. 
  In this subsection we show that $\cA$ is a $(2N-1)$-dimensional smooth surface 
outside of a ``singular set'' and we estimate the Hausdorff dimension of this set.

\begin{lemma}\label{lem:janalytic}
For all $\ell\in\{1,\ldots,N\}$ 
the functions $\bj_\ell (\bz_1,\ldots,\bz_N)$ are analytic on 
any compact subset of $\cD (F)$.
\end{lemma}
\begin{proof}
Observe that if a smooth function $v$ satisfies the partial differential equation
${\rm div}\left( {\bf L}\nabla v\right) = 0$ in $\Omega$, then the function 
$w(x_1,x_2):=v(\lambda x_1,x_2)$ satisfies the  partial differential equation
$\Delta w = 0$ in an open set $U$. Hence, without loss of generality, we may assume that $\lambda = 1$ (i.e. $\bL = \mu \bI$), so that \eqref{laplaciank} and \eqref{u0Neumann} reduce to
\begin{equation}\label{k harmonic1}
	\Delta_{\by}\bk_j (\bx;\by)
	= \bzero, \qquad\quad\,(\bx,\by)\in\RR^2{\times}\R2,\;\bx\neq\by, 
\end{equation}
and, for fixed $\bz_1,\ldots,\bz_N\in\Omega$,
\begin{equation}
  \label{neumann}
\left\{  \begin{array}{l}
  \Delta_{\bx} u_0 (\bx;\bz_1,\ldots,\bz_N)= 0,
\qquad\qquad\qquad\qquad\qquad\qquad\quad\, \bx\in\Omega, \\ 
  \nabla_{\bx} u_0 (\bx;\bz_1,\ldots,\bz_N)\cdot \bn (\bx)
= -\sum_{i=1}^N 
\bk_i(\bx;\bz_i)\cdot \bn(\bx),\quad \bx\in\partial\Omega.
  \end{array}\right.
\end{equation}
A solution to \eqref{neumann} is given by
\begin{equation}
  \label{u_iso}
  u_0(\bx;\bz_1,\ldots, \bz_N) = \int_{\partial \Omega}G(\bx,\by) 
\sum_{i=1}^N \bk_i(\by;\bz_i)\cdot \bn(\by) \,\de s(\by),
\end{equation}
where $G$ is the Green's function for the Neumann problem.
Consider $u_0$ as a function in $\Omega^{N+1}\subset \mathbb{R}^{2N+2}$.
%We continue to write $\bZ : = (\bz_1,\ldots,\bz_N) = (z_{1,1},z_{1,2},\ldots,$ $z_{N,1},z_{N,2})$
%so that  $(\bx,\bZ)\in \mathbb{R}^{2N+2}$. 
Fix $K_i \subset \subset \Omega$ for $i=0,\ldots,N$. 
If $(\bx,\bZ)\in K:=K_0\times K_1\times\cdots\times K_N$, 
then the integrand
in \eqref{u_iso} is uniformly bounded, and we can find the 
derivatives of $u_0$ with respect to each
$z_{i,m}$ by differentiating under the integral sign in \eqref{u_iso}.

Using \eqref{k harmonic1}, \eqref{neumann}, and \eqref{u_iso}
we have
\begin{align*}
  \Delta_{(\bx,\bZ)}u_0 &= \Delta_\bx u_0 + \Delta_{\bz_1}u_0 + \cdots +\Delta_{\bz_N}u_0\\
& = 0 + \sum_{i=1}^N \int_{\partial \Omega}G(\bx,\by) 
\Delta_{\bz_i}(\bk_i(\by;\bz_i)\cdot \bn(\by)) \,\de s (\by) = 0.
\end{align*}

% Note that if we compute $\Delta_{(\bz_i,\bZ)}u$ for some $\bz_i$, we get cross-terms. So,
% the Laplacian of the force $\bj_i$, which involves $\nabla u(\bz_i,\bZ)$, 
% gives cross terms. 
% But, here, we are working in $\mathbb{R}^{2N+2}$ and not in $\mathbb{R}^{2N}$, and we
% do not get any cross-terms.

Observe that in a small ball around $(\bx,\bZ)\in K$, $u_0$ is a $C^2$ function in each variable
 because the formula \eqref{u_iso} has singularities only
on the boundary. 
Since a harmonic $C^2$ function on an open set 
is analytic in that set (cf. \cite[Chapter 2]{Evans}), we deduce that 
$u_0$ is analytic in
the interior of $\Omega^{N+1}$, and thus
$u_0(\bz_i;\bZ)$ is also analytic (though, possibly no longer harmonic).
By \eqref{PKj1} we have that
$\bj_\ell$ is analytic away from the boundary and away from collisions,
because in this case each $\bk_i(\bz_\ell;\bz_i)$ is harmonic in both $\bz_\ell$ and $\bz_i$.% away from collisions.
\end{proof}

%%%%%%%%%%%%%
Fix $\bZ^*\in \cA_\ell$. There are two maximizing glide directions
for $\bz_\ell$, denoted by $\bg_\ell^+(\bZ^*)$ and $\bg_\ell^-(\bZ^*)$
(i.e. $\cG_\ell(\bZ^*) = \{\bg_\ell^+(\bZ^*),\bg_\ell^-(\bZ^*)\}$, as defined in \eqref{f022}).
For simplicity we will write $\bg_\ell^\pm:=\bg_\ell^\pm(\bZ^*)$.
Let $B_h(\bZ^*)$ be a ball around
$\bZ^*$ with radius $h>0$ small enough so that $B_h(\bZ^*)\subset \cD(F)$,
and for any $\bZ \in B_h(\bZ^*)$ one of the following three
possibilities holds: $\cG_\ell(\bZ) = \{\bg_\ell^+\}$, $\cG_\ell(\bZ) = \{\bg_\ell^-\}$,
or $\cG_\ell (\bZ) = \{\bg_\ell^+,\bg_\ell^-\}$. Such  $h$ exists
because of the continuity of $\bj_\ell$ and the fact that
$\bj_\ell(\bZ^*) \neq {\bf 0}$ (cf. the discussion following \eqref{Amb_def}).
We denote by $\bg_0\in \R{2}$ the vector 
\begin{equation}
  \label{e0}
  \bg_0 := \bg_\ell^+ - \bg_\ell^-,
\end{equation}
which is a well-defined constant vector for $\bZ\in B_h(\bZ^*)$ (see the proof of Lemma \eqref{f105}).
Note that if $\partial^{\bbeta} \bj_\ell(\bZ^*)\cdot \bg_0\neq 0$ 
for some multi-index $\bbeta=(\beta_1,\ldots,\beta_N)\in\mathbf{N}_0^N$
with $|\bbeta|=1$, then $\cA_\ell$ is locally a smooth manifold. 
With $\bg_0$ as in \eqref{e0}, we define the \emph{singular sets} 
\begin{equation}
  \label{S_def}
\cS_\ell: = \{\bZ \in \cA_\ell\, : \,
\bj_\ell(\bZ)\cdot \bg_0 = 0, \, \nabla_{\bZ}(\bj_\ell(\bZ)\cdot \bg_0) = \bzero \},
\quad \ell =1,\ldots,N.
\end{equation}
Each $\cS_\ell$ contains the points
where $\cA_\ell$ could fail to be a manifold, and is an obstruction 
to uniqueness of solutions to \eqref{f010}.

We now estimate the Hausdorff dimension of the singular sets.
We adapt an argument from \cite{HanLin}, which 
follows \cite{caffarelli}; 
recall that $\cS_\ell\subset\RR^{2N}$, $\ell=1,\ldots,N$.
\begin{lemma}\label{lem:singularset}
Let $\cS_\ell$ be defined as in \eqref{S_def}. Then
${\rm dim}(\cS_\ell) \leq 2N-2$.
\end{lemma}
\begin{proof}
Fix  $\ell \in \{1,\ldots,N\}$ and $\bZ^* \in \cA_\ell$.
As in the discussion above, set
$\bg_0 := \bg_\ell^+ - \bg_\ell  \in \R2 \setminus\{\bzero\}$,
where $\bg^\pm_\ell$ are uniquely defined in $B_h(\bZ^*)$ for
$h>0$ small enough.

We will be considering derivatives in all the $\bz_i$ directions
except for $i = \ell$. For this purpose, we introduce the notations
$\Delta_{\widehat \bZ_\ell}$, $\nabla_{\widehat \bZ_\ell}$, and
$D^2_{\widehat \bZ_\ell}$ to denote
the Laplacian, the gradient, and the Hessian
with respect to $\bz_1,\ldots,\bz_{\ell-1},
\bz_{\ell+1},\ldots,\bz_N$, respectively. We also write $N_\ell$ for the 
set of multi-indices $\balpha$ such that $\partial^{\balpha}$
does not contain any derivatives in the $\bz_\ell$ directions,
that is,
\begin{equation}
  \label{Nell}
N_\ell: = \{\balpha\in\mathbb{N}_0^{2N} \, : \, \balpha = (\balpha_1,\ldots,\balpha_{\ell-1},{\bf 0},
\balpha_{\ell+1},\ldots,\balpha_N)\}.
\end{equation}
For $m\geq 2$ we define
\begin{align*}
\widetilde M_\ell^m:=  \{ \bZ\, : \,  &\bj_\ell(\bZ)\cdot \bg_0 = 0, \
\partial^{\balpha}(\bj_\ell(\bZ)\cdot \bg_0) = 0 \text{ for all }\balpha \in N_\ell 
\text{ such that }|\balpha|< m, \nonumber \\
&\text{and } \partial^{\balpha}(\bj_\ell(\bZ)\cdot \bg_0) \neq 0
\text{ for some } \balpha \in N_\ell, \text{ with } |\balpha|=m \},\label{Mm}
\end{align*}
and also
\begin{equation}
  \label{Minfty2}
  \widetilde M_\ell^\infty : =  \{ \bZ\, : \,  \bj_\ell(\bZ)\cdot \bg_0 = 0, \
\partial^{\balpha}(\bj_\ell(\bZ)\cdot \bg_0) = 0 \text{ for all }\balpha \in N_\ell \}.
\end{equation}
Therefore %$\cS_\ell = \{\bZ\, : \, \bj_\ell(\bZ)\cdot \be_0 = 0, \, \nabla_{\bZ}(\bj_\ell(\bZ)\cdot\be_0) = 0 \}$, so
\begin{align*}
  \cS_\ell 
\subset \{\bZ\, : \, 
\bj_\ell(\bZ)\cdot \bg_0 = 0, \, \nabla_{\widehat \bZ_\ell}(\bj_\ell(\bZ)\cdot \bg_0) = \bzero \}
= \widetilde M_\ell^\infty\cup\Bigg( \bigcup_{m\geq 2}\widetilde M_\ell^m \Bigg).
\end{align*}
By Lemma \ref{lem:Mempty} in the appendix,
we have that $\widetilde M_\ell^\infty = \emptyset$. 

\noindent Let $m\geq 2$ and let $\bZ_0 \in \widetilde M_\ell^m$.
Then there exists $\bbeta \in N_\ell$ such that $|\bbeta|=m-2$, and
\[
D^2_{\widehat \bZ_\ell}(\partial^{\bbeta}\bj_\ell(\bZ_0)\cdot \bg_0) \neq \bzero.
\]
Thus,
if we define $v(\bZ) := \partial^{\bbeta}\bj_\ell(\bZ)\cdot \bg_0$,
then  $D^2_{\widehat \bZ_\ell}v(\bZ_0)$
is a symmetric matrix that is not identically zero, so it must have
at least one non-zero eigenvalue, say $\lambda_i$. 

Observe that 
${\rm Trace}(D^2_{\widehat \bZ_\ell}v(\bZ))
=\Delta_{\widehat \bZ_\ell}(\partial^{\bbeta}\bj_\ell(\bZ)\cdot \bg_0) = 0$
because $\Delta_{\widehat \bZ_\ell}(\bj_\ell(\bZ)\cdot \bg_0) = 0$.
But 
${\rm Trace}(D^2_{\widehat \bZ_\ell}v(\bZ_0))
=\sum_{k=1}^{2N-2}\lambda_k$, where $\lambda_k$ are the eigenvalues, and
$\lambda_i \neq 0$, and so there is another non-zero eigenvalue, say $\lambda_j$.
Define $w(\bY) :=v(R\bY)$, where $R$ is a rotation matrix such that
\begin{equation*}
D^2_{\widehat \bY_\ell}w(\bY_0)=
\left(
\begin{array}
[c]{ccc}%
\lambda_{1} & \cdots & 0\\
\vdots & \ddots & \vdots\\
0 & \cdots & \lambda_{2N-2}%
\end{array}
\right),
\end{equation*} 
where $\bY_0:=R^{-1}\bZ_0$. 
Since $\lambda_i$ and $\lambda_j$ are different from zero,
there are two distinct multi-indices $\balpha_1,\balpha_2\in N_\ell$
with $|\balpha_k|=1$ such that
\[
\nabla_{\widehat \bY_\ell} \partial^{\balpha_k}w(\bY_0) \neq \bzero,\quad k=1,2.
\]
Hence, applying the Implicit Function Theorem to 
$\partial^{\balpha_1}w$ and $\partial^{\balpha_2}w$, we conclude
that $\cM = \{\bY\, :\, \partial^{\balpha_1}w(\bY)=0,\,
\partial^{\balpha_2}w(\bY)=0\}$ is a $(2N-2)$-dimensional manifold in a 
neighborhood of $\bY_0$.  Since $\widetilde M_\ell^m \subset \cM$,
we have that $\cS_\ell$ is contained in a 
countable union of manifolds with
dimension at most $2N-2$.
\end{proof}

We proved that the collection of \emph{singular points}
\begin{equation*}
%  \label{singular_points}
  \mathcal{E}_{\rm sing} : = \bigcup_{\ell =1}^N \cS_\ell,
\end{equation*}
with $\cS_\ell$ defined in \eqref{S_def}, has dimension at most
$2N-2$. 
Further, each $\cA_\ell$ is a $(2N-1)$-dimensional smooth manifold away from points on $\cS_\ell$ but, 
in general, the set $\cA$ defined in \eqref{Amb_def} will not
be a manifold at points $\bZ\in\cA_\ell \cap \cA_j$ for $\ell \neq j$.
For this reason we need to exclude the set
\begin{equation}
  \label{intersections}
  \mathcal{E}_{\rm int} := \{\bZ \in\R{2N}\, : \, \bZ\in \cA_\ell\cap\cA_j\; 
\text{for some}\; \ell,j\in\{1,\ldots,N\}, \,\ell \neq j\}.
\end{equation}
%% should we make a statement about transversal intersections or coinciding surfaces?
Uniqueness at points in $\mathcal{E}_{\rm int} $ is significantly more delicate and will be discussed in  Section \ref{section intersections}.
%Remark \ref{rem:codim} below.

%In view of Remark \ref{912}, we set
If $\bZ \in \cA_\ell$, then $\bj_\ell(\bZ) \neq \bzero$, but it could be that $\bj_i(\bZ)=\bzero$ for some
$i\neq \ell$. This would mean that the glide direction for $\bz_i$ would  not be  well-defined at $\bZ$,
and could cause an obstruction to uniqueness. In view of this, we set
\begin{equation*}%\label{911}
\cE_{\rm zero}:=\left\{\bZ\in\cD(F): \bj_k(\bZ)=\bzero \text{ for some } k\in\{1,\ldots,N\}\right\}.
\end{equation*}
%Points $\bZ\in\cE_{\rm zero}$ could cause obstruction to uniqueness. 
%The set $\cE_{\rm zero}$ has dimension at most $2N-2$.
Reasoning as in Lemma \ref{lem:singularset},
$\dim(\cE_{\rm zero}\cap \{\nabla \bj_k\;\text{has rank $0$}\})\leq 2N-2$.
On the other hand, $\dim(\cE_{\rm zero}\cap \{\nabla \bj_k\;\text{has rank $2$}\})= 2N-2$, by the Implicit Function Theorem.
The set $\cE_{\rm zero}\cap \{\nabla \bj_k\;\text{has rank $1$}\}$ could have dimension at most $2N-1$.

%
%let $\bZ_0\in\cE_{\rm zero}$.
%
%If $\nabla\bj_k(\bZ_0)$ has rank $1$, then there exists a vector $\be\in\R2\setminus\{\bzero\}$  such that $\nabla\bj_k(\bZ_0)\,\be=\bzero$, hence $\bZ_0\in \{\bZ\in\cD(F):  \bj_\ell(\bZ)\cdot \be = 0, \, \nabla_{\bZ}(\bj_\ell(\bZ)\cdot \be) = \bzero \}$.
%\textcolor{red}{
%By the Implicit Function Theorem, 
%$\dim(\cE_{\rm zero}\cap \{\nabla \bj_k\ne \bzero\})= 2N-1$, 
%while  by Lemma \ref{lem:singularset}, 
%$\dim(\cE_{\rm zero}\cap \{\nabla \bj_k= \bzero\})\leq 2N-2$ .}

For each $\ell\in\{1,\ldots,N\}$ define
\begin{equation}\label{initial}
\cI_\ell:=\cA_\ell\setminus(\cS_\ell\cup \mathcal{E}_{\rm int}\cup\cE_{\rm zero}).
\end{equation}

Let $\hat{\bZ}\in \cI_\ell$.
Since $\hat\bZ\notin\cS_\ell$ (see \eqref{S_def}),
there is an $r>0$
so that $B_r(\hat{\bZ})\cap\cA_\ell$ is a $(2N-1)$-dimensional smooth manifold, and $\cA_\ell$
divides $B_r(\hat{\bZ})$ into two disjoint, open sets $V^\pm$.
Since the functions $\bj_k$ are continuous by Lemma \ref{lem:janalytic} for all $k\in\{1,\ldots,N\}$,
and $\hat\bZ\notin\cE_{\rm zero}$, by taking $r$ smaller, if necessary, we can assume that
$\bj_k(\bZ)\neq\bzero$ for all $\bZ\in B_r(\hat\bZ)$ and for all $k\in\{1,\ldots,N\}$.
In turn, since $\hat\bZ\notin\cE_{\rm int}$, again by continuity and by taking $r$ even smaller, 
$\bg_k(\bZ)\equiv\bg_k(\hat\bZ)$ for all $\bZ\in B_r(\hat\bZ)$ and for all $k\neq\ell$,
and $\bg_\ell(\bZ)\equiv\bg_\ell^\pm(\hat\bZ)$ for $\bZ\in V^\pm$.
Let now
%Thus, we are in a position to apply Theorem \ref{thm:FilippovUn} to the function 
$\bef:B_r(\hat\bZ)\setminus \cA_\ell\to\R{2N}$,
$\bef=(\bef_1,\ldots,\bef_N)$, be the function defined by
\begin{equation}\label{910}
\begin{array}{ll}
\bef_k(\bZ):= (\bj_k(\bZ)\cdot\bg_k(\hat\bZ))\bg_k(\hat\bZ) & \text{if $k\neq\ell$} ,\\ [2pt]
\bef_\ell(\bZ):= (\bj_\ell(\bZ)\cdot\bg_\ell^\pm(\hat\bZ))\bg_\ell^\pm(\hat\bZ) & \text{if $\bZ\in V^\pm$}.
\end{array}
\end{equation}
We define $\bef^\pm$ as the restrictions of $\bef$ to $V^\pm$, and we extend them smoothly to the ball $B_r(\hat\bZ)$ by setting
$\widehat\bef_k^\pm(\bZ):=\bef_k(\bZ)$ if $k\neq\ell$ and $\widehat\bef_\ell^\pm(\bZ):=(\bj_\ell(\bZ)\cdot\bg_\ell^\pm(\hat\bZ))\bg_\ell^\pm(\hat\bZ)$.
%Uniqueness of solutions can also be ruined by the existence of 
%\emph{source points}. 

Let $\bn(\hat{\bZ})$ denote the unit normal vector to $\cA_\ell$ at $\hat{\bZ}$ 
directed from $V^-$ to $V^+$. 
%but multi-valued at $\cA_\ell$. 
%Let $F^\pm$ be the
%limit value of $F(\bZ^\pm_n)$ for any sequences 
%$\{\bZ^\pm_n \}_{n\geq 1}\subset V^\pm$
%converging to $\hat{\bZ}$. 
%Define
%Let $F_\bn^\pm = $ be the projections
%of $F^\pm$ on the normal. 
Motions starting in $V^+$ will move towards or away from $\cA_\ell$
according to whether $\widehat\bef^+(\hat\bZ) \cdot \bn(\hat{\bZ})<0$ or $\widehat\bef^+(\hat\bZ) \cdot \bn(\hat{\bZ})>0$.
Similarly, motions starting in $V^-$ will move towards or away from $\cA_\ell$
according to whether $\widehat\bef^-(\hat\bZ) \cdot \bn(\hat{\bZ})>0$ or $\widehat\bef^-(\hat\bZ) \cdot \bn(\hat{\bZ})<0$. 

We define the set of \emph{source points}
\begin{align*}
%\label{sources}
\cE_{\rm src}:=\large\{\bZ\in\Omega^N: \, & \bZ\in\cI_\ell \text{ for some }\ell\in\{1,\dots,N\},\, \\
&\widehat\bef^+(\bZ) {\cdot} \bn({\bZ})>0\;\text{and}\;\widehat\bef^-(\bZ) {\cdot}\bn({\bZ})<0\large\}.
\nonumber
\end{align*}
If $\hat\bZ\in\cE_{\rm src}$ there are two solution curves originating at $\hat\bZ$, 
one that moves into $V^+$ and one that moves into $V^-$.
Thus there is no uniqueness at source points.
%For points in $\cA_\ell \setminus (\cS_\ell\cup \mathcal{E}_{\rm int}\cup\cE_{\rm zero})$ 
%that are not source points, we are in a position to apply Theorem \ref{thm:FilippovUn}. % have proved the following
%then a motion starting at $\hat{\bZ}$
%will have two available directions $\bef^\pm(\hat\bZ)$, both of which would lead to
%motion off of the surface $\cA_\ell$. Such a point is a source point,
%and uniqueness of solutions fails at source points. 
%We define
%\begin{equation}
%  \label{sources}
%  \mathcal{E}_{\rm src}:=\bigcup_{\ell=1}^N \{ 
%\bZ \in \cA_\ell \setminus (\cS_\ell\cup \mathcal{E}_{\rm int}\cup\cE_{\rm zero})\, :\,
%\bef^+(\hat\bZ) \cdot \bn(\hat{\bZ})>0\;\; {\rm and}\; \; \bef^-(\hat\bZ) \cdot \bn(\hat{\bZ})<0.
%\}
%\end{equation}
% \begin{definition}\label{def:source}
%   Let $\bZ\in \cA_\ell \setminus \cS_\ell$ for some $\ell =1,\ldots,N$.
% If $G^+_\bn >0$ and $G_{\bn}^- <0$ then $\bZ$ is a \emph{source point}. 
% We denote the set of source points as 
% $\mathcal{SP}$.
% \end{definition}
% \textcolor{blue}{Write discussion about the application of Filippov's
%   theorem. The notions of $f^\pm_N$ and the normal projection criteria
%   that allow us to state the uniqueness theorem.}
\begin{theorem}[Local Uniqueness]\label{thm:uniqueness}
	Let $T>0$ and let $\bZ:[-T,T]\to\mathbb{R}^{2N}$ be a solution to \eqref{f010}.
	Assume that there exist $t_1\in [-T,T)$ and $\bZ_1\in\cI_\ell$, for some $\ell\in\{1,\ldots,N\}$, such that $\bZ(t_1)=\bZ_1$ and
\begin{equation}\label{hpuniqueness}
\widehat\bef^-(\bZ_1)\cdot\bn(\bZ_1)>0\qquad\text{or}\qquad\widehat\bef^+(\bZ_1)\cdot\bn(\bZ_1)<0,
\end{equation} 
where $\widehat\bef^\pm$ are the extensions of the functions $\bef^\pm$ defined in terms of the function $\bef$ given in \eqref{910} with $\hat\bZ=\bZ_1$.
Then right uniqueness holds for \eqref{f010} at the point $(t_1,\bZ_1)$.
\end{theorem}
\begin{proof}
By \eqref{hpuniqueness}, $\bZ_0\notin\cE_{\rm src}$, 
therefore, by the previous discussion, the result follows from Theorem \ref{thm:FilippovUn}.
%In view of the discussion above, the result follows from Theorem \ref{thm:FilippovUn}.
\end{proof}
\begin{remark}\label{rem:afteruniqueness}
Existence time is limited by the possibility of collisions between dislocations, that is,  
$|\bz_i-\bz_j|\to 0$, or between a dislocation and $\partial \Omega$, that is,  
$\dist (\bz_i,\partial \Omega) \to 0$. Additionally, uniqueness is limited by 
possible intersections of $\bZ(t)$ with $\cS_\ell\cup \mathcal{E}_{\rm int}\cup\cE_{\rm zero}\cup\cE_{\rm src}$.
The ambiguity set $\cA$ is smooth except possibly on the
singular sets $\cS_\ell$, which are at most $(2N-2)$-dimensional by Lemma \ref{lem:singularset}, 
or points in $\cE_{\rm int}$.
\end{remark}

\subsection{Cross-Slip and Fine Cross-Slip} \label{sec:slip}
%As discussed in \cite{Gurtin}, the equation \eqref{motion1} will be
%violated at some points. This is why we have turned to inclusion
%\eqref{f006}. 
We expect to see two kinds of motion
at points where the force is not single-valued. If a dislocation point
$\bz_\ell$ is moving in the direction $\bg_\ell^-$ and 
the configuration $\bZ = (\bz_1,\ldots,\bz_N)$ arrives
at a point on $\cA_\ell$ where $\bg_\ell^\pm$ are two glide directions
that are equally favorable to $\bz_\ell$, then $\bz_\ell$ could abruptly 
transition from motion along $\bg_\ell^-$ to motion along
$\bg_\ell^+$. 
Such a motion  is called \emph{cross-slip}  
(see Figure \ref{fig:crossslip}).
Heuristically, cross-slip occurs when,  on one side of $\cA_\ell$,
the vector field $F$ (see \eqref{f010}) is pointing toward $\cA_\ell$, while the other
side $F$ is pointing away from $\cA_\ell$.
If the configuration $\bZ$ is in the region where
$F$ points towards $\cA_\ell$, then $\bZ$ approaches $\cA_\ell$ and
arrives at it in a finite time. The configuration 
then leaves $\cA_\ell$, moving into the region
where $F$ points away from $\cA_\ell$.

\begin{figure}[ht]
\begin{center}
\centering{
\labellist
\hair 2pt
\pinlabel $\mathbf{z}_1$ at 77 30
\pinlabel $\mathbf{z}_2$ at 25 122
\pinlabel $\mathbf{z}_3$ at 150 150
\pinlabel $\Omega$ at 220 20
\pinlabel $\mathcal{G}$ at 185 35
\pinlabel $\mathbb{R}^{2N}$ at 520 20
\pinlabel $\mathbf{Z}$ at 390 28
\pinlabel $V^+$ at 340 68
\pinlabel $V^-$ at 340 38
\pinlabel $\cA_1$ at 515 60
\pinlabel $\mathrm{(a)}$ at 100 -5
\pinlabel $\mathrm{(b)}$ at 420 -5
	\endlabellist				
\includegraphics[scale=.6]{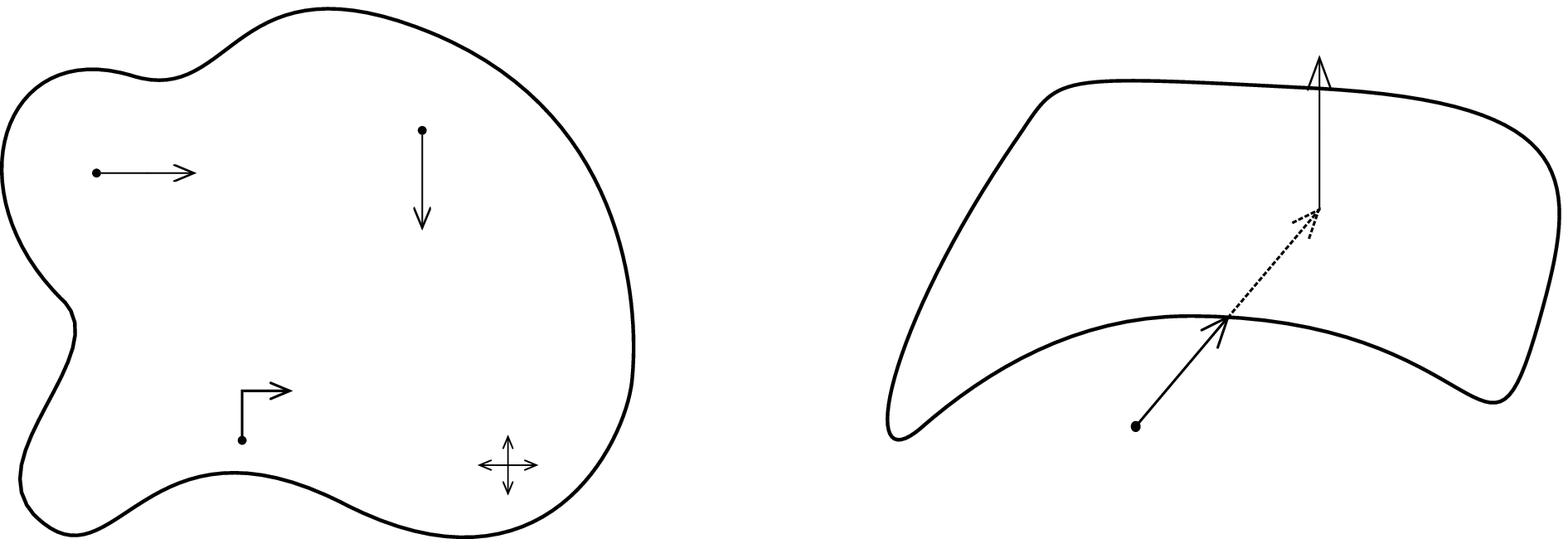}\\ %@fig5
}%\small    \vspace{-1cm} text 1}
\caption{Cross-slip. The glide directions are $\cG=\{\pm\be_1,\pm\be_2\}$, where $\be_i$ is the $i$-th basis vector. 
	In (a), dislocation $\bz_1\in\Omega$ is undergoing cross-slip, switching direction from $\bg_1^-=\be_2$ to $\bg_1^+=\be_1$, 
	while dislocations $\bz_2$ and $\bz_3$ glide normally along directions $\bg_2=\be_1$ and $\bg_3=-\be_2$, respectively.
	In (b) the same motion is represented in $\R{2N}$: the motion of $\bZ$ changes direction while crossing the surface $\cA_1$, 
	where the velocity field is multivalued. (Here, $N=3$.)}\label{fig:crossslip}
\end{center}
\end{figure}

Another possibility is that the vector field $F$ points towards
$\cA_\ell$ on both sides of $\cA_\ell$.
In this case, at a point on $\cA_\ell$, a motion by
$\bz_\ell$ in the 
$\bg_\ell^+$ direction will drive the configuration $\bZ$ 
to a region where $\bj_\ell$ is most closely aligned with
$\bg_\ell^-$, but then motion by $\bz_\ell$ 
along $\bg_\ell^-$ immediately forces $\bZ$ to
intersect the surface $\cA_\ell$ again. Motion by $\bz_\ell$ along $\bg_\ell^-$
then pushes $\bZ$ into a region where $\bj_\ell$ is most closely
aligned with $\bg_\ell^+$, which
forces $\bZ$ back to $\cA_\ell$. A motion such as  this one on a finer
and finer scale will appear as motion \emph{along} the surface
$\cA_\ell$. Following \cite{Gurtin}, 
such a motion is called \emph{fine cross-slip}. 
See Figure \ref{fig:finecrossslip}, where the dislocation $\bz_1$ is undergoing fine cross-slip. 
In part (a) it is shown how it follows a curve $l$ rather than one of
the glide directions $\bg\in\cG$.
In part (b) the same phenomenon is shown in $\R{2N}$ ($N=3$), where the point $\bZ$ hits $\cA_1$ and starts moving along it.
%In Figure \ref{fig:approxlimit}(a) the same motion is shown at a discretized time step. 
%Dislocation $\bz_1$ bounces back and forth across the line $l$, projection on $\Omega$ of $\cA_1$, along directions $\bg_1^-=\be_2$ and $\bg_1^+=\be_1$.
%As the discretization step vanishes, the motion becomes continuous along $l$.

\begin{figure}[ht]
\begin{center}
\centering{
\labellist
\hair 2pt
\pinlabel $\mathbf{z}_1$ at 77 30
\pinlabel $\mathbf{z}_2$ at 25 122
\pinlabel $\mathbf{z}_3$ at 150 150
\pinlabel $\Omega$ at 220 20
\pinlabel $\mathcal{G}$ at 185 35
\pinlabel $\mathbb{R}^{2N}$ at 520 20
\pinlabel $\mathbf{Z}$ at 340 15
\pinlabel $V^+$ at 300 68
\pinlabel $V^-$ at 300 18
\pinlabel $\cA_1$ at 520 90
\pinlabel $l$ at 110 60
\pinlabel $\mathrm{(a)}$ at 100 -5
\pinlabel $\mathrm{(b)}$ at 420 -5
	\endlabellist				
\includegraphics[scale=.6]{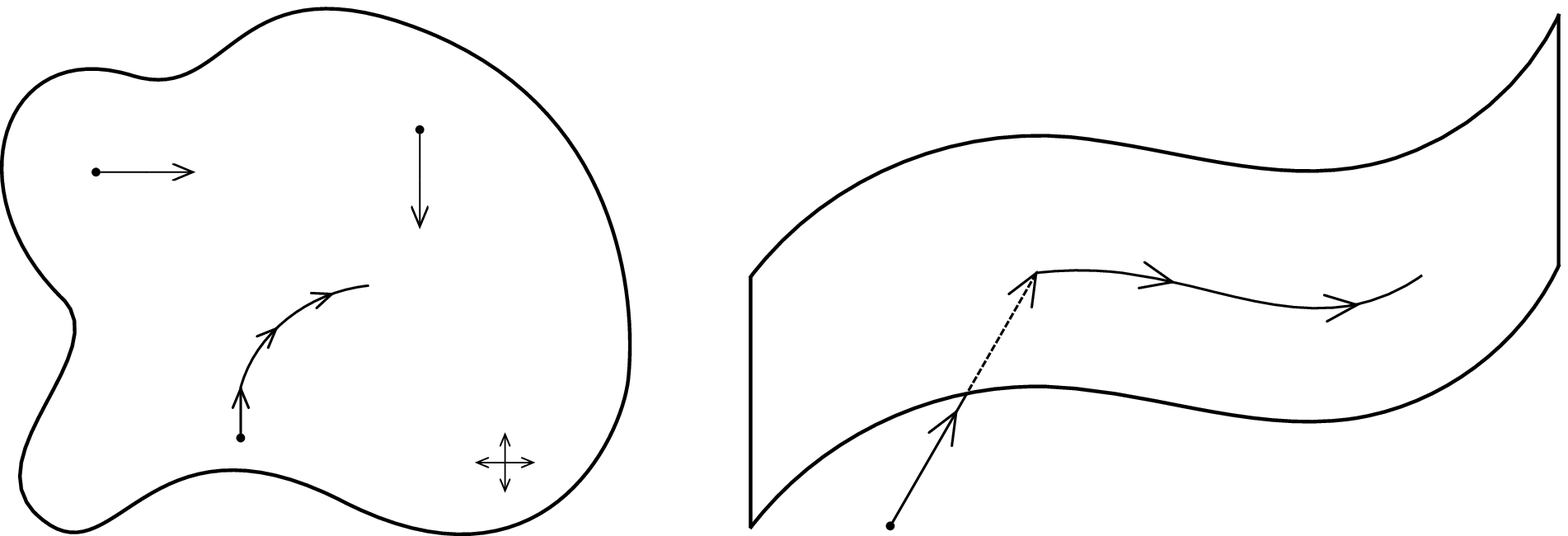}\\ %@fig5
}%\small    \vspace{-1cm} text 1}
\caption{Fine cross-slip. Let $\cG$ be the same as in Figure \ref{fig:crossslip}. %The glide directions are $\cG=\{\be_1,\be_2\}$. 
	In (a), dislocation $\bz_1\in\Omega$ is undergoing fine cross-slip, switching direction from $\bg_1^-=\be_2$ to a curved one which is not in $\cG$, 
	while dislocations $\bz_2$ and $\bz_3$ glide normally along directions $\bg_2=\be_1$ and $\be_3=-\be_2$, respectively.
	In (b) the same motion is represented in $\R{2N}$: the motion
        of $\bZ$, after hitting the surface $\cA_1$ continues on the
        surface following the tangent direction. (Here, $N=3$.)
%	The line $l$ in (a) is the projection of $\cA_1$ on $\Omega$.
	}\label{fig:finecrossslip}
\end{center}
\end{figure}

%\begin{figure}[ht]
%\begin{center}
%\centering{
%\labellist
%\hair 2pt
%\pinlabel $\mathbf{z}_1$ at 77 30
%\pinlabel $\mathbf{z}_2$ at 25 122
%\pinlabel $\mathbf{z}_3$ at 150 150
%\pinlabel $\Omega$ at 220 20
%\pinlabel $\mathcal{G}$ at 185 35
%\pinlabel $\mathbf{z}_1$ at 369 30
%\pinlabel $\mathbf{z}_2$ at 317 122
%\pinlabel $\mathbf{z}_3$ at 442 150
%\pinlabel $\Omega$ at 512 20
%\pinlabel $\mathcal{G}$ at 477 35
%\pinlabel $l$ at 115 100
%\pinlabel $l$ at 407 100
%\pinlabel $\mathrm{(a)}$ at 100 -5
%\pinlabel $\mathrm{(b)}$ at 392 -5
%	\endlabellist				
%\includegraphics[scale=.6]{figures/approxlimit}\\ %@fig5
%}%\small    \vspace{-1cm} text 1}
%\caption{Fine cross-slip in detail. Let $\cG$ be the same as in Figure \ref{fig:crossslip}. %The glide directions are $\cG=\{\be_1,\be_2\}$. 
%	In (a), we observe the motion of dislocation $\bz_1\in\Omega$ undergoing fine cross-slip at a finer, discretized time scale, 
%	switching direction from $\bg_2^-=\be_2$ to direction $\bg_2^+=\be_1$ at each time step.
%	%while dislocations $\bz_2$ and $\bz_3$ glide normally along directions $\be_1$ and $-\be_2$, respectively.
%	In (b) the same motion is represented when the discretization step goes to zero. 
%	A continuous motion along a curved line emerges, but the direction of this motion is not an admissible glide direction.
%	}\label{fig:approxlimit}
%\end{center}
%\end{figure}

%The behaviors described before are formalized in the following theorems.
The following theorems formalize the behaviors described above 
and provide an analytical validation of the notions of cross-slip and fine cross-slip introduced in \cite{Gurtin}.
We refer to the discussion preceding Theorem \ref{thm:uniqueness} for the definitions of $\bn(\bZ)$ and $V^\pm$ for $\bZ\in\cI_\ell$.
\begin{theorem}[Cross-Slip]\label{cs003}
	Let $T>0$ and let $\bZ:[-T,T]\to\mathbb{R}^{2N}$ be a solution to \eqref{f010}.
	Assume that there exist $t_1\in (-T,T)$ and $\bZ_1\in\cI_\ell$, for some $\ell\in\{1,\ldots,N\}$, such that $\bZ(t_1)=\bZ_1$,  
	\begin{equation}\label{hpcs1}
	\widehat\bef^-(\bZ_1)\cdot\bn(\bZ_1)>0,\qquad\text{and}\qquad\widehat\bef^+(\bZ_1)\cdot\bn(\bZ_1)>0,
	\end{equation}
	where $\bef$ is the function defined in \eqref{910}.
Then uniqueness holds for \eqref{f010} at the point $(t_1,\bZ_1)$ and  the  solution passes from $V^-$ to $V^+$. 
Similarly, if
\begin{equation}\label{hpcs2}
\widehat\bef^-(\bZ_1)\cdot\bn(\bZ_1)<0\qquad\text{and}\qquad\widehat\bef^+(\bZ_1)\cdot\bn(\bZ_1)<0,
\end{equation} 
then uniqueness holds for \eqref{f010} at the point $(t_1,\bZ_1)$ and the  solution passes from $V^+$ to $V^-$. 
\end{theorem}
\begin{proof}
Since $\widehat\bef^\pm$ are $C^1$ extensions of $\bef^\pm:=\bef\big|_{V^\pm}$, the result follows from Theorem \ref{cs001}.
\end{proof}

\begin{theorem}[Fine Cross-Slip]\label{cs004} 
		Let $T>0$ and let $\bZ:[-T,T]\to\mathbb{R}^{2N}$ be a solution to \eqref{f010}.
		Assume that there exist $t_1\in (-T,T)$ and $\bZ_1\in\cI_\ell$, for some $\ell\in\{1,\ldots,N\}$, such that $\bZ(t_1)=\bZ_1$,  
\begin{equation*}%\label{hpfcs}
\widehat\bef^-(\bZ_1)\cdot\bn(\bZ_1)>0\qquad\text{and}\qquad\widehat\bef^+(\bZ_1)\cdot\bn(\bZ_1)<0,
\end{equation*}
where $\bef$ is the function defined in \eqref{910}. 
Then right uniqueness holds for \eqref{f010} at the point $(t_1,\bZ_1)$ and there exists $\delta>0$  such that  $\bZ$ belongs to $\cA_\ell$ and solves the ordinary differential equation for all $t\in [t_1,t_1+\delta]$,
\begin{equation*}%\label{cs005}
\dot\bZ=\bef^0(\bZ)\in\ch F(\bZ),\quad\text{where}\quad 
\bef^0(\bZ):=\alpha(\bZ)\widehat\bef^+(\bZ)+(1-\alpha(\bZ))\widehat\bef^-(\bZ)
\end{equation*}
 and 
$\alpha(\bZ)\in(0,1)$ is defined by
\begin{equation*}%\label{cs006}
\alpha(\bZ)
:=\frac{\widehat\bef^-(\bZ)\cdot\bn(\bZ)}{\widehat\bef^-(\bZ)\cdot\bn(\bZ)-\widehat\bef^+(\bZ)\cdot\bn(\bZ)}.
\end{equation*}
\end{theorem}
\begin{proof}
The result follows from Corollary \ref{cor:f0}.
\end{proof}
Note that the cross-slip and fine cross-slip trajectories that we have described
in Theorems \ref{cs003} and \ref{cs004} satisfy the conditions for right uniqueness
in Theorem \ref{thm:uniqueness}.  Specifically,
if \eqref{hpcs1} or \eqref{hpcs2} holds, then \eqref{hpuniqueness} holds (i.e., $\bZ_1\notin\cE_{\rm src}$).

%\begin{corollary}[graze points \note{but not all of them}]\label{ca007}
%The dislocation $\bz_\ell$ will continue the frine cross-slip until either $\bef^-(\hat\bZ)\cdot\bn(\hat\bZ)=0$ or $\bef^+(\hat\bZ)\cdot\bn(\hat\bZ)=0$,
%when $\bef_\ell^*(\bZ)$ will be equal to $\bef^-(\hat\bZ)$ or $\bef^+(\hat\bZ)$, respectively. \note{say better}.
%\end{corollary}

%\input{special}

\section{More on Fine Cross-Slip}\label{section intersections}

In Subsection \ref{sec:setting} we have discussed uniqueness only in the special case 
in which $\bef$ is discontinuous across a $(d-1)$-dimensional hypersurface. The case when two or more such $(d-1)$-dimensional hypersurfaces meet is significantly more involved and  can lead to non-uniqueness of solutions
for Filippov systems (see, e.g., \cite{Dieci}).

In our setting, this situation arises at points in the set $\cE_{\rm int}$ defined in \eqref{intersections}. Indeed, 
in  Theorem \ref{cs004} we assumed that $\bZ_1$ does not belong to the intersection of two hypersurfaces 
(see \eqref{intersections} and \eqref{initial}). In this section we study fine cross-slip in the case in which  $\bZ_1$ belongs to $\mathcal{E}_{\rm int}$. % defined in \eqref{intersections}. 
For simplicity, we consider only the case in which only two hypersurfaces intersect at a point. 
See Figure \ref{fig:dfcs}.

\begin{figure}[ht]
\begin{center}
\centering{
\labellist
\hair 2pt
\pinlabel $\mathbf{z}_1$ at 65 30
\pinlabel $\mathbf{z}_2$ at 25 122
\pinlabel $\mathbf{z}_3$ at 200 65
\pinlabel $\Omega$ at 220 20
\pinlabel $\mathcal{G}$ at 185 35
\pinlabel $\mathbb{R}^{2N}$ at 520 20
\pinlabel $\mathbf{Z}$ at 340 15
%\pinlabel $V^+$ at 300 68
%\pinlabel $V^-$ at 300 18
\pinlabel $\cA_1$ at 535 110
\pinlabel $\cA_3$ at 450 149
%\pinlabel $l_1$ at 80 70
%\pinlabel $l_3$ at 165 100
\pinlabel $\mathrm{(a)}$ at 100 -5
\pinlabel $\mathrm{(b)}$ at 420 -5
	\endlabellist				
\includegraphics[scale=.6]{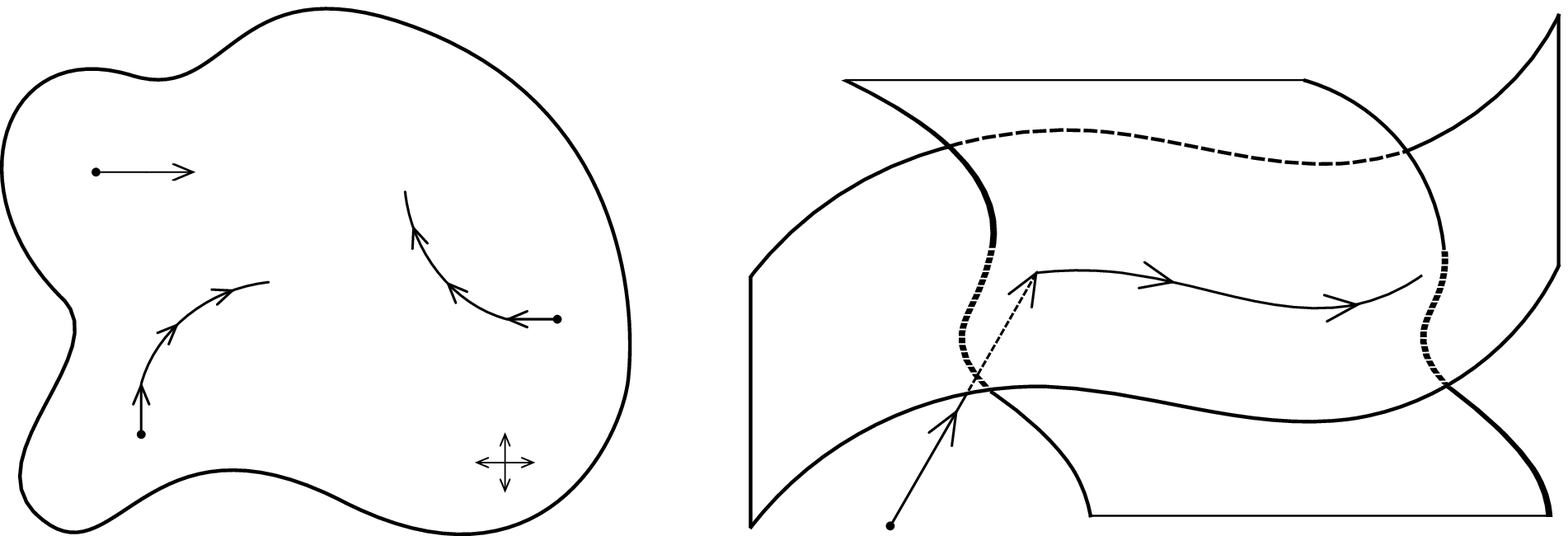}\\ %@fig5
}%\small    \vspace{-1cm} text 1}
\caption{Simultaneous fine cross-slip. Let $\cG$ be the same as in Figure \ref{fig:crossslip}. %The glide directions are $\cG=\{\be_1,\be_2\}$. 
	In (a), dislocations $\bz_1,\bz_3\in\Omega$ are undergoing fine cross-slip, switching directions from $\bg_1^-=\be_2$ and $\bg_3^-=-\be_1$, respectively, to curved ones $l_1,l_3$ which are not in $\cG$,
	while dislocation $\bz_2$ glides normally along direction $\bg_2=\be_1$.
	In (b) the same motion is represented in $\R{2N}$: the motion of $\bZ$, after hitting $\cA_1\cap\cA_3$, continues on the intersection of the two surfaces.
%	The line $l$ in (a) is the projection of $\cA_1$ on $\Omega$.
	}\label{fig:dfcs}
\end{center}
\end{figure}

Assume that there exists  $\bZ_1 \in \cA_\ell \cap \cA_k$ for $k\neq \ell$, with 
$\bZ_1 \notin \cA_i$ for $i\neq k,\ell$ and
$\bZ_1\notin \cE_{{\rm zero}}\cup \cS_\ell \cup \cS_k$.
Consider the case of fine cross-slip conditions along
both $\cA_\ell$ and $\cA_k$. Specifically, at $\bZ_1$, the 
vectors $\bj_\ell(\bZ_1)$ and $\bj_k(\bZ_1)$ are well-defined and
bisect two maximally dissipative glide directions $\bg_\ell^\pm$ and
$\bg_k^\pm$, respectively. By assumption, the other $\bj_i(\bZ_1)$ have uniquely defined
maximally dissipative glide directions. By Lemma \ref{lem:hulls}, the set-valued
vector field has the form
$  {\rm co}\, F (\bZ_1) = ( {\rm co}\,F_1(\bZ_1),\ldots, {\rm
  co}\,F_N(\bZ_1))$, 
where (see \eqref{f100}),
\begin{equation*}
 {\rm co}\,F_i(\bZ_1) = F_i(\bZ_1) =\{(\bj_i(\bZ_1)\cdot  \bg_i(\bZ_1)) \bg_i(\bZ_1)\}
\end{equation*}
for $i \neq k,\ell$ and
\begin{equation*}
 {\rm co}\,F_i(\bZ_1) = \{s_i (\bj_i(\bZ_1)\cdot  \bg_i^+(\bZ_1)) \bg_i^+(\bZ_1) +
(1-s_i)(\bj_i(\bZ_1)\cdot  \bg_i^-(\bZ_1)) \bg_i^-(\bZ_1),
s_i \in [0,1]\},
\end{equation*}
for $i=k,\ell$.
Additionally,  there is a
ball $B_h(\bZ_1)\subset\mathcal{D}(F)$ that is separated into two open sets $V_\ell^\pm$ by $\cA_\ell$,
such that, for   $\bZ\in V_\ell^+$, 
$F_\ell(\bZ) =  \{(\bj_\ell(\bZ_1)\cdot  \bg_\ell^+,(\bZ_1))\bg_\ell^+,(\bZ_1)\}$
and for   $\bZ\in V_\ell^-$, 
$F_\ell(\bZ) =  \{(\bj_\ell(\bZ_1)\cdot  \bg_\ell^-,(\bZ_1))\bg_\ell^-,(\bZ_1)\}$.
Similarly,  $B_h(\bZ_1)$ is separated into two open sets $V_k^\pm$ by $\cA_k$
where the corresponding equalities hold.
Since we are avoiding singular points, let $\bn_\ell(\bZ_1)$ and $\bn_k(\bZ_1)$ denote
the normals to $\cA_\ell$ and $\cA_k$ at $\bZ_1$, where $\bn_i(\bZ_1)$ points
from $V_i^-$ to $V_i^+$ for $i=k,\ell$.

Now, at the intersection of two surfaces, $B_h$ is divided into four regions, so there will
be four vector fields that will need to satisfy some projection conditions in order
for fine cross-slip to occur. 
For $i\neq k,\ell$, set $\bef_i(\bZ): = (\bj_i(\bZ)\cdot  \bg_i(\bZ)) \bg_i(\bZ)$
for $\bZ\in B_h(\bZ_1)$.
Set $\bef_k^\pm(\bZ) := (\bj_k(\bZ)\cdot  \bg_k^\pm(\bZ)) \bg_k^\pm(\bZ)$ for $\bZ\in V_k^\pm$
and $\bef_\ell^\pm(\bZ) := (\bj_\ell(\bZ)\cdot  \bg_\ell^\pm(\bZ)) \bg_\ell^\pm(\bZ)$
for $\bZ\in V_\ell^\pm$. By assumption, $\bef_k^\pm$ and $\bef_\ell^\pm$ can 
be extended in a $C^1$ way to $B_h(\bZ_1)$, we denote these extensions by $\hbf_k^\pm$ and $\hbf_\ell^\pm$.
Define the extended vector fields in $B_h(\bZ_1)$,
\begin{subequations}\label{f plus minus}
\begin{align}
  \bef^{(+,+)}(\bZ) &= (\bef_1(\bZ),\ldots,\hbf_k^+(\bZ),\ldots,\hbf_\ell^+(\bZ),\ldots,\bef_N(\bZ)), \\
  \bef^{(+,-)}(\bZ) &= (\bef_1(\bZ),\ldots,\hbf_k^+(\bZ),\ldots,\hbf_\ell^-(\bZ),\ldots,\bef_N(\bZ)),\\
  \bef^{(-,+)}(\bZ) &= (\bef_1(\bZ),\ldots,\hbf_k^-(\bZ),\ldots,\hbf_\ell^+(\bZ),\ldots,\bef_N(\bZ)),\\
  \bef^{(-,-)}(\bZ) &= (\bef_1(\bZ),\ldots,\hbf_k^-(\bZ),\ldots,\hbf_\ell^-(\bZ),\ldots,\bef_N(\bZ)).
\end{align}
\end{subequations}
The fine cross-slip conditions are that the surfaces
$\cA_k$ and $\cA_\ell$ are attracting at $\bZ_1$, so that
\begin{subequations}\label{conditions}
\begin{align}
&  \bef^{(+,+)}(\bZ_1)\cdot \bn_k(\bZ_1) <0,  \qquad \bef^{(+,+)}(\bZ_1)\cdot \bn_\ell(\bZ_1) <0,
\label{cond1} \\
& \bef^{(+,-)}(\bZ_1)\cdot \bn_k(\bZ_1) <0, \qquad \bef^{(+,-)}(\bZ_1)\cdot \bn_\ell(\bZ_1) > 0,  
\label{cond2}\\
&  \bef^{(-,+)}(\bZ_1) \cdot \bn_k(\bZ_1) >0, \qquad  \bef^{(-,+)}(\bZ_1) \cdot \bn_\ell(\bZ_1) <0, 
\label{cond3} \\
& \bef^{(-,-)}(\bZ_1) \cdot \bn_k(\bZ_1) >0, \qquad  \bef^{(-,-)}(\bZ_1) \cdot \bn_\ell(\bZ_1) >0.
\label{cond4}
\end{align}
\end{subequations}
By taking $h$ smaller, if necessary, we can assume that $\bZ \notin \cA_i$ for $i\neq k,\ell$,  that $\bZ\notin \cE_{{\rm zero}}\cup \cS_\ell \cup \cS_k$, and that
\eqref{cond1}-\eqref{cond4} continue to hold for all $\bZ\in B_h(\bZ_1)$. 
%and that   for all $\bZ\in B_h(\bZ_1)$.

We now show that the only possible motion is along the intersection $\cA_k\cap\cA_\ell$. 

\begin{theorem}\label{theorem intersections} Let $T>0$ and let $\bZ:[-T,T]\to\mathbb{R}^{2N}$ be a solution to \eqref{f010}.
Assume that there exist $t_1\in [-T,T)$ and $\bZ_1$ as above 
such that $\bZ(t_1)=\bZ_1$. 
Then there exists $\delta>0$  such that $\bZ$ is unique in $ [t_1,t_1+\delta]$ and $\bZ(t)$ belongs to $\cA_k\cap \cA_{\ell}$ for all $t\in [t_1,t_1+\delta]$. 
 \end{theorem}

\begin{proof} \noindent {\bf Step 1.} Since $\bZ(t_1)=\bZ_1$, by continuity we can find $t_2>t_1$ such that 
$\bZ(t)\in B_h(\bZ_1)$ for all $t\in [t_1,t_2]$. We claim that $\bZ(t)$ belongs to $\cA_k\cap \cA_{\ell}$ for all $t\in [t_1,t_2]$.
Indeed, suppose by contradiction that there exists $t_3\in [t_1,t_2]$ such that $\bZ$ leaves $\cA_k$, that is, 
 $\bZ (t_3) \in V_k^+$ (the case of $V_k^-$ is similar, as well as
the case of leaving $\cA_\ell$ and going into $V_\ell^\pm$). 
Define 
\[
 \tau_1:= \sup \{ s\in [t_1,t_3)\, :\, \bZ(s)\notin V_k^+\},
\] 
which is
the last time $\bZ$ was in $\cA_k$ before entering and remaining in $V_k^+$. 

\noindent {\bf Case 1.}  Suppose that $\bZ(\tau_1)\notin \cA_\ell$. Then $\bZ(\tau_1)$ belongs to either
$V_\ell^+$ or $V_\ell^-$. Without loss of generality,
we assume that $\bZ(\tau_1)\in V_\ell^+$. 
Since $\bZ(\tau_1)\in \cA_k$  by definition, and it does not belong to any
other  $\cA_i$, only the $k$-th component of the force is
double-valued at $\bZ(\tau_1)$. Thus, $\bZ(\tau_1)$ is a point
satisfying the hypotheses of Theorem \ref{cs004}
because $\bef^{(+, \pm)}(\bZ(\tau_1))\cdot \bn_k
(\bZ(\tau_1))<0$. 
Therefore there is $\delta >0$ such that
$\bZ(t)\in \cA_k$ for $t\in [\tau_1,\tau_1+\delta]$, which contradicts the
definition of $\tau_1$.

\noindent {\bf Case 2.} By Case 1,  $\bZ(\tau_1)\in \cA_\ell$. 
We claim that
\begin{equation}\label{int1}
\bZ(t)\in\cA_\ell\quad\text{ for all }t\in [\tau_1,t_3].
\end{equation}
If \eqref{int1} fails, then there is $t_4\in (\tau_1,t_3]$
such that  $\bZ(t_4)\notin \cA_\ell$, and so $\bZ(t_4)$ is in
$V_\ell^+\cup V_\ell^-$. Without loss of generality,
assume $\bZ(t_4)\in V_\ell^+$,
and define 
\[
\tau_2:= \sup \{ s\in [\tau_1,t_4]\, :\, \bZ(s)\notin V_\ell^+\}.
\]
which is
the last time $\bZ$ was in $\cA_\ell$. 
%Since $\bZ(\tau_1)\in \cA_\ell$ we know
%$\bZ(\tau_1)\notin V_\ell^+$, so it could be that
%$\tau_2 = \tau_1$. 
If $\tau_2 > \tau_1$, then $\bZ(\tau_2)\in \cA_\ell$
and $\bZ(\tau_2)\notin \cA_k$ because $\bZ(t) \in V_k^+$ on $[\tau_1,t_4]$. 
Hence $\bZ(\tau_2)$ is a point that satisfies the hypotheses of the
fine cross-slip theorem because 
$\bef^{(+,\pm)}(\bZ(\tau_2))\cdot \bn (\bZ(\tau_2))<0$,
and so there is $\delta >0$ such that $\bZ(t) \in \cA_\ell$ for 
$t\in [\tau_2,\tau_2+\delta]$. This contradicts the definition of $\tau_2$. 

Therefore $\tau_2 = \tau_1$,
$\bZ(\tau_2)\in \cA_k\cap\cA_\ell$, 
and $\bZ(t) \in V_k^+\cap V_\ell^+$ for
$t\in(\tau_2,t_1]$.
We deduce that $\bZ$ satisfies $\dot \bZ = \bef^{(+,+)}(\bZ)$
on $(\tau_2,t_3]$, thus
% and $\dot \bZ = \hbf^{(+,+)}(\bZ)$
% on $[s_2,t_1]$ for the extension onto 
% $\cA_k\cap\cA_\ell$.
\begin{equation}
  \label{Zpp}
  \bZ(t) = \bZ(\tau_2) + \int_{\tau_2}^t \bef^{(+,+)}(\bZ(s))\, \de s,
\qquad t\in [\tau_2,t_4].
\end{equation}
Applying the argument from the proof of Corollary \ref{cor:f0}, we can
reach a contradiction as follows. Locally $\cA_k$ is given by the
graph of a function, so 
without loss of generality we can write 
$\cA_k \cap B_r(\bZ(\tau_2)) = \{ \bZ = (\bxi,y)\in B_r(\bZ(\tau_2)) \, :\, y =
\Phi(\bxi)\}$ for a function 
$\Phi$ of class $C^2$. 
Denote $\bZ(\tau_2)$ as $( \bxi_0,y_0) = \bZ(\tau_2)$.
Without loss of generality, we can assume that
$\nabla \Phi(\bxi_0)=\bzero$ so $\bn_k(\bZ(\tau_2)) = (\bzero,1)$ and 
\begin{align*}
V_k^+\cap B_r(\bZ(\tau_2)) = \{(\bxi,y)\in B_r(\bZ(\tau_2)) \, :\, y >
\Phi(\bxi)\},\\
V_k^-\cap B_r(\bZ(\tau_2)) = \{(\bxi,y)\in B_r(\bZ(\tau_2)) \, :\, y <
\Phi(\bxi)\}.
\end{align*}
From \eqref{cond1}, which holds in $B_h(\hat \bZ)$,
we have the same condition as \eqref{cond1} at the point
$\bZ(\tau_2)\in B_h(\hat \bZ)$. 
Set $h: = -\bef^{(+,+)}(\bZ(\tau_2))\cdot \bn(\bZ(\tau_2))
>0$, and find a neighborhood $V$ of $\bZ(\tau_2)$ such that 
$ -\bef^{(+,+)}(\bZ)\cdot \bn(\widetilde \bZ)>\frac12 h$
for $\bZ\in V$ and $\widetilde \bZ \in V\cap \cA_k$. From \eqref{Zpp}
we have
\begin{align*}
  \bZ(t) \cdot \bn(\bZ(\tau_2))  &= \bZ(\tau_2) \cdot \bn(\bZ(\tau_2)) 
 + \int_{\tau_2}^t \bef^{(+,+)}(\bZ(s)) \cdot \bn(\bZ(\tau_2)) \, \de s\\
&<\bZ(\tau_2) \cdot \bn(\bZ(\tau_2))
-\frac{t-\tau_2}{2}h.
\end{align*}
Using $\bn(\bZ(\tau_2)) = (\bzero,1)$ 
and writing $\bZ(t) = (\bxi(t),y(t))$,
we obtain
\begin{equation}
  \label{Zpp2}
  y(t) < y_0 -\frac{t-\tau_2}{2}h.
\end{equation}
But $\Phi(\bxi(t)) = \Phi(\bxi(\tau_2)) + \bzero + o(t-\tau_2) = 
 \Phi(\bxi_0)  + o(t-\tau_2) = y_0 + o(t-\tau_2)$. So \eqref{Zpp2} becomes
\[
 y(t) < y_0 -\frac{t-\tau_2}{2}h = \Phi(\bxi(t))
-\frac{t-\tau_2}{2}h +o(t-\tau_2) <\Phi(\bxi(t))
\]
for $0<t-\tau_2 < \delta$ for some $\delta >0$. 
This implies that $\bZ(t) \in V_k^-$ for $t\in (\tau_2,\tau_2+\delta]$,
which contradicts the fact that
$\bZ(t)\in V_k^+$, for $t\in (\tau_2,t_4]$. 

Thus, we have shown that \eqref{int1} holds. Since $\bZ(t_3)\in V_k^+$
by the definition of $\tau_1$,
% and $\tau_1$ is the last time that $\bZ$ is in $\cA_k$, 
%without loss of generality, we can assume that 
$\bZ(t)\in V_k^+$ for all $t\in (\tau_1,t_3]$. 
This, together with \eqref{int1} and Theorem \ref{cs004}, implies that %$\bZ$ satisfies
\[
\dot \bZ(t) = \bef^{(+,0)}(\bZ(t))
  = \alpha (\bZ(t)) \bef^{(+,+)}(\bZ(t)) +
(1-\alpha (\bZ(t)))\bef^{(+,-)}(\bZ(t))
\]
for $t\in  (\tau_1,t_3]$, where 
\begin{equation}\label{3.5*}
\alpha(\bZ(t))=\frac{\bef^{(+,-)}(\bZ(t))\cdot\bn_\ell(\bZ(t))}{\bef^{(+,-)}(\bZ(t))\cdot\bn_\ell(\bZ(t))-\bef^{(+,+)}(\bZ(t))\cdot\bn_\ell(\bZ(t))}.
\end{equation}
Using the same argument 
with $\Phi$ as above (starting from \eqref{Zpp}) and the fact that
$\bef^{(+,0)}(\bZ(t))\cdot \bn_k(\bZ(t))<0$, 
we conclude that $\bZ(t) \in V_k^-$, yielding a contradiction.
 This shows that $t_3$  cannot exist, and, in turn, that
$\bZ(t)\in \cA_k\cap\cA_\ell$ for all $t\in [t_1,t_2]$.

\noindent {\bf Step 2.}  In view of the previous step, we have that $\bZ(t)\in \cA_k\cap\cA_\ell$ for all $t\in [t_1,t_2]$.
In turn, 
$$ \dot \bZ(t)\cdot \bn_k(\bZ(t))=0\quad\text{and}\quad \dot \bZ(t)\cdot \bn_{\ell}(\bZ(t))=0
$$
for $\mathcal L^1$-a.e.\@ $t\in [t_1,t_2]$. Moreover, $\dot \bZ(t)\in {\rm co} F(\bZ(t))$ for $\mathcal L^1$-a.e.\@ $t\in [t_1,t_2]$.
Finally, since $\bZ(t)\in B_h(\bZ_1)$ for all $t\in [t_1,t_2]$, we have that \eqref{cond1}-\eqref{cond4} hold 
with $\bZ(t)$ in place of $\bZ_1$ for all $t\in [t_1,t_2]$ and $\bZ(t) \notin \cA_i$ for $i\neq k,\ell$ and $\bZ(t)\notin \cE_{{\rm zero}}\cup \cS_\ell \cup \cS_k$ for all $t\in [t_1,t_2]$. 
Hence, we can apply  Lemma \ref{lemma unique} in the appendix with $\bZ(t)$ in place of $\bZ_1$ to conclude that $\dot \bZ(t)$ is uniquely determined for $\mathcal L^1$-a.e. $t\in [t_1,t_2]$.
This concludes the proof.
\end{proof}

\begin{remark}
The argument in Step 1 does not rely on the fact that only two surfaces are intersecting.
Any number of surfaces would be treated the same way, but with more subcases for
showing the motion does not leave the intersection. However,
establishing uniqueness would require a different argument from the one in
Lemma \ref{lemma unique}.
\end{remark}

% Either $\bZ(t)\in \cA_\ell$ for
% all $t\in [s_1,t_1]$, or there is some $r_1\in (s_1,t_1]$ such that 
% $\bZ(r_1)\notin \cA_\ell$. Without loss of generality, assume that
% $\bZ(r_1) \in V_\ell^+$.

% Then $\bZ(s_2)\in \cA_\ell$ and $\bZ(s_2)\notin \cA_k$
% because $s_2\in (s_1,r_2]\subset (s_1,t_1]$ and $\bZ(t) \notin \cA_k$ 
% for $t\in (s_1,t_1]$.

\subsection{Identification of $\cA_\ell$ with a curve in $\Omega$}

Each dislocation point $\bz_\ell$ moves in $\Omega\subset \R{2}$ according to 
$\dot \bz_\ell = (\bj_\ell (\bZ) \cdot \bg_\ell (\bZ) ) \bg_\ell(\bZ)$,
but the dynamics is
understood in the larger space $\Omega^N \subset \R{2N}$. If
$\bz_\ell$ is exhibiting fine cross-slip,  then $\bz_\ell$ 
moves along a curve that is not a straight line parallel to a glide
direction.
In this section, we describe the fine
cross-slip motion of $\bz_\ell$ in $\Omega$ in terms of the dynamics
of the system in $\Omega^N$. That is, we will examine
fine cross-slip for $\bz_\ell$, which
 occurs when the
solution curve $\bZ(t) \in \Omega^{N}$ lies inside the set
 $\cA_\ell$, via a projection
 into $\Omega$.

The projection $\bz_\ell(t)$ of $\bZ(t)$ onto its $\ell$-th components is the fine cross-slip curve in
$\Omega$, with 
%That is, if $\bZ(t)\in \cA_\ell$ for $t\in [t_0,t_1]$, the fine cross-slip
%curve for $\bz_\ell$ in $\Omega$ is simply the curve 
$\bz_\ell(t) = (z_{\ell,1}(t),\,z_{\ell,2}(t))$ for $t\in [t_0,t_1]$.

%To see how this curve is related to $\cA_\ell$,  r
Recall that $\cA_\ell$ is locally given by the zero-level set of the
function $\bj_\ell \cdot \bg_0$. Specifically, if $\bZ_0
=(\bz_{0,1},\ldots,\bz_{0,N})\in \cA_\ell$,
then there exists $r>0$ such that 
\begin{equation}
  \label{Aell}
  \cA_\ell \cap B_r(\bZ_0) = \{\bZ\in \Omega^{N}\, :\,  \bj_\ell (\bZ)
  \cdot \bg_\ell^0(\bZ) = 0\},
\end{equation}
where $\bg_\ell^0(\bZ) = \bg_\ell^+ - \bg_\ell^-$ is constant in
$B_r(\bZ_0)$. Additionally, the normal to $\cA_\ell$ is given 
(up to a sign) by
\begin{equation}
  \label{normal}
  \bn : = \frac{\nabla \left(\bj_\ell (\bZ)
  \cdot \bg_\ell^0(\bZ)\right)}{\left|\nabla \left(\bj_\ell (\bZ)
  \cdot \bg_\ell^0(\bZ)\right)\right|}
  \in \R{2N},
\end{equation}
which is assumed to be non-zero in $\cA_\ell\cap B_r(\bZ_0)$. 
We write
$\bn  = (\bn _1,\ldots,\bn _N)$, with $\bn _i\in\R2$, for $i=1,\ldots,N$.

Assuming that
no other dislocations exhibit fine cross-slip, the fine cross-slip
conditions at
 $\bZ_0\in \cA_\ell$ are (with the appropriate sign for $\bn$)
\begin{align*}
%  \label{fcs}
&  \bn  \cdot \Big(
(\bj_1(\bZ_0)\cdot \bg_1)\bg_1,\ldots,
(\bj_\ell(\bZ_0)\cdot \bg^+_\ell)\bg^+_\ell,\ldots,
(\bj_N(\bZ_0)\cdot \bg_N)\bg_N\Big) <0,\\
&  \bn  \cdot \Big(
(\bj_1(\bZ_0)\cdot \bg_1)\bg_1,\ldots,
(\bj_\ell(\bZ_0)\cdot \bg^-_\ell)\bg^-_\ell,\ldots,
(\bj_N(\bZ_0)\cdot \bg_N)\bg_N\Big) >0.
\end{align*}
Note that we dropped the explicit dependence of each $\bg_i$ on $\bZ$
because they are constant in $B_r(\bZ_0)$.
Thus,
since $(\bj_\ell(\bZ_0)\cdot \bg^+_\ell)=(\bj_\ell(\bZ_0)\cdot \bg^-_\ell)$,
\begin{equation*}
%  \label{component}
   0 > \bn  \cdot \Big(
{\bf 0},\ldots,
(\bj_\ell(\bZ_0)\cdot \bg^\pm_\ell)\bg^0_\ell,\ldots,
{\bf 0}\Big) =
(\bj_\ell(\bZ_0)\cdot \bg^\pm_\ell)\bn _\ell \cdot \bg^0_\ell.
\end{equation*}
This implies $\bn _\ell \neq {\bf 0}\in\R2$, i.e., %Without loss of
%generality, let us assume that the first component is non-zero,
% $\bn^{(\ell)}_{\ell,1} \neq 0$. 
% Additionally, to avoid cumbersome
%  notation let us assume that $\ell = 1$. 
by \eqref{normal}, we have
\begin{equation}
  \label{grad}
\frac{\partial}{\partial z_{\ell,1}} \left( \bj_\ell
  (\bz_1,\ldots,\bz_N) \cdot \bg_\ell^0  (\bz_1,\ldots,\bz_N) \right) \neq 0.
\end{equation}
Let us write $\check \bZ$ for points in $\R{2N-1}$ 
of the form
$\check \bZ
:=(\bz_1,\ldots,\bz_{\ell-1},z_{\ell,2},\bz_{\ell+1},\ldots,\bz_N)$,
where the $z_{\ell,1}$ component is omitted. 

From \eqref{Aell} and \eqref{grad}, the Implicit Function Theorem yields
$r_1>0$, $r_2\in (0,r)$,
and a function
$\varphi :B_{r_1}(\check \bZ_0)\subset\R{2N-1} \to \R{}$, where  
$\check \bZ_0 :=(\bz_{0,1},\ldots, z_{0,\ell,2}, \ldots,\bz_{0,N})$, such that
$\varphi(\check \bZ_0) = z_{0,\ell,1}$ and
\begin{equation*}
%  \label{Aell2}
  \cA_\ell \cap B_{r_2}(\bZ_0) = \{\bZ\in\Omega^{N}\, :\,
z_{\ell,1} = \varphi (\check \bZ)\}. 
\end{equation*}
That is, locally, $\cA_\ell$ is the graph of $\varphi$. %We understand how
%$\cA_\ell$ projects into $\Omega$ via  $\varphi$.
If $\bZ(t)$ is a solution curve lying in $\cA_\ell\cap B_{r_2}(\bZ_0)$
for $t\in[t_0,t_1]$ with $\bZ(t_0)=\bZ_0$, then
\begin{equation*}
%  \label{Zt}
  \bZ(t) = (\bz_1(t),\ldots,\varphi(\check \bZ(t)),z_{\ell,2}(t),\ldots,\bz_N(t))\in\cA_\ell
\end{equation*}
for $t\in[t_0,t_1]$. In particular, the projection of $\bZ(t)$ onto its $\ell$-th
components gives the fine cross-slip curve
\begin{equation}
  \label{curve}
\bz_\ell(t) = (z_{\ell,1}(t),z_{\ell,2}(t)) =
(\varphi(\check \bZ(t)),z_{\ell,2}(t))  ,\quad
t\in[t_0,t_1].
\end{equation}
% \subsection{Projection of $\bn^{(\ell)}$}
% Note that by \eqref{component}, the projection of $\bn^{(\ell)} (\bZ(t))$ onto
%  $\bn^{(\ell)}_\ell (\bZ(t))$ defines a curve in $\Omega$ by
%  considering the curve whose tangent is a $\frac{\pi}{2}$-rotation of 
%  $\bn^{(\ell)}_\ell (\bZ(t))$.
% Thus, whether or not $\bz_\ell$ exhibits fine cross-slip depends on
% the entire configuration, $\bZ = (\bz_1,\ldots,\bz_N)$

Note that $\bn _\ell (\bZ(t))$ is not directly related to the
fine cross-slip curve given by \eqref{curve} because 
$\bn _\ell (\bZ(t))$ is not orthogonal to $\dot \bz_\ell(t)$,
in general. We have
\[
0=\bn  (\bZ(t)) \cdot \dot\bZ(t) = \sum_{i=1}^N \bn _i
(\bZ(t))\cdot \dot\bz_i(t),
\]
 so
\begin{equation*}
  \label{not orthogonal}
  \bn _\ell (\bZ(t))\cdot  \dot\bz_\ell(t) = 
-\sum_{i\neq \ell} \bn _i
(\bZ(t))\cdot\dot \bz_i(t),
\end{equation*}
and the sum on the right-hand side need not be zero.

% Then, for $t\in [s_1,t_1]$,
% $F_i(\bZ(t))=\{f_i(\bZ(t))\}$ for $i\neq \ell$ (since the field is
% single-valued in those components) and either
% $F_\ell(\^bZ(t)) = \{f_\ell(\bZ(t)) \}$ if $\bZ(t)\notin \cA_\ell$ or else
% $F_\ell(\bZ(t)) = \{f_\ell^-(\bZ(t)),\, f^+_\ell(\bZ(t)) \}$.

%$\dot \bZ \in {\rm  \co}\, F(\bZ)$

%--%--%--%

\subsection{Numerical Simulations}\label{numerics}
The simulation of \eqref{f010} may be undertaken using standard numerical ODE
integrators, provided sufficient care is taken in resolving the 
evolution near the ``ambiguity surfaces'' $\cA_\ell$.
A discrete time step leads
to a numerical integration that oscillates back and forth across an
attracting ambiguity surface in case of fine cross-slip. 
On the macro-scale, this appears as fine cross-slip since the
small oscillations across the surface average out and what remains is
motion approximately tangent to $\cA_\ell$. To compute the
vector field, one must solve the Neumann problem \eqref{u0Neumann} at
each time step, so a fast elliptic PDE solver is needed in practice. 

An example is
shown in Figures \ref{fig:twelve1} and \ref{fig:detail}, where we have simulated
%In this section we show some numerical simulation of
a system of
$N=12$ 
screw dislocations 
with each Burgers modulus $b_i=1$ for $i=1,\ldots,12$, 
and where the domain is
the unit disk. 
The integration is done in $\Omega^{12}\subset\R{24}$, 
but the graphics depict the path
each $\bz_i$ takes in $\Omega\subset \R{2}$.
All but one dislocation exhibit normal
glide motions, while the dislocation at the center exhibits fine
cross-slip, as is visible in Figure \ref{fig:detail}.
 In this case, the solution to the Neumann problem is
explicit (cf. \eqref{pkdisk}), so it is not difficult to simulate
systems with more dislocations and observe more complicate behavior, such as
multiple dislocations simultaneously exhibiting fine cross-slip,
corresponding to motion along the intersection of multiple ambiguity
surfaces in the full space $\Omega^N$. 
The simulation depicted in Figures \ref{fig:twelve1} and
\ref{fig:detail} was run until a dislocation collided with the
boundary. Since all dislocations have positive Burgers moduli, they
repel each other, and no collision between dislocations occurs, and the
dynamics can be continued until a boundary collision. 

\begin{figure}[ht]
\begin{center}
\centering{
\labellist
\hair 2pt
\pinlabel $\mathbf{z}_1$ at 115 122
\pinlabel $\searrow$ at 117 113
\pinlabel $\mathbf{z}_2$ at 65 132
\pinlabel $\mathbf{z}_3$ at 85 112
\pinlabel $\mathbf{z}_4$ at 65 55
\pinlabel $\mathbf{z}_5$ at 95 56
\pinlabel $\mathbf{z}_6$ at 130 53
\pinlabel $\mathbf{z}_7$ at 150 30
\pinlabel $\mathbf{z}_8$ at 200 84
\pinlabel $\mathbf{z}_9$ at 168 108
\pinlabel $\mathbf{z}_{10}$ at 165 125
\pinlabel $\mathbf{z}_{11}$ at 185 157
\pinlabel $\mathbf{z}_{12}$ at 137 160
%\pinlabel $\Omega$ at 220 20
\pinlabel $\mathcal{G}$ at 237 40
%\pinlabel $\mathbb{R}^{2N}$ at 520 20
%\pinlabel $\mathbf{Z}$ at 390 28
%\pinlabel $V^+$ at 340 68
%\pinlabel $V^-$ at 340 38
%\pinlabel $\cA_1$ at 515 60
%\pinlabel $\mathrm{(a)}$ at 100 -5
%\pinlabel $\mathrm{(b)}$ at 420 -5
	\endlabellist				
\includegraphics[scale=.7]{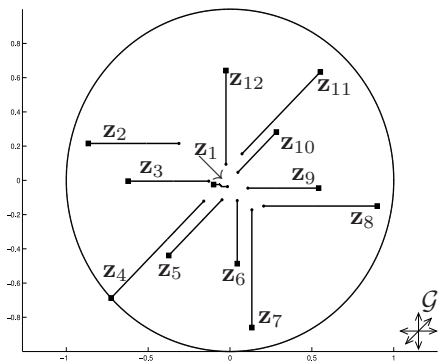}
}%\small    \vspace{-1cm} text 1}
\caption{The forces are repulsive and the dislocations move mostly along the glide directions $\cG=\{\pm\be_1,\pm\be_2,\pm\tfrac1{\sqrt2}(\be_1+\be_2)\}$. All but one (the one at the 	center) move along a glide direction until one of them hits the boundary. The dislocation in the middle moves along $-\be_1$ but then exhibits fine cross-slip.}
\label{fig:twelve1}
\end{center}
\end{figure}

\begin{figure}[ht]
\begin{center}
\centering{
\labellist
\hair 2pt
\pinlabel $\mathbf{z}_1(0)$ at 165 80
\pinlabel $\mathbf{z}_1(T)$ at 50 95
%\pinlabel $\mathbf{z}_3$ at 150 150
%\pinlabel $\Omega$ at 220 20
\pinlabel $\mathcal{G}$ at 215 40
%\pinlabel $\mathbb{R}^{2N}$ at 520 20
%\pinlabel $\mathbf{Z}$ at 390 28
%\pinlabel $V^+$ at 340 68
%\pinlabel $V^-$ at 340 38
%\pinlabel $\cA_1$ at 515 60
%\pinlabel $\mathrm{(a)}$ at 100 -5
%\pinlabel $\mathrm{(b)}$ at 420 -5
	\endlabellist				
\includegraphics[scale=.7]{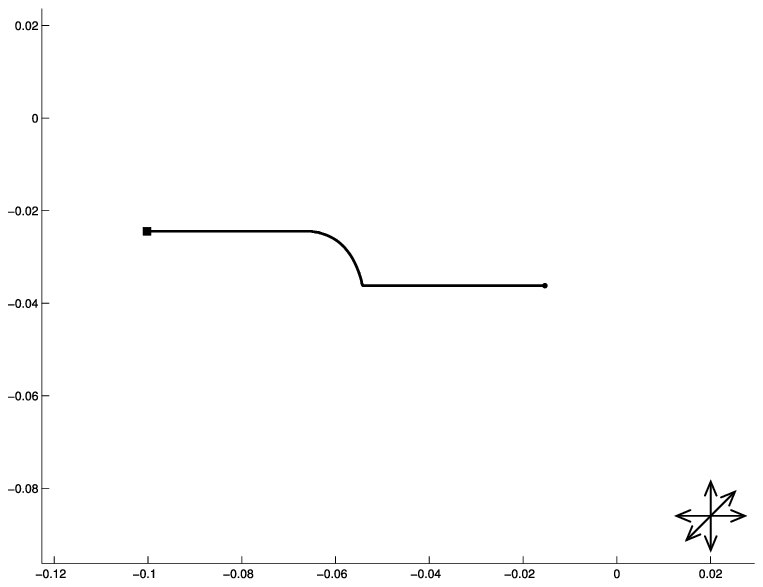}\quad
\labellist
\pinlabel $\mathbf{z}_1(0)$ at 125 77
\pinlabel $\mathbf{z}_1(T)$ at 70 85
\pinlabel $\mathbf{z}_3(0)$ at 48 120
\pinlabel $\mathbf{z}_4(0)$ at 28 40
\pinlabel $\mathbf{z}_5(0)$ at 101 47
\pinlabel $\mathbf{z}_6(0)$ at 161 43
\pinlabel $\mathbf{z}_9(0)$ at 195 90
\pinlabel $\mathbf{z}_{10}(0)$ at 164 133
\pinlabel $\mathcal{G}$ at 205 40
\endlabellist
\includegraphics[scale=.7]{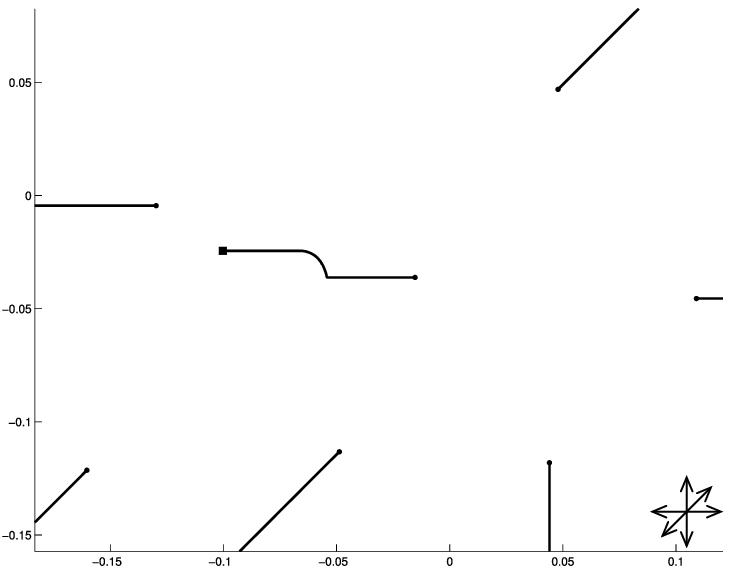}\\ %@fig5
}%\small    \vspace{-1cm} text 1}
\caption{These plots are magnified views of the motion of $\bz_1$. The motion begins at the dot on the right and ends at the square on the left. The motion abruptly begins to fine cross-slip 	and eventually moves back to a gliding motion as the fine cross-slip motion becomes aligned with $-\be_1$.}
\label{fig:detail}
\end{center}
\end{figure}

\section{Special Cases }\label{sec:special}
In this section we consider some special domains $\Omega$ for which
the Peach-K\"ohler force can be explicitly determined
(i.e. the solution to the Neumann problem \eqref{u0Neumann} is known), specifically
the unit disk $B_1$, the half-plane $\R{2}_+ $, and the plane $\R2$.
The last two cases do not technically fit in our previous
discussion, because $\Omega$ is unbounded. However, 
the Neumann problem is well-defined for these settings 
and we are able to discuss the dislocation dynamics.

In what follows we will use the fact that the 
boundary-response strains generated from each dislocation
are ``decoupled'' in the following sense.
Define $u_0^i$ as
\begin{equation*}
%\label{u0i}
u_0^i (\bx;\bz_i) := \int_{\partial \Omega}G(\bx,\by)
\bL \bk_i (\by;\bz_i)\cdot \bn(\by)\, \de s(\by),
\end{equation*}
where $G$ is the Green's function for the Neumann problem.
Then $u_0^i(\cdot;\bz_i)$ solves \eqref{u0Neumann} with only one dislocation, i.e.,
%That is,
\begin{equation*}
%  \label{u0iNeumann}
\left\{  \begin{array}{ll}
{\rm div}_{\bx}\left( {\bf L}\nabla_{\bx} u_0^i (\bx;\bz_i)\right) = 0,&  \bx\in\Omega, \\
{\bf L}\left(\nabla_{\bx} u_0^i (\bx;\bz_i)+ {\bf k}_i (\bx;\bz_i) \right)\cdot \bn(\bx)\ = 0,&
\bx\in \partial \Omega .
\end{array}\right.
\end{equation*}
Thus the boundary-response strain at $\bx$ due to a dislocation 
at $\bz_i$ with Burgers modulus $b_i$ is given by 
$\nabla_\bx u_0^i(\bx;\bz_i)$, % Hence, $u_0$ is just the sum
%of the $u_0^i$, 
and the total boundary-response strain at 
$\bx$ due to the system $\cZ$ is 
$\nabla_\bx u_0 (\bx;\bz_1,\ldots,\bz_N) = 
\sum_{i=1}^N \nabla_\bx u_0^i(\bx;\bz_i)$.

If we consider two dislocations 
$\bz_1$ and $\bz_2$ with Burgers moduli $b_1$ and $b_2$, respectively, that collide in $\Omega$,
then by \eqref{k_def} the boundary data in \eqref{u0Neumann} satisfies
\[
\bL(\bk_1(\bx;\bz_1)+\bk_2(\bx;\bz_2))\cdot\bn(\bx)\to \bL\Big(\bk_1(\bx;\bz_1)+
\frac{b_2}{b_1}\bk_1(\bx;\bz_1)\Big)\cdot\bn(\bx),\quad\text{as $\bz_2\to\bz_1$}.
\]
Notice that $\bk_1(\cdot;\bz_1)+(b_2/b_1)\bk_1(\cdot;\bz_1)$ is the 
singular strain generated by a single dislocation located at $\bz_1$ 
with Burgers modulus $b_1+b_2$.
The same argument applies to an arbitrary number $N$ 
of dislocation by linearity of \eqref{u0Neumann}.
Thus, unlike the singular strain which becomes infinite if any two 
dislocations collide in $\Omega$ (see \eqref{k_def}), 
the boundary-response strain is oblivious 
to collisions between dislocations. 
Although the boundary-response
strain is well-defined  when dislocations collide with each other,
%as mentioned in Remark \ref{rem:int2}, 
it is not well-defined if
a dislocation collides with $\partial \Omega$.

\subsection{The Unit Disk}\label{sec:undisk}
Consider the case $\Omega = B_1 = \{\bx\in\R2\, : \, |\bx|<1\}$ and $\lambda=\mu =1$, so that $\bL=\bI$.
For  $\bz \in B_1$ we define $\bbz\in B_1^c$ to be the 
reflection of $\bz$ across the unit circle $\partial B_1$,
\begin{equation*}
%  \label{zbar}
  \bbz: = \begin{cases}
  \displaystyle\frac{\bz}{|\bz|^2} & \text{if $\bz\in B_1\setminus\{\bzero\}$,} \\
  \infty & \text{if $\bz=0$.}
  \end{cases}
\end{equation*}

For fixed $\bz_i \in B_1$,
 it can be seen  that the function
\begin{equation}
  \label{ui0disk}
   u^i_0(\bx;\bz_i): =\begin{cases}
   \displaystyle -\frac{b_i}{\pi}
\arctan \left(\frac{x_2-\bzz_{i,2}}{x_1-\bzz_{i,1} + |\bx-\bbz_i|\,}\right) & \text{if $\bz\neq0$,} \\
0 & \text{if $\bz=0$}
\end{cases}
\end{equation}
satisfies
\begin{equation*}
%  \label{NeumannDisk}
  \left\{
    \begin{array}{ll}
      \Delta_\bx u_0^i(\bx;\bz_i) = 0,& \bx\in B_1, \\
      \nabla_\bx u_0^i (\bx;\bz_i)\cdot \bn(\bx) = -\bk_i(\bx;\bz_i)\cdot \bn(\bx),
      & \bx\in \partial B_1,
    \end{array}
\right.
\end{equation*}
and
\begin{equation}
  \label{gradu0disk}
  \nabla_\bx u_0^i (\bx;\bz_i) = -\bk_i(\bx;\bbz_i)\quad\text{for all $\bx\in B_1$}.
\end{equation}
Note that $\nabla u_0^i$
is singular only at the point $\bx = \bbz_i\notin B_1$. 

As discussed at the beginning of Section \ref{sec:special}, for a system of dislocations
given by $\cZ$ and $\cB$, the solution to 
the Neumann problem \eqref{u0Neumann}
is given by
\begin{equation*}
%  \label{u0disk}
  u_0(\bx;\bz_1,\ldots,\bz_N) = \sum_{i=1}^Nu_0^i(\bx;\bz_i)
\end{equation*}
with $u_0^i$ as in \eqref{ui0disk}. Thus, combining \eqref{PKj1} and
\eqref{gradu0disk}, we have 
\begin{equation}
  \label{pkdisk}
  \bj_\ell(\bz_1,\ldots,\bz_N) =b_\ell\bJ \Bigg(
\sum_{i\neq\ell}\bk_i(\bz_\ell;\bz_i) - \sum_{i=1}^N\bk_i(\bz_\ell; \bbz_i) \Bigg).
\end{equation}
 Formula \eqref{pkdisk} greatly simplifies numerical simulations of the
dislocation dynamics. Without
an explicit formula, one must solve the Neumann problem at each timestep.

From \eqref{pkdisk}, we can see that the boundary of $B_1$ attracts dislocations.
If $N=1$ and $\bz_1 \in B_1\setminus \{ \bzero\}$, then 
\begin{equation*}
%  \label{j1N1}
  \bj_1(\bz_1) = -b_1\bJ\bk_1(\bz_1;\bbz_1) = 
-\frac{b_1^2}{2\pi}\frac{\bz_1-\bbz_1}{|\bz_1-\bbz_1|^2}=
\frac{b_1^2}{2\pi}\frac{\bz_1}{(1-|\bz_1|^{2})}
\end{equation*}
since $\bz-\bbz = \bz(1-|\bz|^{-2})$. Thus, the force is directed 
radially outward (toward
the nearest boundary point to $\bz_1$) and diverges as $\bz_1 \to \partial B_1$.
If $\bz_1={\bf 0}$ then $\bj_1=\bzero$ and $\bz_1$ will not move. Otherwise, a single
dislocation in $B_1$ will be pulled to $\partial B_1$, and will collide
with $\partial B_1$ in a finite time (assuming the glide directions
span $\R2$).
If $N>1$, then the other dislocations produce boundary forces that will
pull on $\bz_\ell$ in the directions $-b_\ell b_i(\bz_\ell - \bbz_i)$ for each $i$.

%\begin{remark}
%  For $\lambda \neq 1$, define $\varphi : E_1\to B_1$ by 
%$\varphi(\bx) = (x_1,x_2/\lambda)$, where $E_1$ is the ellipse
%\begin{equation}\label{Eepsilon}
%E_1=
%\left\{ (x_1,x_2)\in\R2: (x_1)^2 + \left(\frac{x_2}{\lambda} \right)^2 < 1 \right\}.
%\end{equation}
%For $i=1,\ldots, N$, let $\bz_i \in E_1$.
%Define $v_0(\bx) = u_0(\varphi(\bx);\varphi(\bz))$ for $\bx \in E_1$. Then,
%\begin{equation*}
%%  \label{NeumannE}
%  \left\{
%    \begin{array}{ll}
%      {\rm div}_\bx (\bL\nabla_\bx v_0 )= 0& \text{ in } E_1,\\
%      \bL \nabla_\bx v_0 (\by)\cdot \bn = -\sum_{i=1}^N\bL \bk_i(\by;\bz_i)\cdot \bn
%      & \text{ on } \partial E_1.
%    \end{array}
%\right.
%\end{equation*}
%This allows us to easily compute the forces in the special case
%$\Omega=E_1$  for an anisotropic material ($\lambda \neq 1$) with the same $\lambda$
%as in the definition of $E_1$ \eqref{Eepsilon}.
%\end{remark}

The sets $\cA_\ell$ as given in \eqref{Amb_def} are smooth, because
they are locally given by ${\bf j}_\ell \cdot \bg_0=0$ for a fixed
vector $\bg_0$ (cf. equation \eqref{e0}), and by
\eqref{pkdisk}, ${\bf j}_\ell \cdot \bg_0$ is a rational function
with singularities only at collision points.
%Uniqueness of solutions holds until a dislocation arrives at a source
%point or collides with another dislocation or with the boundary. 

\subsection{The Half-Plane}\label{sec:hpp}
Although the theory developed in this paper only applies to bounded domains,
the equation for the Peach-K\"ohler force \eqref{PKj} is still well-defined, 
provided there is a weak solution to the Neumann problem \eqref{u0Neumann}.
For the special cases of the half-plane and 
the plane we present an explicit expression for the Peach-K\"ohler force 
without resorting to the renormalized energy.

Let $\Omega = \R2_+ := \{\bx\in \R2\, :\, x_2>0\}$ and let $\lambda=\mu =1$. % is similar to
%the case $\Omega = B_1$. 
%We take the definition 
%of the strain $\bh_0$ in \eqref{h0def}
%as our starting point, and again we can write explicitly
%and we find that we can solve exactly the
%Neumann problem \eqref{u0Neumann} on $\R2_+$.
%We can write 
The solution to \eqref{u0Neumann} is given in terms of the inverse tangent,
 using a reflected point across $\partial \R2_+ = \{x_2 = 0\}$. 
%For any $\bz \in \R2_+$ we 
For all $\bz=(z_1,z_2)\in\R{2}$ define %$\tilde\bz\in \R2_-$ as
%\begin{equation*}
%  \label{zbar2}
 $\tilde\bz: =(z_1,-z_2)$. %\quad\text{for all } \bz = (z_1,z_2), \text{ with } z_2 >0.
%\end{equation*}
Then
for $\bz_i \in \R2_+$,
\begin{equation}
  \label{ui0H}
   u^i_0(\bx;\bz_i): = -\frac{b_i}{\pi}
\arctan \left(\frac{x_2-\tilde z_{i,2}}{x_1-\tilde z_{i,1} + |\bx-\tilde\bz_i|\,}\right)
\end{equation}
satisfies
\begin{equation*}
%  \label{NeumannH}
  \left\{
    \begin{array}{ll}
      \Delta_\bx u_0^i(\bx;\bz_i) = 0,& \bx\in \R2_+,\\
      \nabla_\bx u_0^i (\bx;\bz_i)\cdot \bn(\bx) = -\bk_i(\bx;\bz_i)\cdot \bn(\bx),
      & \bx\in \partial \R2_+,
    \end{array}
\right.
\end{equation*}
and
\begin{equation*}
%  \label{gradu0H}
  \nabla_\bx u_0^i (\bx;\bz_i) = -\bk_i(\bx;\tilde\bz_i)\quad\text{for all $\bx\in \R2_+$}.
\end{equation*}
Again, we have 
$u_0(\bx;\bz_1,\ldots,\bz_N) = \sum_{i=1}^Nu_0^i(\bx;\bz_i)$
with $u_0^i$ as in \eqref{ui0H}, and the
Peach-K\"ohler force is 
\begin{equation}
  \label{pkH}
  \bj_\ell(\bz_1,\ldots,\bz_N) =b_\ell \bJ \Bigg(
\sum_{i\neq\ell}\bk_i(\bz_\ell;\bz_i) - \sum_{i=1}^N\bk_i(\bz_\ell; \tilde\bz_i) \Bigg) .
\end{equation}
From \eqref{pkH} it is again not difficult to see that a single dislocation $\bz_1$
in $\R2_+$ with Burgers modulus $b_1$ 
is attracted to $\partial \R2_+$.
%  as if there were a mirror
% dislocation at the position $\tilde\bz$
% with   Burgers modulus $-b_1$. 
As in the case of the disk,
the ambiguity set $\cA$ is smooth except at the intersections of the
$\cA_\ell$. %, but these do not present an obstruction to uniqueness of
%solutions as discussed in Section \ref{sec:unique}. 

%The above formulas could also be derived by using the Green's function for
%$\R2_+$, which is
%\begin{equation}
%  \label{Green-halfplane}
%  G_H(\bx,\by): = \frac{1}{4\pi}\left(
%\log |\bx-\by|^2 + \log |\overline \bx - \by|^2\right).
%\end{equation}
%The function $G_H$ will play a role in determining the 
%behavior of dislocations near the boundary of smooth domains,
%as described in Section \ref{sec:boundary}.

\subsection{The Plane}\label{sec:fullplane}
The case $\Omega = \R2$ and $\lambda=\mu =1$ is the simplest case. There  is no boundary so
$u_0 \equiv 0$ and, by \eqref{PKj}, the Peach-K\"ohler force is
%we can seek a function $U$ satisfying 
then 
\begin{equation}
  \label{jplane}
  \bj_\ell (\bz_1,\ldots,\bz_N)
=b_\ell {\bf J}\sum_{i\neq\ell}{\bf k}_i({\bf z}_\ell;{\bf z}_i).
\end{equation}
Even though the renormalized energy has not been defined for 
unbounded domains, in the case of the plane we can 
formally write $\bj_\ell = -\nabla_{\bz_\ell}U$, where, 
up to an additive constant,
%we have
\begin{equation*}
%  \label{Uplane}
  U(\bz_1,\ldots,\bz_N) = 
- \sum_{i=1}^{N-1}  \sum_{j=i+1}^N \frac{ b_ib_j}{2\pi}\log |\bLam (\bz_i-\bz_j)|,
\end{equation*}
with $\bLam$ defined in \eqref{matrices}.  
% The renormalized energy needs to be redefined in the case of unbounded
% $\Omega$, since the energy (cf. \eqref{J_def}) will 
% be infinite on any 
% strain satisfying the Euler-Lagrange equation (of the form $\bh_0$ in \eqref{h0def}).
%In this case, solutions to the dynamics
%exist until a collision between dislocations occurs. 

In general, it can be difficult to exhibit an example that shows analytically
fine cross-slip (though it is regularly observed
in numerical simulations). However, in the case 
$\Omega = \R{2}$, this can be done with two dislocations as follows.
Suppose we have a system of two dislocations $\bZ = (\bz,\bw)\in\R{4}$ with Burgers moduli $b_1 = -b_2 =: b>0$, respectively.
Under these assumptions, \eqref{jplane} reduces to
\begin{equation}
  \label{j1j2}
\bj_1(\bz,\bw) = -\frac{b^2}{2\pi}
\frac{\bz-\bw}{|\bz-\bw|^2} = -\bj_2(\bz,\bw).   
\end{equation}
 Assume that the glide directions are along the lines $x_2=\pm x_1$,
\begin{equation}
  \label{glide_directions}
  \mathcal{G} = \left\{ 
\pm \bg_1,\; \pm \bg_2
\right\},\quad
\bg_1 := \frac{1}{\sqrt{2}}\left(
    \begin{array}{c}      1 \\ 1   \end{array}\right) ,\;
\bg_2 := \frac{1}{\sqrt{2}}\left(
    \begin{array}{r}      1 \\ -1   \end{array}\right) .
\end{equation}
There are two cases of initial conditions $\bZ_0 = (\bz_0,\bw_0)$ 
with $\bz_0 = (z_{0,1},z_{0,2})$, $\bw_0=(w_{0,1},w_{0,2})$
 to consider: either $\bz_0$ and $\bw_0$ are aligned
along a vertical or horizontal line, or they are not. That is, either
$z_{0,1} = w_{0,1}$ or $z_{0,2}=w_{0,2}$ (but not both), or $z_{0,i} \neq w_{0,i}$ for
$i=1,2$. 

We begin by considering the case $z_{0,2}=w_{0,2}$. 
Let $y:=z_{0,2}=w_{0,2}$, and without loss of
generality take $w_{0,1}>z_{0,1}$.
From \eqref{j1j2} we have 
\begin{equation}
  \label{j1j2_y}
  \bj_1(\bZ_0)=
\bj_1(z_{0,1},y,w_{0,1},y) = \frac{b^2}{2\pi}\frac{1}{w_{0,1}-z_{0,1}}\left(
    \begin{array}{c}      1 \\ 0    \end{array}\right) 
%= -  \bj_2(z_1^0,y,w_1^0,y)
=-\bj_2(\bZ_0).
\end{equation}
Since $w_{0,1}-z_{0,1} >0$, we see that $\bj_1(\bZ_0)$ %$\bj_1(z_1^0,y,w_1^0,y)$ 
is aligned with $(1,0)$ and $\bj_2(\bZ_0)$ %$\bj_2(z_1^0,y,w_1^0,y)$ 
is aligned with $(-1,0)$. Thus, the 
maximally dissipative glide directions for $\bz$ are 
$\bg_1$ and $\bg_2$ (see \eqref{glide_directions}) and the maximally
dissipative glide directions for $\bw$ are $-\bg_1$ and $-\bg_2$. 
Define $\bg_1^0: = \bg_1-\bg_2 = (0,\sqrt{2})$ and $\bg_2^0: =
-\bg_1 +\bg_2 = -\bg_1^0$, so that locally, near $\bZ_0$, the ambiguity
surfaces are 
$\cA_1\cap B_r(\bZ_0) = \{ \bZ \, :\, \bj_1(\bZ)\cdot\bg_1^0 = 0\}$,
 $\cA_2\cap B_r(\bZ_0) = \{ \bZ \, :\, \bj_2(\bZ)\cdot\bg_2^0 = 0\}$
for some small $r>0$. 
From \eqref{j1j2} we see that
$\bj_1(\bZ)\cdot\bg_1^0 = 0$ if and only if $z_2 = w_2$, and the same holds
for $\bj_2(\bZ)\cdot\bg_2^0 = 0$, so that
\begin{equation*}
%  \label{A1A2}
  \cA_1\cap B_r(\bZ_0) = \cA_2\cap B_r(\bZ_0) =
\{ \bZ = (\bz,\bw) \in B_r(\bZ_0)\, :\, z_2 = w_2\}.
\end{equation*}
This is a degenerate situation, since the ambiguity surfaces $\cA_1$ and $\cA_2$ coincide locally,
and instead of having four vector fields near the intersection, we
have two vector fields. That is, the fields $\bef^{(+,+)}$
and $\bef^{(-,-)}$ (see \eqref{f plus minus}) are defined on either side of the surface $\cA_1$,
but since $\cA_1=\cA_2$, there are no regions where the
fields $\bef^{(-,+)}$ or $\bef^{(+,-)}$ are defined. 
We choose a sign for the normal  to
$\cA_1$ and $\cA_2$ at $\bZ_0$ and
set  
\begin{equation}
  \label{normal_plane}
\bn:= \frac{1}{\sqrt{2}}(0,1,0,-1)  .
\end{equation}
% \left(
%   \begin{array}{r}
%     0\\1\\0\\-1
%   \end{array}
% \right)
% which is
% the normal to 
% (as opposed to $-\bn$). 

Recall the convention that $\cA_1$ (and $\cA_2$) divides $B_r(\bZ_0)$
into two regions, $V^\pm$ and $\bn$ points from $V^-$ to $V^+$.
 A point in $V^+$ is of the form
$\bZ_0 + \ep \bn = (z_{0,1},y+\ep/\sqrt2,w_{0,1}, y-\ep/\sqrt2)$, and from \eqref{j1j2}
\begin{equation*}
%  \label{jplus}
  \bj_1(\bZ_0+\ep \bn) = \frac{b^2}{2\pi}\frac{1}{(z_{0,1}-w_{0,1})^2+2\ep^2}\left(
    \begin{array}{c}      w_{0,1}-z_{0,1} \\ -\sqrt2\ep    \end{array}\right) 
= -  \bj_2(\bZ_0+\ep \bn),
\end{equation*}
so $\bg_2$ is the maximally dissipative glide direction for $\bz$, and 
$-\bg_2$ is the maximally dissipative glide direction for $\bw$ if
$\bZ \in V^+$. Similarly, a  point in $V^-$ is of the form
$\bZ_0 -\ep \bn = (z_{0,1},y-\ep/\sqrt2,w_{0,1}, y+\ep/\sqrt2)$, and the maximally
dissipative glide directions for $\bz$ and $\bw$ in this case are
$\bg_1$ and $-\bg_1$, respectively. 
Thus, we have for $\bZ \in B_r(\bZ_0)$,
\begin{align*}
  &\bef^{(+,+)}(\bZ) := ( (\bj_1(\bZ)\cdot \bg_2)\bg_2\, , \,
 (\bj_2(\bZ)\cdot (-\bg_2))(-\bg_2) ), \\ % \label{fpp}\\
  &\bef^{(-,-)}(\bZ) := ( (\bj_1(\bZ)\cdot \bg_1)\bg_1 \, , \,
 (\bj_2(\bZ)\cdot (-\bg_1))(-\bg_1) ).   %\label{fmm}
\end{align*}
Since $\bj_1(\bZ) = -\bj_2(\bZ)$ we have
\begin{align}
 \bef^{(+,+)}(\bZ) := (\bj_1(\bZ)\cdot \bg_2) ( \bg_2\, , \,
  -\bg_2 )  ,\quad
  \bef^{(-,-)}(\bZ) := (\bj_1(\bZ)\cdot \bg_1) ( \bg_1\, , \,
  -\bg_1 ).   \label{fppfmm}
\end{align}
From \eqref{glide_directions} and \eqref{j1j2_y} we have
$\bj_1(\bZ_0)\cdot \bg_1= 
\bj_1(\bZ_0)\cdot \bg_2 = \frac{b^2}{2\sqrt2 \pi}(w_{0,1}-z_{0,1})^{-1}>0$,
and from
\eqref{glide_directions} and \eqref{normal_plane}  we have
$\bn \cdot (\bg_2, -\bg_2) = -1$ and
$  \bn \cdot (\bg_1, -\bg_1) = 1$. 
Thus,
\begin{equation*}
%  \label{nfpp}
%  \begin{split}
\bn\cdot \bef^{(+,+)} (\bZ_0)=  -\frac{b^2}{2\sqrt2 \pi (w_{0,1}-z_{0,1})} <0, \quad 
\bn\cdot \bef^{(-,-)} (\bZ_0)=  \frac{b^2}{2\sqrt2 \pi (w_{0,1}-z_{0,1})} >0,
%\end{split}
\end{equation*}
so the fine cross-slip conditions \eqref{conditions} are satisfied (there are no
conditions for $\bef^{(+,-)}$ or $\bef^{(-,+)}$ since locally 
$\cA_1 = \cA_2$). 
By \eqref{3.5*}, $\dot \bZ$ must be a convex combination of 
$\bef^{(+,+)}$ and $\bef^{(-,-)}$,
$\dot \bZ = \alpha \bef^{(+,+)} (\bZ)+ (1-\alpha) \bef^{(-,-)}(\bZ)$,
and the trajectory $\bZ(t) \in \cA_1 = \cA_2$ for some time interval
$[0,T]$. 
Therefore, $\bZ(t) = (\bz(t),\bw(t)) = (z_1(t), z_2(t),
w_1(t), w_2(t))$ and $z_2(t) = w_2(t)$ for $t\in [0,T]$. 
From \eqref{fppfmm} and the fact that
$ \bj_1(\bZ)\cdot \bg_1 = \bj_1(\bZ)\cdot \bg_2$ 
whenever $z_2 = w_2$, we have
\begin{align*}
  \dot \bZ &
=\alpha (\bj_1(\bZ)\cdot \bg_2) ( \bg_2\, , \,  -\bg_2 )  
 + (1-\alpha) (\bj_1(\bZ)\cdot \bg_1) ( \bg_1\, , \,  -\bg_1 )  \\
 &=\frac{b^2}{4 \pi (w_1-z_1)} \left(
1,1-2\alpha,-1,2\alpha-1\right).
\end{align*}
The condition $\bn\cdot \dot \bZ = 0$ yields $\alpha = \frac12$, so
the equations of motion \eqref{3.5*} are
\begin{align*}
  \dot \bZ = (\dot z_1,\, \dot z_2,\, \dot w_1,\, \dot w_2)
=\frac12 \left(\bef^{(+,+)}(\bZ) +\bef^{(-,-)}(\bZ)\right)
 = \frac{b^2}{4 \pi (w_1-z_1)} (1,0,-1,0).
\end{align*}
In particular, $\dot z_2 = 0$, $\dot w_2=0$, and
$z_2(0) = y = w_2(0)$,
so $z_2(t) = y = w_2(t)$ for
$t\in[0,T]$. The equations for $z_1$ and $w_1$ are easily
solved with
\begin{align*}
  z_1(t) &= -\frac12 \left( \left (w_{0,1}-z_{0,1} \right)^2-
\frac{b^2}{\pi} t \right)^{\frac12}+\frac12 (z_{0,1}+w_{0,1})\\
  w_1(t) &= \frac12 \left( \left (w_{0,1}-z_{0,1} \right)^2-
\frac{b^2}{\pi} t \right)^{\frac12}+\frac12 (z_{0,1}+w_{0,1}).
\end{align*}
This implies that the trajectory $\bZ(t)$ moves on $\cA_1 = \cA_2$ up to
the maximal time $T = \frac{\pi}{b^2}(w_{0,1}-z_{0,1})^2$, and $z_1(t)$
increases from $z_{0,1}$ while $w_1(t)$ decreases from $w_{0,1}$, with the two
meeting at $z_1(T) = w_1(T) = \frac12(z_{0,1}+w_{0,1})$. At this collision,
the dynamics are no longer well-defined.

If the initial condition has $\bz_0$ and $\bw_0$ vertically aligned, then
the same analysis applies, but the situation is rotated. 

If $\bz_0$ and $\bw_0$ are not aligned vertically or horizontally, then a regular glide
motion occurs until either $z_1 = w_1$ or $z_2 = w_2$, and then the
above analysis applies. To see this, consider $\bz_0 = (z_{0,1},z_{0,2})$ and
$\bw_0 = (w_{0,1},w_{0,2})$, and without loss of generality, assume that
 $w_{0,1}>z_{0,1}$ and
$w_{0,2}>z_{0,2}$ (the other cases are similar). 
In this case
\begin{equation*}
%  \label{j1j2xy}
%  \begin{split}
  \bj_1(\bZ_0)= 
\frac{b^2}{2\pi |\bz_0 - \bw_0|^2}
\left(  \begin{array}{c}
    w_{0,1}-z_{0,1}\\w_{0,2}-z_{0,2} 
  \end{array}\right) \\ = -\bj_2(\bZ_0).
%\end{split}
\end{equation*}
Since $w_{0,1}-z_{0,1}>0$ and $w_{0,2}-z_{0,2}>0$, the maximally dissipative glide
directions for $\bj_1$ and $\bj_2$ are $\bg_1$ and $-\bg_1$,
respectively. Thus, $\bz$ glides in the $\bg_1$ direction, so that 
$z_1$ and $z_2$ increase from $z_{0,1}$ and $z_{0,2}$, while $\bw$ glides in
the $-\bg_1$ direction, so $w_1$ and $w_2$ decrease from $w_{0,1}$ and
$w_{0,2}$. At some time $t_1$ we must obtain either $z_1(t_1)=w_1(t_1)$ or $z_2(t_1) =
w_2(t_1)$. If only one of these equalities holds, we are in the situations
described above and fine cross-slip occurs. If both of these
equalities hold, then $\bz$ and $\bw$ have collided and the  dynamics
is no longer defined.

\begin{remark}[Mirror Dislocations]\label{sec:mirror}
A direct inspection of equations \eqref{pkdisk} and \eqref{pkH} shows
 that the force on $\bz_\ell$ in $\Omega=B_1$ and $\Omega=\R2_+$
is the same as the force on $\bz_\ell$ in $\R2$ if one adds $N$ 
dislocations with opposite Burgers moduli at the points $\bar \bz_i$
in the case $\Omega=B_1$, and at $\tilde \bz_i$ in the case
$\Omega=\R2_+$, for $i=1,\ldots,N$.
\end{remark}
\section{Appendix} \label{sec:hull_proof}
We collect some technical results that are needed in the proofs from
Section \ref{sec:dynamics}.
\subsection{Proof of Lemma \ref{lem:hulls}}\label{proofoflem:hulls}
\begin{proof}[Proof of Lemma \ref{lem:hulls}]
Let $\bZ\in\R{2N}$ be fixed. For simplicity, in this proof we drop the explicit 
dependence on $\bZ$.
By \eqref{f023} we can write $F_\ell%(\bZ)
=\{\bp_\ell,\bq_\ell\}$, with $\bp_\ell,\bq_\ell\in\R2$, %$\bp_\ell\neq\bq_\ell$ 
for all 
$\ell=1,\ldots,N$. By definition, we have
\begin{align*}
\hat F%(\bZ)
& = \left\{\vecc{s_1\bp_1+(1-s_1)\bq_1}{s_N\bp_N+(1-s_N)\bq_N}, 
\quad s_1,\ldots ,s_N \in [0,1] \right\} \quad \mbox{and} \\ %\label{tFN}\\
\ch F%(\bZ) 
& =  \left\{\bV\in \R{2N} \!: \bV = 
\alpha_1 \vecN{\bp_1}{\bp_2}{\bp_N} +\alpha_2 \vecN{\bq_1}{\bp_2}{\bp_N} + 
\ldots +\alpha_{2^N} \vecN{\bq_1}{\bq_2}{\bq_N}, \right.\nonumber \\
&\qquad \qquad \mbox{where}\quad 
\left. \alpha_i \in [0,1] \text{ for all } i=1, \ldots, 2^N,\,\,\mbox{and}\,\, \sum_{i=1}^{2^N} \alpha_i = 1
 \right\}.%\label{coFN}
\end{align*}

To see that $\ch F %(\bZ) 
\subseteq \hat F$, %(\bZ)$, 
first note that $F %(\bZ) 
\subseteq \hat F$ %(\bZ)$ 
because if $\bX \in F$ %(\bZ)$
then each component is either $\bp_i$ or $\bq_i$, which is
a point in $\hat F$ %(\bZ)$ 
with $s_i = 1$ or $0$.

%To show that $\ch F%(\bZ) 
%\subseteq \hat F$, %(\bZ)$, 
Next we show that 
$\hat F$ %(\bZ)$ 
is convex.
Let $\bV,\bW \in \hat F$. %(\bZ)$. 
Then their $i$-th components are
$\bv_i = s_i\bp_i + (1-s_i)\bq_i, \,\, \bw_i = r_i\bp_i + (1-r_i)\bq_i$, respectively.
Let $\lambda \in [0,1]$, then the $i$-th component of
$\lambda \bV + (1-\lambda)\bW$ is 
\begin{align*}
\lambda \bv_ i +(1-\lambda)\bw_i&=
\lambda (s_i\bp_i + (1-s_i)\bq_i) + (1-\lambda)(r_i\bp_i + (1-r_i)\bq_i)   \\
&=(\lambda s_i + (1-\lambda)r_i)\bp_i  + 
(\lambda (1-s_i) + (1-\lambda)(1-r_i))\bq_i.
\end{align*}
Setting $\theta_i := \lambda s_i + (1-\lambda)r_i$,
then $\theta_i \in [0,1]$ because $s_i, r_i \in [0,1]$
and 
$$\lambda (1-s_i) + (1-\lambda)(1-r_i) = 
1 - (\lambda s_i + (1-\lambda)r_i) = 1-\theta_i,$$
so 
$\lambda \bv_ i +(1-\lambda)\bw_i = \theta_i \bp_i + (1-\theta_i)\bq_i$,
with $\theta_i \in [0,1],$ for every $ i=1,\ldots,2N.$
Hence, $\lambda \bV + (1-\lambda)\bW \in \hat F$, %(\bZ)$, 
so $\hat F$ 
is convex. % and therefore

We prove that $\hat F (\bZ)
\subseteq \ch F (\bZ)$
by induction on $N$. 
To highlight the dependence on the dimension, we write
$F^{(N)}(\bZ)\subseteq\R{2N}$ 
and $\hat F^{(N)}(\bZ)\subseteq\R{2N}$ 
for the sets $F (\bZ)$ and $\hat F (\bZ)$ defined in \eqref{f009} and 
\eqref{f026}. % for the case of $N$ dislocations.

The case $N=1$ is trivial since
$F^{(1)} (\bZ)= \{\bp_1,\,\bq_1\}$ and
 any $\bV^{(1)}\in \hat F^{(1)}(\bZ)$ is of the form
$\bV^{(1)} = s_1\bp_1+(1-s_1)\bq_1 \in \ch F^{(1)}(\bZ)$.
Now assume that $\hat F^{(N-1)}(\bZ) 
\subseteq \ch F^{(N-1)}(\bZ)$ for some $N$. 
Let $\bV^{(N)}\in \hat F^{(N)}(\bZ)$, so
\begin{equation*}
%  \label{eq:1}
\bV  ^{(N)} = \left(
  \begin{array}{c}
    s_1\bp_1+(1-s_1)\bq_1 \\ \vdots \\
    s_{N}\bp_{N}+(1-s_{N})\bq_{N} 
  \end{array}
\right) = \left(
  \begin{array}{c}
    \bV^{(N-1)}  \\
    s_{N}\bp_{N}+(1-s_{N})\bq_{N} 
  \end{array}
\right)
\end{equation*}
for $\bV^{(N-1)}\in \hat F^{(N-1)}(\bZ)$. By the induction hypothesis,
$\bV^{(N-1)}\in \ch F^{(N-1)}(\bZ)$, so there exist
$\{\alpha_i\}_{i=1}^{2^{N-1}}$
and $\hat \bV_i^{(N-1)}\in F^{(N-1)}(\bZ)$ 
 such that $\alpha_i\in [0,1]$,
$\sum _{i=1}^{2^{N-1}}\alpha_i = 1$ and
\begin{equation*}
%  \label{eq:2}
  \bV^{(N-1)} =
\left(
  \begin{array}{c}
    s_1\bp_1+(1-s_1)\bq_1 \\ \vdots \\
    s_{N-1}\bp_{N-1}+(1-s_{N-1})\bq_{N-1} 
  \end{array}
\right) 
= \sum_{i=1}^{2^{N-1}}\alpha_i \hat \bV_i^{(N-1)}.
\end{equation*}
% where $\hat \bV_i^{(N-1)}\in F^{(N-1)}(\bZ)$ for each $i$, so
% \begin{equation}
%   \label{eq:3}
%   \hat \bV_i^{(N-1)} = \left(
%   \begin{array}{c}
%     \bw_1\\ \vdots \\
%     \bw_{N-1}
%   \end{array}
% \right) ,\qquad
% \bw_k \in \{ \bp_k,\,\bq_k\} \;\; \forall \, k=1,\ldots,N-1.
% \end{equation}
We define $\hat \bV_i^{(N)}\in F^{(N)}(\bZ)$ for $i=1,\ldots, 2^N$ as
\begin{equation*}
%  \label{eq:4}
  \hat \bV_i^{(N)} := \left(
  \begin{array}{c}
    \hat \bV_i^{(N-1)}  \\
  \bp_{N}  
  \end{array}
\right),\quad
 \hat \bV_{i+2^{N-1}}^{(N)}: = \left(
  \begin{array}{c}
    \hat \bV_i^{(N-1)}  \\
  \bq_{N}  
  \end{array}
\right)\quad 
\mbox{for $i=1,\ldots,2^{N-1}$,}
\end{equation*}
and we define the coefficients $\beta_i\in[0,1]$
for $i=1,\ldots, 2^N$ as
\begin{equation*}
 % \label{eq:5}
  \beta_i: = s_N\alpha_i,\quad \beta_{i+2^{N-1}}: = (1-s_N)\alpha_i
\quad \mbox{for $i=1,\ldots,2^{N-1}$.}
\end{equation*}
Hence, $\sum_{i=1}^{2^N}\beta_i = 1$ and
\begin{align*}
  \bV  ^{(N)} &= \left(
  \begin{array}{c}
    \bV^{(N-1)}  \\
    s_{N}\bp_{N}+(1-s_{N})\bq_{N} 
  \end{array}
\right)
= \left(
  \begin{array}{c}
    \sum_{i=1}^{2^{N-1}}\alpha_i \hat \bV_i^{(N-1)}\\
    s_{N}\bp_{N}+(1-s_{N})\bq_{N} 
  \end{array}
\right)\\
& = \left(
  \begin{array}{c}
    \sum_{i=1}^{2^{N-1}}s_N\alpha_i \hat \bV_i^{(N-1)}
+    \sum_{i=1}^{2^{N-1}}(1-s_N)\alpha_i \hat \bV_i^{(N-1)}\\
    \sum_{i=1}^{2^{N-1}}\alpha_i  s_{N}\bp_{N}+
\sum_{i=1}^{2^{N-1}}\alpha_i  (1-s_{N})\bq_{N} 
  \end{array}
\right)\\
&= \sum_{i=1}^{2^{N-1}}s_N\alpha_i \left(
  \begin{array}{c}
    \hat \bV_i^{(N-1)}  \\
    \bp_{N} 
  \end{array}
\right)
+  \sum_{i=1}^{2^{N-1}}(1-s_N)\alpha_i \left(
  \begin{array}{c}
    \hat \bV_i^{(N-1)}  \\
    \bq_{N} 
  \end{array}
\right)\\
&= \sum_{i=1}^{2^{N-1}}\beta_i  \hat \bV_i^{(N)}
+  \sum_{i=1}^{2^{N-1}}\beta_{i+2^{N-1}}  \hat \bV_{i+2^{N-1}}^{(N)}
= \sum_{i=1}^{2^{N}}\beta_i  \hat \bV_i^{(N)} \in \ch F^{(N)}(\bZ).
\end{align*}

\end{proof}

\subsection{Lemmas on the Singular Set}
\begin{lemma}\label{lem:path}
The set $\cD (F)$, as defined in \eqref{f007}, is open and connected.
\end{lemma}
\begin{proof}
From \eqref{f007} and \eqref{f024}, it is clear that $\cD(F)$ is open. We
will now show that $\cD(F)$ is path connected.
Let $\bw, \bz_1, \ldots, \bz_N \in \Omega$ be distinct points, and
  let $\bZ, \widehat \bZ\in \cD(F)$ be given by
 $\bZ = (\bz_1, \ldots,\bz_N)$ and 
$\widehat \bZ =(\bz_1,\ldots,\bz_{\ell-1},\bw,\bz_{\ell+1},\ldots,\bz_N)$. 
We construct  a continuous path
 $\gamma : [0,1]\to\cD(F)$ with  $\gamma(0) = \bZ$ and 
$\gamma (1) =  \widehat \bZ$ as follows.

Note that $\Omega \setminus\{\bz_1,\ldots,\bz_{\ell-1},\bz_{\ell+1},\ldots,\bz_N\}$ is path
connected. 
Thus there is a path 
$\gamma_\ell:[0,1]\to \Omega  \setminus\{\bz_1,\ldots,\bz_{\ell-1},\bz_{\ell+1},\ldots,\bz_N\}$
with $\gamma_\ell(0) = \bz_\ell$ and $\gamma_\ell(1) = \bw$. Then setting 
$\gamma(t) = (\bz_1,\ldots,\bz_{\ell-1},\gamma_\ell(t),\bz_{\ell+1},\ldots,\bz_N)$ 
for each $t\in [0,1]$
gives a path in $\cD(F)$ from $\bZ$ to $\widehat \bZ$.

We can now connect any 
$\bZ = (\bz_1, \ldots,\bz_N)\in \cD(F)$
to any other
$\bW = (\bw_1, \ldots,\bw_N)\in \cD(F)$
by first moving $\bz_1$ to $\bw_1$ as above, then $\bz_2$ to
$\bw_2$, and so on, until all the $\bz_i$ are moved to $\bw_i$,
producing a path from $\bZ$ to $\bW$. 
\end{proof}

To prove the following lemma we will use the fact that the renormalized energy (see \eqref{expansion}) diverges logarithmically with
the relative distance between the dislocations, that is,
\begin{equation}\label{U diverges}
U (\bz_1,\ldots,\bz_N)=-\sum_{i=1}^{N-1}  \sum_{j=i+1}^N
\frac{\mu \lambda b_i b_j}{4\pi}\log|\bLam(\bz_i - \bz_j)| + O(1)
\end{equation}
as $|\bz_i - \bz_j| \longrightarrow 0$. 
We refer to \cite{BM14} for a proof.

\begin{lemma}\label{lem:open}
Fix  $\ell \in \{1,\ldots,N\}$
and let $\be \in \R2 \setminus \{0\}$ be fixed. 
Then the set 
$V=\{\bZ \in \cD(F)\, :\, \bj_\ell (\bZ) \cdot \be = 0\}$
% is nowhere dense.%not open in $\cD(G)$.
has empty interior.
\end{lemma}
\begin{proof}
The set $V$ is closed because $\bj_\ell$ is continuous. Suppose
there is a ball $B\subset V$.
From Lemma \ref{lem:janalytic}, we have that $\bj_\ell (\bZ)\cdot \be$ is
analytic in $B$ and is constant, therefore $\bj_\ell (\bZ)\cdot \be$ is
constant in the largest connected component of $\cD(F)$ 
containing $B$. Hence, by Lemma \ref{lem:path}, $\bj_\ell (\bZ)\cdot \be=0$
in $\cD(F)$. From \eqref{PK=DU}, we have that
\begin{equation}
  \label{Ue0}
  \nabla_{\bz_\ell}U(\bZ)\cdot \be = 0 \quad {\rm in}\,\,\, \cD(F), 
\end{equation}
so $U$ is constant when  $\bz_\ell$ varies along the direction $\be$.

Consider a fixed $\bZ^* = (\bz_1,\bz_2, \ldots,\bz_N)\in \cD(F)$.
Let $h >0$, and for  $\delta\in (0,h]$ define $\bz_\ell^\delta:=\bz_\ell+\delta\be$. 
We assume that $h_0$ small enough so that $\bz_\ell^\delta \in \Omega \setminus\{\bz_1,\ldots,\bz_N\}$
for  $\delta \in (0,h_0]$. 
Fix a $k\neq \ell$ and $h\in(0,h_0]$, and let $\bZ^h$ be the point in $\cD(F)$
obtained by replacing $\bz_k$ in $\bZ^*$ with $\bz_\ell^h$, i.e.,
\[
\bZ^h:=\{\bz_1,\ldots,\bz_\ell,\ldots,\bz_{k-1},\bz_\ell^h,\bz_{k+1},\ldots,\bz_N\}.
\]
Letting $\delta_n = \left(1-\frac1n \right)h$,
we construct the sequence $\{\bZ_n\}\subset \cD(F)$ given by
$$\bZ_n:=\left\{\bz_1,\ldots,\bz_\ell+\delta_n\be,
\ldots,\bz_{k-1},\bz_\ell^{h},\bz_{k+1},\ldots,\bz_N\right\}.$$
We have $\bZ_1=\bZ^h$, and 
$$\bZ_n\to\bZ_{\infty}:=
\{\bz_1,\ldots,\bz_\ell^h,\ldots,\bz_{k-1},\bz_\ell^h,\bz_{k+1},\ldots,\bz_N\}
\quad {\rm as} \,\,\, n\to\infty.$$
Note that $\bZ_{\infty}\notin \cD(F)$ because $\bz_\ell$ and $\bz_k$ are colliding
as $n\to \infty$. In particular, by \eqref{U diverges},
$|U(\bZ_n)|\to \infty$ as $n \to \infty$.
On the other hand, in the sequence $\{\bZ_n\}$, only the 
$\ell$-th dislocation is moving, and it is moving along the direction $\be$, 
so from \eqref{Ue0}, 
$U(\bZ_n)$ remains constant % = U(\bZ^h)<\infty$ 
for all $n$.
We have reached a contradiction and we conclude that
$V$ does not contain any ball.
\end{proof}

\begin{lemma}
  \label{lem:Mempty}
The set $\widetilde M_\ell^\infty $, as defined in \eqref{Minfty2}, is empty.
\begin{proof}
Without loss of generality, let $\ell = 1$.
Recall that
\[
 \widetilde M_1^\infty  =  \{ \bZ\, : \,  \bj_1(\bZ)\cdot \bg_0 = 0, \
\partial^{\balpha}(\bj_1(\bZ)\cdot \bg_0) = 0 \text{ for all }\balpha \in N_1 \},
\]
with $N_1$ defined in \eqref{Nell}.
Suppose that $\widetilde M_1^\infty \neq \emptyset$ and 
$\tilde\bZ = (\tilde\bz_{1},\ldots,\tilde\bz_{N}) \in \widetilde M_1^\infty$.
Since $\bj_1\cdot \bg_0$ is analytic and $\tilde\bZ\in\widetilde M_1^\infty$,
we have that
$\bj_1(\bZ)\cdot \bg_0=\bj_1(\tilde\bZ)\cdot \bg_0$ for 
$\bZ \in \{\tilde\bz_{1}\}\times V$,
where $V$ is open in $\R{2N-2}$. 
Take $V$ to be the largest connected
component of $\cD(F)$ with $\bz_1 = \tilde\bz_{1}$, which, by the same
argument as Lemma \ref{lem:path}, can be written 
$\{\tilde\bz_{1}\}\times V = \{\bZ\in \cD(F)\, : \, \bz_1 = \tilde\bz_{1}\}$. 
We cannot follow the energy approach of Lemma \ref{lem:open},
because that would require moving $\bz_1$, which is fixed.
Instead, let $0<\ep_0 \ll 1$ and
construct a sequence $\{\bZ_n\}\subset  V_0$, where
\[
 V_0 : = \left\{\bZ\in V\, : \, \min_{i\in\{1,\ldots,N\}}\dist (\bz_i,\partial \Omega)>\ep_0\right\},
\]
($\ep_0$ is only required to assure we do not have boundary collisions).
To be precise, choose $\bz_3,\ldots,\bz_N\in\Omega$ 
pairwise distinct and such that $\bz_k\neq \tilde\bz_1$ and 
$\dist(\bz_k,\partial\Omega)>\ep_0$ for every $k=3,\ldots,N$.
Therefore, for $n\geq1$ and $\delta_0>0$ sufficiently small, 
$\bZ_n:=(\tilde\bz_{1}, \tilde\bz_{1} + \delta_n \bg_0,\bz_3,\ldots,\bz_N)$ 
belongs to $V_0$, where
%for $n\geq 1$, let
%\[
%\bZ_n =(\bz_1^0, \bz_1^0 + \delta_n \bg_0,\bz_3,\ldots,\bz_N),
%\]
where   $\delta_n =\delta_0/n$.
%and  $\delta_0>0$ is chosen small enough so that 
%$(\bz_1^0, \bz_1^0 + \delta\bg_0,\ldots,\bz_N)\in V_0$ for all $\delta\leq \delta_0$.
Then $\bj_1(\bZ_n)\cdot \bg_0=\bj_1(\tilde\bZ)\cdot \bg_0$ by construction,
but $\bZ_\infty \notin \cD(F)$, where
$\bZ_\infty = \lim_{n\to \infty}\bZ_n$, because 
the first and second dislocations have collided. 

For each $n$, all the components of $\bZ_n$ are a bounded distance from $\partial \Omega$.
Thus, by \eqref{k_def}, \eqref{u0Neumann}, and standard elliptic estimates, there exists $C>0$ such that
$|\nabla u(\tilde\bz_1;\bZ_n)|\leq C$ for all $n$. For  each $n$ the singular strains
$|\bk_i(\tilde\bz_1;\bz_i)|$ are  bounded 
% for $i\neq 1,2$,
% where $\bz_i^n=\bz_i^0$ is the $i$-th dislocation in $\bZ_n$ 
for $i\geq 3$.
However, 
$|\bk_2(\tilde\bz_1;\tilde\bz_1+\delta_n\bg_0)|\geq c/\delta_n \to \infty$ as $n\to \infty$, for some $c>0$.
Thus, we see from \eqref{k_def} and \eqref{PKj} that for large $n$, the force $\bj_1(\bZ_n)$
will be large in magnitude and  
aligned closely with $b_1b_2\big( \tilde\bz_1 -(\tilde\bz_1 +\delta_n\bg_0)\big)$ 
(i.e., $\bj_1(\bZ_n)$ will be nearly parallel or anti-parallel  to  $\bg_0$). 
Therefore, 
$\bj_1(\tilde\bZ)\cdot \bg_0=\bj_1(\bZ_n)\cdot \bg_0 
\geq c_1 |\bj_1(\bZ_n) |{\cdot}|\bg_0| \to \infty$
as $n\to \infty$, for some $c_1>0$, which contradicts the fact that 
$\bj_1(\bZ_n)\cdot \bg_0=\bj_1(\tilde\bZ)\cdot \bg_0$
We conclude that $\widetilde M_1^\infty =\emptyset$.
\end{proof}
\end{lemma}

The following lemma was used in the proof of Theorem \ref{theorem intersections}.
\begin{lemma}\label{lemma unique}
Let $\bZ_1 \in \cA_\ell \cap \cA_k$ for $k\neq \ell$, but 
$\bZ_1 \notin \cA_i$ for $i\neq k,\ell$.  
Also, assume that
$\bZ_1\notin \cE_{{\rm zero}}\cup \cS_\ell \cup \cS_k$ and that \eqref{cond1}-\eqref{cond4} hold. 
Then there is at most one  $\bZ \in{\rm co}\, F (\bZ_1)$ such that
\begin{equation}\label{ortho} 
 \bn_k(\bZ_1)\cdot \bZ=0\quad \text{and}\quad \bn_{\ell}(\bZ_1)\cdot \bZ=0,
\end{equation}
where $\bn_k$ and $\bn_\ell$ are normal vectors for $\cA_k$ and
$\cA_\ell$, respectively.
\end{lemma}

\begin{proof}
Without loss of generality, assume that $k=1$ and $\ell=2$. Then  \eqref{f plus minus} become
%\begin{subequations}\label{f plus minu}
\begin{align*}
  \bef^{(+,+)}(\bZ_1) &= (\hbf_1^+(\bZ_1),\hbf_2^+(\bZ_1),\bef_3(\bZ_1),\ldots,\bef_N(\bZ_1)), \\
  \bef^{(+,-)}(\bZ_1) &= (\hbf_1^+(\bZ_1),\hbf_2^-(\bZ_1),\bef_3(\bZ_1),\ldots,\bef_N(\bZ_1)),\\
  \bef^{(-,+)}(\bZ_1) &= (\hbf_1^-(\bZ_1),\hbf_2^+(\bZ_1),\bef_3(\bZ_1),\ldots,\bef_N(\bZ_1)),\\
  \bef^{(-,-)}(\bZ_1) &= (\hbf_1^-(\bZ_1),\hbf_2^-(\bZ_1),\bef_3(\bZ_1),\ldots,\bef_N(\bZ_1)).
\end{align*}
%\end{subequations}
From now on, we will omit the dependence on $\bZ_1$ and simply write
$\bef^{(+,+)}$, etc.
The important feature of the form of these four fields is that
\begin{subequations} \label{eq:2}
\begin{align}
&  (\hbf_1^+-\hbf_1^-,\bzero,\ldots,\bzero)= \bef^{(+,+)}-\bef^{(-,+)} = \bef^{(+,-)} - \bef^{(-,-)}, \label{eq:2a}\\
 & (\bzero,\hbf_2^+-\hbf_2^-,\ldots,\bzero)= \bef^{(+,+)}-\bef^{(+,-)} = \bef^{(-,+)} - \bef^{(-,-)}. \label{eq:2b}
\end{align}
\end{subequations}
Then \eqref{cond1}-\eqref{cond4} become
\begin{subequations}\label{eq:3}
\begin{align}
&  \bn_1\cdot \bef^{(+,+)} <0,\quad  \bn_1\cdot \bef^{(+,-)} <0,\quad
 \bn_1\cdot\bef^{(-,+)} >0,\quad  \bn_1\cdot\bef^{(-,-)} >0\label{eq:3a}\\
& \bn_2\cdot \bef^{(+,+)} <0,\quad  \bn_2\cdot\bef^{(+,-)} >0,\quad
 \bn_2\cdot\bef^{(-,+)} <0, \quad  \bn_2\cdot \bef^{(-,-)} >0.\label{eq:3b}
\end{align}
\end{subequations}
%(Everything above  is evaluated at $\bZ_1$.)

Let $\bZ \in{\rm co}\, F (\bZ_1)$ satisfy \eqref{ortho}. From  
Lemma \ref{lem:hulls}, we have that there
exist $s,t\in [0,1]$ such that
\begin{equation}
  \label{eq:4}
  \begin{split}  
  \bZ &= (s\hbf_1^+ + (1-s)\hbf_1^-, \, t \hbf_2^+ + (1-t)\hbf_2^-,\bef_3,\ldots,\bef_N)\\
& = (s (\hbf_1^+-\hbf_1^-),\, t(\hbf_2^+ -\hbf_2^-),\bzero) + (\hbf_1^-,\hbf_2^-,\bef_3,\ldots,\bef_N)\\
& = s(\bef^{(+,+)} - \bef^{(-,+)}) + t(\bef^{(+,+)} - \bef^{(+,-)}) + \bef^{(-,-)},
\end{split}
\end{equation}
where we used \eqref{eq:2}. 
By \eqref{eq:4}, the conditions in \eqref{ortho} become
\begin{align*}
%  \label{eq:6}
&    \bn_1\cdot (\bef^{(+,+)} - \bef^{(-,+)}) s + \bn_1\cdot (\bef^{(+,+)} - \bef^{(+,-)}) t=-\bn_1\cdot
    \bef^{(-,-)}, \\
 &   \bn_2\cdot (\bef^{(+,+)} - \bef^{(-,+)}) s + \bn_2\cdot (\bef^{(+,+)} - \bef^{(+,-)}) t=-\bn_2\cdot
    \bef^{(-,-)},
\end{align*}
or, equivalently,
\begin{equation}\label{Astb}
A\vec{s}{t}=\bb,
\end{equation}
%This, in matrix form, is $A\left( s\atop t\right)=b$,
% \begin{align}
%   \label{eq:7}
%  A\left(   \begin{array}{c}
%      s \\ t
%    \end{array}\right)=b, 
% \end{align}
where
\begin{align}
  \label{eq:8}
A := \left(
  \begin{array}{cc}
    a_{11} & a_{12} \\ a_{21} & a_{22}
  \end{array}
\right)=
\left(
  \begin{array}{cc}
    \bn_1\cdot (\bef^{(+,+)} - \bef^{(-,+)}) & \bn_1\cdot (\bef^{(+,+)} - \bef^{(+,-)})\\
\bn_2\cdot (\bef^{(+,+)} - \bef^{(-,+)}) & \bn_2\cdot (\bef^{(+,+)} - \bef^{(+,-)})
  \end{array}\right)
\end{align}
and
\begin{align*}
%  \label{eq:100}
\bb := \left(
  \begin{array}{c}
    -\bn_1\cdot    \bef^{(-,-)} \\ -\bn_2\cdot    \bef^{(-,-)}
  \end{array}
\right).
\end{align*}
To prove the lemma it is enough to show that there is a unique choice of $s$ and $t$ that satisfy \eqref{Astb}. 
We prove this by using \eqref{eq:3} to show that ${\rm det}A >0$.

From \eqref{eq:3a} we have
\begin{align}
  \label{eq:9}
 a_{11}= \bn_1 \cdot (\bef^{(+,+)} - \bef^{(-,+)}) <0, \quad{\rm and}\quad
\bn_1\cdot (\bef^{(+,+)} - \bef^{(-,-)}) <0
\end{align}
Thus, by \eqref{eq:2b} and \eqref{eq:8}
\begin{align}
  \label{eq:10}
  0 > \bn_1 \cdot (\bef^{(+,+)} - \bef^{(-,-)}) &= 
\bn_1 \cdot (\bef^{(+,+)} - \bef^{(-,+)}) +\bn_1 \cdot (\bef^{(-,+)} - \bef^{(-,-)})\\
 &= a_{11} + a_{12}.\nonumber
\end{align}
Again using \eqref{eq:3a}, we have
\begin{align*}
%  \label{eq:11}
  \bn_1 \cdot (\bef^{(+,-)} - \bef^{(-,+)})<0.
\end{align*}
Thus
\begin{align}
  \label{eq:12}
  0> \bn_1 \cdot (\bef^{(+,-)} - \bef^{(-,+)})=&
\bn_1 \cdot (\bef^{(+,-)} - \bef^{(+,+)})+\bn_1 \cdot (\bef^{(+,+)} - \bef^{(-,+)}) 
\\&= -a_{12}+a_{11}.\nonumber
\end{align}
Combining equations \eqref{eq:10} and \eqref{eq:12} we have
\begin{align}
  \label{eq:13}
  a_{11} < a_{12}< -a_{11} \implies |a_{12}| < -a_{11} = |a_{11}|.
\end{align}
Similarly,
\begin{align}
  \label{eq:14}
  a_{22} =& \bn_2\cdot (\bef^{(+,+)}-\bef^{(+,-)})<0,\quad
\bn_2\cdot(\bef^{(+,+)}-\bef^{(-,-)})<0,\\
&\bn_2\cdot(\bef^{(-,+)}-\bef^{(+,-)})<0.\nonumber
\end{align}
Noting that, from \eqref{eq:2} and \eqref{eq:8}, 
we have $a_{21} = \bn_2\cdot(\bef^{(+,+)}-\bef^{(-,+)}) = \bn_2\cdot(\bef^{(+,-)}-\bef^{(-,-)})$
so
\begin{align}
  \label{eq:15}
  0>\bn_2\cdot(\bef^{(+,+)}-\bef^{(-,-)}) &= 
\bn_2\cdot(\bef^{(+,+)}-\bef^{(+,-)}) + \bn_2\cdot(\bef^{(+,-)}-\bef^{(-,-)})
\\& = a_{22}+a_{21},\nonumber
\end{align}
and
\begin{align}
  \label{eq:16}
  0>\bn_2\cdot(\bef^{(-,+)}-\bef^{(+,-)}) &= 
\bn_2\cdot(\bef^{(-,+)}-\bef^{(+,+)})+\bn_2\cdot(\bef^{(+,+)}-\bef^{(+,-)})
\\
& = -a_{21}+a_{22}.\nonumber
\end{align}
Combining equations \eqref{eq:15} and \eqref{eq:16} we have
\begin{align}
  \label{eq:17}
  a_{22} < a_{21}< -a_{22} \implies |a_{21}| < -a_{22} = |a_{22}|.
\end{align}
From \eqref{eq:13} and \eqref{eq:17} we have
\begin{align*}
%  \label{eq:18}
  0&< |a_{11}||a_{22}| -|a_{12}||a_{21}| 
\leq |a_{11}||a_{22}|- a_{12} a_{21} =
a_{11} a_{22}- a_{12} a_{21} ={\rm det}A,
\end{align*}
where we also used that $a_{11},\, a_{22} <0$ from 
\eqref{eq:9} and \eqref{eq:14}.
\end{proof}

\section*{Acknowledgments}
The authors warmly thank the Center for Nonlinear Analysis (NSF Grant No.\@ DMS-0635983), where part of this research was carried out. The
research of I.~Fonseca was partially funded by the National Science Foundation
under Grant No. DMS-0905778 and that of G.~Leoni under Grant No.~DMS-1007989.
T.~Blass, I.~Fonseca, and G.~Leoni also acknowledge support of the National Science
Foundation under the PIRE Grant No.~OISE-0967140. The work of M.~Morandotti was partially supported by grant FCT$\_$UTA/CMU/MAT/0005/2009.

\bibliographystyle{plain}
\bibliography{disloc}

\end{document}